\documentclass[final,1p,times]{elsarticle}

\usepackage{amssymb}
\usepackage{amsmath}
\usepackage{amsthm}
\usepackage{subcaption}
\usepackage{marginnote}
\usepackage[usenames,dvipsnames]{color}

\usepackage{float}

\def \blue{\color{black}}
\def \red{\color{black}}
\newcommand{\newblue}[1]{ {\color{black} #1} }
\newcommand{\newred}[1]{ {\color{black} #1} }

\biboptions{sort}

\graphicspath{ {./images/} }

\newtheorem{Definition}{Definition}[section]
\newtheorem{Assumption}{Assumption}[section]
\newtheorem{Theorem}{Theorem}[section]
\newtheorem{Lemma}{Lemma}[section]

\journal{ArXiv}

\begin{document}

\begin{frontmatter}

\title{Split$-$step Milstein methods for multi-channel stiff stochastic differential systems}

\author[label5]{V. Reshniak\corref{cor1}}
\ead{vr2m@mtmail.mtsu.edu}

\author[label5]{A.Q.M. Khaliq}
\address[label5]{Department of Mathematical Sciences and Center for Computational Science, Middle Tennessee State University, Murfreesboro, TN 37132, USA}
\ead{Abdul.Khaliq@mtsu.edu}

\author[label1]{D.A. Voss}
\address[label1]{Professor Emeritus, Department of Mathematics, Western Illinois University, 1 University Circle, Macomb, IL 61455, USA}
\ead{d-voss1@wiu.edu}

\author[label6]{G. Zhang}
\address[label6]{Computer Science and Mathematics Division, Oak Ridge National Laboratory, Oak Ridge, TN 37831, USA}
\ead{zhangg@ornl.gov}

\cortext[cor1]{Corresponding author}

\begin{abstract}
We consider split-step Milstein methods for the solution of stiff stochastic differential equations with an emphasis on systems driven by multi-channel noise. We show their strong order of convergence and investigate mean-square stability properties for different noise and drift structures. The stability matrices are established in a form convenient for analyzing their impact arising from different deterministic drift integrators. Numerical examples are provided to illustrate the effectiveness and reliability of these methods.
\end{abstract}


\end{frontmatter}


\section{Introduction}
\label{sec1}

%





It is hard to overestimate the role which stochastic differential equations play in mathematical modelling of phenomena with uncertainties and random effects which cannot be properly described in terms of classical deterministic models. 
{\color{black} For instance, it has been known for a long time that many real-life processes are intrinsically stochastic with multiple sources of randomness. Numerous examples include various applications in electrical circuit engineering \cite{Penski2000}, neuroscience  \cite{laing2009}, gene regulatory networks \cite{Samad2005}, chemistry \cite{Gillespie2007}, biology \cite{Wilkinson2012} and other fields of science and engineering. Often the dynamical behavior of such systems can be effectively modeled by the system of It\^o stochastic differential equations \cite{allen2007}}

\begin{equation}\label{eq:1}
    dX(t) = f(t,X) dt + \sum_{j=1}^m g_j (t,X) dW_j, \quad X(t_0)=X_0, \quad t \in [t_0,T],
\end{equation}
where $X \in \mathbb{R}^d$, $f:[t_0,T] \times \mathbb{R}^d \to \mathbb{R}^d$ is a drift vector, $g=(g_1,g_2,..,g_m):[t_0,T] \times \mathbb{R}^d \to \mathbb{R}^{d \times m}$  is a diffusion matrix and $W(t)$ is an $m$-dimensional Wiener process defined on the complete probability space $(\Omega, \mathcal{F},\mathbb{P})$ with a filtration $\{\mathcal{F}_t\}_{t \geq 0}$ satisfying the usual conditions  \cite{Kloeden1992}. 

This paper is concerned with strong solutions of the system in \eqref{eq:1}. Unfortunately, only a few of such equations can be solved analytically and numerical methods are necessary to obtain approximate solutions. In this paper we use the denotation $X_n$ for the value of numerical approximation of the solution $X_{t_n}$ defined at the nodes of equidistant time discretization $t_n = t_0 + nh, \; n = 0,1,...,N$ with a time step $h = (T-t_0)/N$. As we consider strong numerical solutions we recall the following
\begin{Definition}
(\cite{Kloeden1992})
A discrete approximation $X_n$ converges strongly with order $\gamma \in (0, \infty)$ if there exist a finite constant $K$ and a positive constant $\delta_0$ (independent of h) such that for each $h \in (0,\delta_0)$
\begin{align}\label{eq:2}
    \mathsf{E} \Big[ | X_{t_n}-X_n | \Big] \leq K h^{\gamma}.
\end{align}
\end{Definition}

{\color{black} Stability of the scheme represents another important property of a numerical method illustrating its ability to preserve the qualitative behavior of the solution of the original system. There are two main contributors to the stability properties of SDEs: stiffness and geometry of stochastic perturbations. It should be mentioned that, despite their equal importance, most of the research effort in the literature was concentrated on the study of the stiffness property. }
This seems to be logical but unfair. 
Of course, stiffness has been studied for a long time in the theory of deterministic ODEs and a lot of valuable results have been obtained (see, for example, \cite{Hairer1996}). Besides, stochastic stiffness is a generalization of its deterministic counterpart in the sense that a deterministically stiff problem is also stochastically stiff. 
{\color{black} However, that is where the similarities between deterministic and stochastic systems end. For example, in the above mentioned  biological and chemical networks, main issues arise from the multiscale nature of the underlying problem: presence of multiple timescales and the need to include in the simulation both species that are present in relatively small quantities and should be modeled by a discrete stochastic process, and species that are present in larger quantities and are more efficiently modeled by a deterministic differential equation \cite{Samad2005}. Moreover, chemical components usually interact in a highly nonlinear way. Therefore, small random fluctuations can significantly influence the behavior of the entire system and change its stiffness with uncertainty.}

Stability of stochastic systems also depends on the geometry of random noise and the type of interaction between the drift and diffusion components of SDEs.
\newred{ 
It is well known that noise can be used to stabilize or destabilize deterministic systems. The literature on stabilization~/~destabilization theory is extensive, analysis of the asymptotic stability in almost sure sense can be found in \cite{Mao1994, Appleby2008, Huang2013, Buckwar2010} and mean-square stability of systems with stochastic perturbation was discussed in \cite{Buckwar2010, Higham2006}.
For instance, it was shown in \cite{Higham2006} that under certain conditions imposed on the drift and diffusion geometries even a non-stiff stable deterministic system can be destabilized in the mean-square sense by an infinitesimally small white noise perturbation. 
In the following we will use this example in our analysis of the mean-square asymptotic stability of numerical methods which will discussed later.}

{\color{black}  Additionally, the presence of several independent noise channels by itself might become the source of additional difficulties, both analytical and computational, which are not present in SDEs with a scalar noise.  Of course, it is not difficult to extend algorithms for the case of multiple noise channels but it is known that the stability condition of numerical methods might become very restrictive for an increasing number of noise terms \cite{Buckwar2011}.}

Hence, in view of the above mentioned it should be evident that a novel approach is required for the proper analysis of multi-dimensional stochastic systems.
A classical way to resolve the problem of stiffness is by incorporating implicitness into numerical schemes. However, in the case of SDEs, the straightforward implementation of an implicit approach can lead to unbounded solutions \cite{Kloeden1992}. Various implicit schemes appeared in the literature to overcome this issue. The family of semi-implicit (drift-implicit) methods has been known for a long time \cite{Kloeden1992, Saito1996, Buckwar2006} and is well adapted to the problems with stiff deterministic part. 
{\color{black} For equations where both drift and diffusion parts are stiff, fully implicit schemes can be used  at a cost of higher computational complexity} \cite{Ahmad2009, Omar2011, Wang2012}.
Also, an elegant explicit approach based on a stochastic modification of Chebyshev methods which possess very good stability properties was recently proposed  in \cite{Abdulle2007, Abdulle2008}. 

{\color{black} Split-step methods represent a class of fully implicit methods which allow the incorporation of  implicitness in the stochastic part of the system with relatively little additional cost.
This feature makes them very attractive for solving equations with multi-dimensional noises.}
The first method of this type was introduced by Higham in \cite{Higham2002} as a modification of the classical Euler-Maruyama method. Then Wang {\it{et al.}} in \cite{WANG2009,Wang2010} presented several variants of split-step backward and forward Milstein methods for systems driven by a single noise term.  {\color{black} Also, some valuable results for multi-dimensional SDEs with multiplicative multi-channel noise were obtained recently in \cite{Haghighi2012} where convergence properties of split-step methods with  one-step ODE solvers possessing at least first order were investigated.
However their stability analysis was restricted to a single noise term only and numerical examples included systems with at most two noise channels. 
In this work we explore the trend.  

The paper is organized as follows. In section 2, we provide an overview of existing split-step methods and also introduce the new methods. In section 3, we briefly discuss convergence properties of split-step Milstein methods. In section 4, we give the mean-square stability analysis of the proposed methods and apply results to different test systems. 
Numerical examples in section 5 are used to confirm theoretical results. Example 1 is a benchmark problem for testing convergence rates of the numerical methods. Example 2 illustrates the influence of stiffness on the stability of the system. Finally, example 3 is an application problem.
All the results were computed in FORTRAN 90 on Intel(R) Core(TM) i7-4930K, 3.4 GHz processor along with 16 GB memory.

%
%
%

\section{Split-step methods}
\label{sec2}



In this section we extend existing split-step methods to systems with multi-dimensional noise and introduce the new split-step Adams-Moulton-Milstein method together with its modified analogue.
We start with the first split-step method proposed in the literature - the split-step Euler $\theta$-method 

\begin{align}\label{eq:3}
    &\hat{X}_{n} = X_n + h \Big[ (1-\theta)f(t_n,X_n) + \theta f(t_n,\hat{X}_n) \Big]
    \\ \nonumber
    &X_{n+1} = \hat{X}_n + g(t_n,\hat{X}_n) \Delta W_n.
\end{align}
Ding et al. \cite{Ding2010} investigated the convergence and stability properties of this method under the global Lipschitz condition. The backward Euler variant of this scheme with $\theta=1$ was first introduced and analyzed by Higham et al. in \cite{Higham2002} under one-sided Lipschitz condition.

Though it is known that Euler-type methods always possess bigger stability regions than their Milstein-type counterparts \cite{Buckwar2011}, we are interested in the latter because of their higher rate of strong convergence. With respect to the used splitting technique, split-step Milstein methods can be classified into two families. According to this classification we distinguish between the class of \textit{split-step composite Milstein methods}

\begin{align}\label{eq:4}
    \hat{X}_{n} &= X_n + h \Phi(t_n,t_{n+1},X_n,\hat{X}_{n})
    \\ \nonumber
    X_{n+1} &= \hat{X}_n + \eta \left[ \sum_{j=1}^m g_j(t_n,\hat{X}_n) I_{(j)} + \sum_{j_1=1}^m \sum_{j_2=1}^m L^{j_1} g_{j_2}(t_n,\hat{X}_n) I_{(j_1,j_2)} \right] 
    \\ \nonumber
    &+ (1-\eta) \left[ \sum_{j=1}^m  g_j(t_n,{X}_n) I_{(j)} + \sum_{j_1=1}^m \sum_{j_2=1}^m L^{j_1} g_{j_2}(t_n,X_n) I_{(j_1,j_2)} \right]
\end{align}
and the class of \textit{modified split-step composite Milstein methods}

\begin{align}\label{eq:5}
    \hat{X}_{n} &= X_n + h \left[ \Phi(t_n,t_{n+1},X_n,\hat{X}_{n}) - \frac{1}{2} \left( \eta \sum_{j=1}^m L^{j} g_{j}(t_n,\hat{X}_n) + (1-\eta) \sum_{j=1}^m L^{j} g_{j}(t_n,X_n) \right) \right]
    \\ \nonumber
    X_{n+1} &= \hat{X}_n + \eta \left[ \sum_{j=1}^m g_j(t_n,\hat{X}_n) I_{(j)} + \sum_{j_1=1}^m \sum_{j_2=1}^m L^{j_1} g_{j_2}(t_n,\hat{X}_n) J_{(j_1,j_2)} \right] 
    \\ \nonumber
    &+ (1-\eta) \left[ \sum_{j=1}^m  g_j(t_n,{X}_n) I_{(j)} + \sum_{j_1=1}^m \sum_{j_2=1}^m L^{j_1} g_{j_2}(t_n,X_n) J_{(j_1,j_2)} \right]
\end{align}
where $\Phi(t_n,t_{n+1},X_n,\hat{X}_{n})$ is an increment function of the deterministic ODE solver and $\eta \in [0,1]$ is a parameter.
The latter scheme is obtained by collecting all deterministic terms in the first splitting step of the method using the fact that $I_{(j,j)} = \frac{1}{2} \left( \Big[ \Delta W_n^j \Big]^2 - h \right)$.

\newblue{
\textit{Remark.} In the equations \eqref{eq:4} and \eqref{eq:5} $I_{(j)}$, $I_{(j_1,j_2)}$ and $J_{(j_1,j_2)}$ denote stochastic It\^o and Stratonovich integrals respectively

\begin{align}\label{eq:10}
    I_{(j)} &= \int_{t_n}^{t_n+\Delta t} dW_j(s) \cong \Delta W_n^j = \sqrt{h} \xi^j,
    \\ \nonumber
    I_{(j_1,j_2)} &= \int_{t_n}^{t_n+\Delta t} \left( \int_{t_n}^{s_1} dW_{j_1}(s_2) \right) dW_{j_2}(s_1),
    \\ \nonumber
    I_{(j_1,j_2)} &= J_{(j_1,j_2)}, \quad j_1 \neq j_2,
    \\ \nonumber
    I_{(j,j)} &= J_{(j,j)} - \frac{1}{2} h
\end{align}
where $\xi^j =  N(0,1)$ is a normally distributed random variable. $L^{j}$ is a differential operator defined as follows

\begin{align}\label{eq:11}
    L^{j} = \sum_{k=1}^d g_{j}^k \frac{\partial}{\partial y^k}.
\end{align}

Computational aspects of the efficient evaluation of triple product sums of the form

$$
\sum_{j_1=1}^m \sum_{j_2=1}^m L^{j_1} g_{j_2}(t_n,\hat{X}_n) I_{(j_1,j_2)} = 
\sum_{j_1=1}^m \sum_{j_2=1}^m \sum_{k=1}^d g_{j_1}^k \frac{\partial g_{j_2}(t_n,\hat{X}_n)}{\partial y^k} I_{(j_1,j_2)}
$$

will be discussed in Appendix A.
}

It is obvious that the choice of deterministic solver determines the properties of the above numerical schemes. Here we provide several variants of the first (deterministic) stage which appeared in the literature earlier.
Split-step composite $\theta$-Milstein method (\textbf{SSCTM}) and modified split-step composite $\theta$-Milstein method (\textbf{MSSCTM}) were discussed by Guo et al. \cite{Guo2013} and correspond to the classical $\theta$-method used as deterministic solver

\begin{equation}\label{eq:6}
	\Phi(t_n,t_{n+1},X_n,\hat{X}_{n}) = \theta f(t_n,\hat{X}_n) + (1-\theta)f(t_n,X_n).
\end{equation}

{\color{black} With $\eta = 1$, the drifting split-step backward Milstein method (\textbf{DSSBM}) and the modified backward Milstein method (\textbf{MSSBM}) are implicit variants of methods \eqref{eq:4} and \eqref{eq:5}, respectively, using
$\theta = 1$ in the deterministic solver \eqref{eq:6}.

Deterministic solvers with more than one implicit stage have been considered by various authors. In particular, a two-stage  Rosenbrock method (\textbf{RSB}) considered in \cite{Haghighi2012} has form

\begin{align}\label{eq:8a}
	(I - \gamma h J_f) K_1 &= f(t_n,X_n)
	\\
	(I - \gamma h J_f) K_2 &= f(t_n,X_n + \alpha_{21} K_1) + \gamma_{21} J_f K_1,
\end{align}
with increment function

\begin{align}\label{eq:8b}
	\Phi(t_n,t_{n+1},X_n) = b_1 K_1 + b_2 K_2,
\end{align}
where $J_f$ is a Jacobian matrix and the determining parameters are given in \cite{Haghighi2012}. With $\eta = 1$, method \textbf{RSB} results when $\hat{X}_n$ in \eqref{eq:5} is replaced by

\begin{align}\label{eq:8c}
    \hat{X}_{n} = X_n + h \left[ \Phi(t_n,t_{n+1},X_n) - \frac{1}{2} \sum_{j=1}^m L^{j} g_{j}(t_n,\hat{X}_n) \right].
\end{align}

Finally, we consider a new split-step Adams-Moulton-Milstein (\textbf{SSAMM}) method and its modified analogue (\textbf{MSSAMM}) based on their development in \cite{Khaliq2014} for single channel noise. Their deterministic component is also a two-stage method, but based on the predictor-corrector approach wherein both the predictor and corrector are implicit but share the same Newton iteration matrix, $I-(\frac{1}{2}-\theta)hJ_f$. The method \textbf{SSAMM} replaces $\hat{X}_n$ in \eqref{eq:4} by

\begin{align}\label{eq:8d}
    &\tilde{X}_{n} = X_n + h \Phi_1(t_n,t_{n+1},X_n,\tilde{X}_{n})
    \\ \nonumber
    &\hat{X}_{n} = X_n + h \Phi_2(t_n,t_{n+1},X_n,\tilde{X}_{n},\hat{X}_{n})
\end{align}
where

\begin{align}\label{eq:8e}
	\Phi_1(t_n,t_{n+1},X_n,\tilde{X}_{n}) &=  ( \frac{1}{2} + \theta ) f(t_n,X_n) + ( \frac{1}{2} - \theta ) f(t_n,\tilde{X}_n) 
	\\
	\Phi_2(t_n,t_{n+1},X_n,\tilde{X}_{n},\hat{X}_{n}) &=  \frac{1}{2} f(t_n,X_n) + ( \frac{1}{2}-\theta ) f(t_n,\hat{X}_n) + \theta f(t_n,\tilde{X}_n)
\end{align}
with $\theta = -\frac{1}{2} \pm \frac{1}{\sqrt{2}}$. For details on the construction of deterministic solver and choice of values for parameters $\theta$ we refer to \cite{Voss1989}. In the following \textbf{SSAMM+} will be used for \textbf{SSAMM} method with $\theta = -\frac{1}{2} + \frac{1}{\sqrt{2}}$ while \textbf{SSAMM-} will be used for \textbf{SSAMM} method with $\theta = -\frac{1}{2} - \frac{1}{\sqrt{2}}$. For method \textbf{MSSAMM}, $\hat{X}_{n}$ in \eqref{eq:5} is replaced by

\begin{align}\label{eq:8f}
    &\tilde{X}_{n} = X_n + h \left[ \Phi_1(t_n,t_{n+1},X_n,\tilde{X}_{n}) - \frac{1}{2} \left( \eta \sum_{j=1}^m L^{j} g_{j}(t_n,\tilde{X}_n) + (1-\eta) \sum_{j=1}^m L^{j} g_{j}(t_n,X_n) \right) \right]
    \\ \nonumber
    &\hat{X}_{n} = X_n + h \left[ \Phi_2(t_n,t_{n+1},X_n,\tilde{X}_{n},\hat{X}_{n}) - \frac{1}{2} \left( \eta \sum_{j=1}^m L^{j} g_{j}(t_n,\hat{X}_n) + (1-\eta) \sum_{j=1}^m L^{j} g_{j}(t_n,X_n) \right) \right]
\end{align}
where $\Phi_1$ and $\Phi_2$ are given above.
}

\section{Convergence properties}
\label{sec:3}
In this section we provide convergence analysis of methods \eqref{eq:4} and \eqref{eq:5} {with different deterministic integrators}. Using the proof similar to that in \cite{Haghighi2012}  and \cite{WANG2009} we find that these methods have strong order of convergence 1.0. 
We accomplish our proof under the usual conditions on $f$ and $g$.

\begin{Assumption}
The functions $f$, $g_j$ and $L^{j_1}g_{j_2}$ for all $j, j_1, j_2 = 1,...,m$  satisfy the global Lipschitz condition for constant $K>0$ and linear growth bound, as follows:
    
\begin{align}\label{eq:12}
    &\Big| f(t,a)-f(t,b) \Big| + \Big| g(t,a)-g(t,b) \Big| 
    + \Big| L^{j_1}g_{j_2}(t,a) - L^{j_1}g_{j_2}(t,b) \Big| \leq K_1 \Big| a-b \Big|,
    \\ \nonumber
    & 
    \forall a,b \in \mathbb{R}^d, \; t \in [t_0,T],
    \\[1em] \label{eq:13}
    &\Big| f(t,a) \Big| + \Big| g(t,a) \Big| + \Big| L^{j_1}g_{j_2}(t,a) \Big| \leq K_2 \Big( 1+ |a| \Big)^{\frac{1}{2}}, 
    \\ \nonumber
    & \forall a \in \mathbb{R}^d, \; t \in [t_0,T].
\end{align}
\end{Assumption}

\newblue{It is worth noting that deterministic steps in the schemes \eqref{eq:4} and \eqref{eq:5} are given by implicit equations and therefore the question of existence and uniqueness of the solution $\hat{X}_n$ arises. However since the first step in the scheme \eqref{eq:4} merely approximates the drift flow of  $X(t),~t~\in~[t_n,~ t_{n+1}]$, existence of this solution follows from the properties of the corresponding ODE solver. For instance, using the fixed point theorem and Assumption 3.1 one may show that for the classical $\theta$-method \eqref{eq:6} a unique solution exists, with probability one, for $K_1 \theta h < 1$ \cite[Theorem 7.2]{Hairer2008}. Moreover, for the same $\theta$-method global Lipschitz continuity of $L^{j_1}g_{j_2}$ implies existence of the solution for the deterministic part of \eqref{eq:5} if $K_1' \theta h < 1$ with $K_1' > K_1$.
Similar results also hold for more general Runge-Kutta and multistep methods; see, for example, \cite{Hairer2008, Hairer1996, Stuart1998}.
}

Now to measure the order of convergence we use the following theorem given by Milstein:

\begin{Theorem}
(\cite{Milstein1995})
Assume for a one-step discrete time approximation $y$ that the local mean error and mean square error for all $N=1,2,...$ and $n=0,2,...,N-1$ satisfy the estimates
\begin{align}\label{eq:14}
    \Big| \mathsf{E} \Big[ X_{n+1} - X_{t_{n+1}} | \mathcal{F}_n \Big] \Big| 
    \leq K \Big( 1 + |X_n|^2 \Big)^{1/2} h^{p_1}
\end{align}
and
\begin{align}\label{eq:15}
    \Big| \mathsf{E} \Big[ \Big( X_{n+1} - X_{t_{n+1}} \Big)^2 | \mathcal{F}_n \Big] \Big|^{1/2} 
    \leq K \Big( 1 + |X_n|^2 \Big)^{1/2} h^{p_2}
\end{align}
with $p_2 \geq \frac{1}{2}$ and $p_1 \geq p_2 + \frac{1}{2}$. Then
\begin{align}\label{eq:16}
    \Big| \mathsf{E} \Big[ \Big( X_{n} - X(t_{n}) \Big)^2 | \mathcal{F}_0 \Big] \Big|^{1/2} 
    \leq K \Big(1 + |X_0|^2 \Big)^{1/2} h^{p_2-1/2}
\end{align}
holds for each $k=0,1,2,...,N$. Here $K$ is independent of $h$ but dependent on the length of the time interval $T-t_0$.
\end{Theorem}

{\color{black} To prove the convergence of the methods SSCTM, MSSCTM, SSAMM and MSSAMM we need the following Lemma.

\begin{Lemma}
For the methods \eqref{eq:4} and \eqref{eq:5} \newblue{under assumption 3.1} we have the following estimate
	
\begin{align}\label{eq:15a}
	\Big| \hat{X}_n - X_n \Big| \leq K_{method} h \Big( 1+|X_n|^2 \Big)^{1/2},
\end{align}
where for the methods SSCTM, MSSCTM, SSAMM and MSSAMM we obtain

\begin{eqnarray}
	\nonumber
	&K_{SSCTM}  =  \displaystyle{ \frac{ {K_2}}{1-h \theta {K_1}} },
	\qquad
	&K_{SSAMM}  = \frac{ 2 {K_2}}{2- h {K_1}} ,
	\\[1em] \label{eq:15b}
	&K_{MSSCTM}  = \displaystyle{ \frac{(2+m) {K_2}}{2-h (2 \theta+\eta m ) {K_1}} },
    \qquad
	&K_{MSSAMM}  =  \frac{ (2+m) {K_2}}{2- (h+\eta m ) {K_1}}.
\end{eqnarray}
\end{Lemma}

\begin{proof}
Based on the Assumption 3.1 we obtain

for SSCTM method:
$$
	\Big| \hat{X}_n - X_n \Big| \leq h \Big( \theta | f(\hat{X}_n) - f(X_n) | +| f(X_n) | \Big) \leq \frac{h {K_2}}{1-h \theta {K_1}} \Big( 1+|X_n|^2 \Big)^{1/2},
$$

for MSSCTM method:
\begin{align*}
	\Big| \hat{X}_n - X_n \Big| 
	&\leq h \left( \theta | f(\hat{X}_n) - f(X_n) | +| f_n | + \frac{\eta}{2} \left| \sum_{j=1}^m \left( L^{j}g_{j}(\hat{X}_n) - L^{j}g_{j}(X_n) \right) \right| \right) 
	\\
	&\leq \frac{h \left( 1 + \frac{m}{2} \right) {K_2}}{1-h \theta \left( \theta + \frac{m \eta}{2} \right) {K_1}} \Big( 1+|X_n|^2 \Big)^{1/2},
\end{align*}

for SSAMM method:
\begin{align*}
	\Big| \tilde{X}_n - X_n \Big| 
	\leq h \left( \big| \frac{1}{2}-\theta \big| \big| f(\tilde{X}_n) - f(X_n) \big| + | f_n | \right) 
	\leq \frac{h {K_2}}{1-h \big| \frac{1}{2}-\theta \big| {K_1}} \Big( 1+|X_n|^2 \Big)^{1/2}.
\end{align*}

Then
\begin{align*}
	\Big| \hat{X}_n - \tilde{X}_n \Big| 
	&\leq h \left( \big| \frac{1}{2}-\theta \big| \big| f(\hat{X}_n) - f(\tilde{X}_n) \big| + |\theta| \big| f(X_n) - f(\tilde{X}_n) \big| \right) 
	\\
	&\leq \frac{h |\theta| {K_1}}{1-h \big| \frac{1}{2}-\theta \big| {K_1}} \big|X_n - \tilde{X}_n \big|
	\leq \frac{h^2 |\theta| {K_1 K_2}}{ \big(1-h \big| \frac{1}{2}-\theta \big| {K_1} \big)^2} \Big( 1+|X_n|^2 \Big)^{1/2},
\end{align*}
and finally
\begin{align*}
	\Big| \hat{X}_n - X_n \Big| 
	&\leq h \left( | \theta | \big| f(\hat{X}_n) - f(\tilde{X}_n) \big| + \frac{1}{2} \big| f(X_n) - f(\hat{X}_n) \big| + \big| f(X_n) \big|\right) 
	\\
	&\leq \frac{h |\theta| {K_1}}{1- \frac{h}{2} {K_1}} \big| \hat{X}_n - \tilde{X}_n \big|
	+ \frac{h {K_2}}{1- \frac{h}{2} {K_1}} \Big( 1+|X_n|^2 \Big)^{1/2}
	\\
	&\leq \frac{h {K_2}}{1- \frac{h}{2} {K_1}} \Big( 1+|X_n|^2 \Big)^{1/2} + O(h^3),
\end{align*}

for MSSAMM method:
\begin{align*}
	\Big| \hat{X}_n - X_n \Big| 
	\leq \frac{h \left( 1 + \frac{m}{2} \right) {K_2}}{1- \frac{1}{2} (h+ \eta m) {K_1}} \Big( 1+|X_n|^2 \Big)^{1/2} + O(h^3).
\end{align*}	
	
\end{proof}
}

Then the following theorem establishes convergence of methods \eqref{eq:4} and \eqref{eq:5}.

\begin{Theorem}
Let $X_n$ be the numerical approximation to $X(t_n)$ at the time $T$ after $k$ steps with step size $h=T/N$, $N=1,2,...$. 
Then for methods \eqref{eq:4} and \eqref{eq:5} with an increment function $\Phi(t_n,t_{n+1},X_n,\hat{X}_{n})$ of the selected ODE solver of order at least one under Assumption 3.1 we have
\begin{align}\label{eq:17}
    \Big| \mathsf{E} \Big[ \Big( y_{k}-X(t_{k}) \Big)^2 | \mathcal{F}_0 \Big] \Big|^{1/2} = O(h).
\end{align}
\end{Theorem}

\begin{proof}
We try to estimate local mean and mean-square errors \eqref{eq:14} and \eqref{eq:15} first. Denote the local Milstein approximation step as
\begin{align}\label{eq:18}
    X_{n+1}^M = X_n + h f(t_n,X_n) + \sum_{j=1}^m g_j(t_n,X_n) I_{(j)} + \sum_{j_1=1}^m \sum_{j_2=1}^m L^{j_1} g_{j_2}(t_n,X_n) I_{(j_1,j_2)}.
\end{align}

Then we arrive at
\begin{align}\label{eq:19}
    \Big| \mathsf{E} \Big[ X_{n+1} - X_{t_{n+1}} | \mathcal{F}_n \Big] \Big| 
    &\leq 
     \Big| \mathsf{E} \Big[ X_{n+1} - X_{n+1}^M | \mathcal{F}_n \Big] \Big| 
  +  \Big| \mathsf{E} \Big[ X_{n+1}^M - X_{t_{n+1}} |\mathcal{F}_n \Big] \Big| 
    \\ \nonumber
    &\leq K \Big( 1+|X_n|^2 \Big)^{1/2}h^2 + H_1.
\end{align}

{\color{black}
Now from Lemma 3.1 we have that
$$
H_1 = \Big| E \Big[ X_{n+1}^M-X_{n+1} | \mathcal{F}_n \Big] \Big| 
\leq  h^2 K_{method} {K_1} \Big( 1+|X_n|^2 \Big)^{1/2}.
$$

For example, for method SSAMM using the property of It\^o integrals $\mathsf{E}(I_{(j)})=\mathsf{E}(I_{(j_1,j_2)})=0$ we obtain


\begin{align*}
    H_1 &\leq h \left( \frac{1}{2} \Big| f(t_n,X_n) - f(t_n,\hat{X}_n) \Big| + \theta \Big| f(t_n,\hat{X}_n) - f(t_n,\tilde{X}_{n}) \Big| \right)
    \\
    &\leq
     \frac{h^2}{2} K_{SSAMM} {K_1} \Big( 1+|X_n|^2 \Big)^{1/2} + O(h^4).
\end{align*}

Hence condition \eqref{eq:14} is satisfied with $p_1=2$. On the other hand we have

\begin{align}\label{eq:20}
    &\Big| X_{n+1}^M - X_{n+1} \Big| 
    \\ \nonumber
    &\leq  h^2 \big( K + K_{method} {K_1} \big) \Big( 1+|X_n|^2 \Big)^{1/2} 
    \\ \nonumber
    &+ \eta \left| \sum_{j=1}^m ( g_j(t_n,X_n)-g_j(t_n,\hat{X}_n)) I_{(j)} 
    + \sum_{j_1=1}^m \sum_{j_2=1}^m (L^{j_1} g_{j_2}(t_n,X_n)-L^{j_1} g_{j_2}(t_n,\hat{X}_n)) I_{(j_1,j_2)} \right|
    \\ \nonumber
    & \leq h^2 \Big( K + K_{method} {K_1} \Big) \Big( 1+|X_n|^2 \Big)^{1/2} + \eta {K_1} \Big| X_n - \hat{X}_n \Big| 
      \left|  \sum_{j=1}^m \Big| I_{(j)} \Big| + \sum_{j_1=1}^m \sum_{j_2=1}^m \Big| I_{(j_1,j_2)} \Big| \right| 
    \\ \nonumber
     & \leq 
     h \left( \Big( K + K_{method} {K_1} \Big) h + \eta K_{method} K_1^2 
    \left|  \sum_{j=1}^m \Big| I_{(j)} \Big| + \sum_{j_1=1}^m \sum_{j_2=1}^m \Big| I_{(j_1,j_2)} \Big| \right| 
    \right) \Big( 1+|X_n|^2 \Big)^{1/2} .
\end{align}

%

Therefore using the inequality $(a_1+...+a_n)^2 \leq n(a_1^2+...+a_n^2)$ we get

\begin{align*}
    |X_{n+1}^M-X_{n+1}|^2 
    \leq h^2 K' \left( h^2 K'' + \sum_{j=1}^m I_{(j)}^2 + \sum_{j_1=1}^m \sum_{j_2=1}^m I_{(j_1,j_2)}^2 \right)
    \Big( 1+|X_n|^2 \Big) 
\end{align*}
for some positive constants $K'$ and $K''$. From Lemma 5.7.2 in \cite{Kloeden1992} it is known that $\mathsf{E}\Big[ I_{(j)}^2 | \mathcal{F}_n \Big] \leq O(h)$ and $\mathsf{E} \Big[ I_{(j_1,j_2)}^2 | \mathcal{F}_n \Big] \leq O(h^2)$. Thus we obtain the estimate

\begin{align}\label{eq:21}
	\Big| \mathsf{E} \Big[ \Big( X_{n+1}^M-X_{n+1} \Big)^2 | \mathcal{F}_n \Big] \Big|^{1/2} \leq 
	\sqrt{K'} \Big( 1+|X_n|^2 \Big)^{1/2} h^{3/2}.
\end{align}

This completes the proof.
}
\end{proof}

%
%

\section{Mean-square stability properties}
\label{sec:4}

Similarly to deterministic ODEs, stability analysis of numerical schemes for SDEs has been mainly focused on the ability of the scheme to preserve qualitative behavior of properly designed scalar linear test equation. We have already mentioned that in stochastic theory this approach is in general not justified because it does not capture the role of the diffusion structure. 
\newred{
It is known that stochastic perturbations destabilize deterministic systems in the mean square sense. For stiff and non-normal systems, geometry of the perturbation can be more important factor than intensity of the noise.
This effect cannot happen in one dimension and therefore is not captured by scalar linear stability analysis \cite{Higham2006, Buckwar2010}.
Buckwar and Kelly in \cite{Buckwar2010} provided several test systems which can be used to illustrate how the stability of the SDEs depend on the interaction of the drift and diffusion geometries.
}

Motivated by these results, we will investigate stability of the following linearized multi-channel system \cite{Buckwar2012a}

\begin{align}\label{eq:22}
	&dX(t) = F X(t) dt + \sum_{r=1}^m G_r X(t) dW_r(t)
	, \quad X \in \mathbb{R}^d
	, \quad F, G_r \in \mathbb{R}^{d \times d},
	\\ \nonumber
	&X(t_0) = X_0	.
\end{align}

\begin{Definition}
(\cite{Khasminskii2012, Buckwar2012a})
The zero solution of system \eqref{eq:22} is said to be 

1. \textit{Mean-square stable}, if for each $\epsilon>0$, there exist a $\delta \geq 0$ such that
$$
    || X(t;t_0,X_0) ||_{L_2}^2 < \epsilon, \quad t \geq t_0
$$
whenever $||X_0||_{L_2}^2 < \delta$.

2. \textit{Asymptotically mean-square stable}, if it is mean-square stable and, when $||X_0||_{L_2}^2 < \delta$ 
$$
    || X(t;t_0,X_0) ||_{L_2}^2 \to 0 \text{ for } t \to \infty
$$

where $||X(t)||_{L_2} = \Big( \mathsf{E}  | X(t)|^2  \Big)^{1/2}$ denotes the mean-square norm.
\end{Definition}

On the other hand the zero solution of system \eqref{eq:22} is asymptotically mean-square stable if and only if the zero solution of deterministic differential system of second moments is asymptotically stable \cite{Arnold1974, Khasminskii2012}

\begin{align}\label{eq:23}
	&d \mathsf{E} \Big[ Y(t) \Big] = \left( F \mathsf{E} \Big[ Y(t) \Big] + \mathsf{E} \Big[ Y(t) \Big] F^T + \sum_{r=1}^m G_r \mathsf{E} \Big[ Y(t) \Big] G_r^T \right) dt,
	\\ \nonumber
	&Y(t) = X(t) X^T(t).
\end{align}

Applying the vectorization operation to both sides of \eqref{eq:23} yields

\begin{align}\label{eq:24}
	&d \mathsf{E} \Big[ Z(t) \Big] = S \mathsf{E} \Big[ Z(t) \Big] dt,
	\\[1em] \nonumber
	&Z(t) = vec \Big( Y(t) \Big) 
	\\ \nonumber
	&= \Big( X_1^2(t), X_2(t) X_1(t),...,X_d(t) X_1(t), X_1(t)X_2(t), X_2^2(t),...,X_d(t)X_2(t),...,X_d^2(t) \Big)
\end{align}
with the stability matrix

\begin{align}\label{eq:25}
	S_{dif} = I \otimes F + F \otimes I + \sum_{j=1}^m G_j \otimes G_j.
\end{align}
Here $\otimes$ denotes the Kronecker product.

Stability condition of the system in \eqref{eq:23} is given now by the following

\begin{Lemma}
(\cite{Buckwar2012a})
The zero solution of deterministic system \eqref{eq:23} is asymptotically stable if and only if

\begin{align}\label{eq:26}
	\alpha (S_{dif}) < 0
\end{align}
where $\alpha (S_{method}) = max_i \mathfrak{R}(\eta_i)$ is the spectral abscissa of matrix $S$ and $\mathfrak{R}$ is the real part of the eigenvalues $\eta_i$ of the matrix $S$. 
\end{Lemma}

Analogously, for the numerical method written in the form of the recurrence

\begin{align}\label{eq:27}
	X_{n+1} = R_n X_n, \quad n=0,1,...
\end{align}
with sequence of random matrices $R_n$, the stability condition is given by

\begin{Lemma}
(\cite{Buckwar2012a})
The zero solution of deterministic system \eqref{eq:27} is asymptotically stable if and only if

\begin{align}\label{eq:28}
	\rho(S) < 1
\end{align}
where $S = \mathsf{E}(R_n \otimes R_n)$ and $\rho$ is the spectral radius of matrix $S$.
\end{Lemma}

{\blue \textit{Remark.} It is convenient to represent the stability matrix of the numerical method as the sum of two matrices, one corresponding to the deterministic part of the system and the other corresponding to the stochastic part

$$
	S_{method} = S^{det} + S^{stoch}.
$$
For multi-channel stochastic systems the major distinction between stability properties of the original differential system and the numerical method is contributed by matrix $S^{stoch}$. For increasing number of noise channels the stability condition of classical numerical methods becomes very restrictive destroying their ability to preserve the qualitative behavior of the differential system and making them absolutely impractical for applied purposes. 

%
}

In the following we will construct stability matrices for SSCTM, MSSCTM, SSAMM and MSSAMM methods for SDEs with different diffusion and drift structures {\blue and show how the splitting technique with proper choice of drift integrator can significantly improve efficiency of the numerical method}.

\subsection{Stability of Milstein method}

Stability of split-step methods depends on the corresponding properties of deterministic and stochastic integrators. In this paper, the stochastic part of SDEs is treated with the Milstein method. Therefore to understand the role of diffusion geometry in the stability of the numerical solution of stochastic systems we have to analyze this method first. In this section we investigate stability of the classical Milstein method applied to systems with commutative and non-commutative noise structures.

The Milstein method applied to the SDE \eqref{eq:22} has the form
\begin{align}\label{eq:29a}
	X_{n+1} = 
	 \left( I + h F + \sqrt{h} \sum_{r=1}^m G_r \xi_n^r + \sum_{r_1,r_2=1}^m G_{r_1} G_{r_2} I_{(r_1,r_2)} \right) X_n.
\end{align}

\paragraph{\textbf{Commutative noise}}
For linear SDE \eqref{eq:22} commutativity condition reads as $G_{j_1} G_{j_2} = G_{j_2} G_{j_1}$. Together with the identity $ I_{(j_1,j_2)} = I_{(j_1)} I_{(j_2)} + I_{(j_2)} I_{(j_1)}$ the Milstein method in \eqref{eq:29a} converts to

\begin{align}\label{eq:29}
X_{n+1} =
	 \underbrace{ 
	 \left( I + h F - \frac{h}{2} \sum_{r=1}^m G_{r}^2 + \sqrt{h} \sum_{r=1}^m G_{r} \xi_n^r + \frac{h}{2} \sum_{r_1,r_2=1}^m G_{r_1} G_{r_2} \xi_n^{r_1} \xi_n^{r_2} 
	  \right) 
	 }_\text{$R_n$} X_n
\end{align}

\paragraph{\textbf{Non-commutative noise}} 
It is well known that for SDEs with a non-commutative noise it is impossible to build high-order strong schemes which contains Wiener increments only. One also needs to calculate multiple stochastic integrals $I_{(j_1,j_2)}$. Authors of \cite{Kloeden1992} proposed the following representation of the stochastic integrals based on the Karhunen-Lo\'eve expansion

\begin{align}\label{eq:32}
	\nonumber
	I_{(r)} &= \sqrt{h} \xi_{r},
	\qquad
	I_{(r_1,r_2)} = \frac{1}{2}(I_{r_1} I_{r_2} - h \delta_{r_1,r_2}) + A_{(r_1,r_2)},
	\\
	A_{(r_1,r_2)} &= \frac{h}{2 \pi} \sum_{k=1}^{\infty} \frac{1}{k} \Big[ \chi_{r_1,k} (\zeta_{r_2,k}+\sqrt{2} \xi_{r_2}) - \chi_{r_2,k} (\zeta_{r_1,k}+\sqrt{2} \xi_{r_1})\Big],
\end{align}
where $\xi_{r}$, $\chi_{r,k}$ and $\zeta_{r,k}$ are normally distributed independent random variables and $\delta_{r_1,r_2}$ is the Kronecker delta. It was shown in \cite{Kloeden1992} that by taking $p \geq 1/h$ terms in the truncated series in \eqref{eq:32} numerical scheme preserves the strong order 1 of convergence.

Thus, Milstein method applied to the system \eqref{eq:22} with non-commutative noise reads

\begin{align}\label{eq:33}
	X_{n+1} = 
	 \underbrace{ 
	 \left( I + h F - \frac{h}{2} \sum_{r=1}^m G_{r}^2 + \sqrt{h} \sum_{r=1}^m G_{r} \xi_n^r + \frac{h}{2} \sum_{r_1,r_2=1}^m G_{r_1} G_{r_2} \xi_n^{r_1} \xi_n^{r_2} 
	 + \sum_{\substack{
                       r_1,r_2=1\\
                       r_1 \neq r_2
                      }
            }^m
     G_{r_1} G_{r_2} A_{(r_1,r_2)}^p \right) 
	 }_\text{$R_n$} X_n,
\end{align}
where $A_{(r_1,r_2)}^p$ is a truncated series from \eqref{eq:32}.


\begin{Theorem}
The mean-square stability matrix of the Milstein method applied to the system \eqref{eq:22} is given by

%
%

{\red

\begin{align}\label{eq:31}
	S^{Mil}&= S_{Mil}^{det} + S_{Mil}^{stoch},
\end{align}
where matrices $S_{Mil}^{det}$ and $S_{Mil}^{stoch}$ are determined by

\begin{align*}
	\nonumber
	S_{Mil}^{det} &= (I+hF) \otimes (I+hF) ,
	\\ \nonumber
	S_{Mil}^{stoch} &= 
	h\sum_{r=1}^m G_r \otimes G_r
			+ \frac{h^2}{2} \sum_{r=1}^m \left(G_{r}^2 \otimes G_{r}^2 \right) 	
			+ \frac{h^2}{4} \left(
			\sum_{\substack{
		                    r_1,r_2=1\\
		                    r_1 \neq r_2
		                   }
		         }^m
		    G_{r_1} G_{r_2} 
		    	\otimes 
		    \sum_{\substack{
		                    r_1,r_2=1\\
		                    r_1 \neq r_2
		                   }
		         }^m
		    G_{r_1} G_{r_2} 
		    \right)
		    \\ 
	&+ \delta_{com}
	\frac{h^2}{4}
	\left(
	\sum_{\substack{
                       r_1,r_2=1\\
                       r_1 < r_2}
                       }^m G_{r_1} G_{r_2}
	-\sum_{\substack{
                       r_1,r_2=1\\
                       r_1 > r_2}
                       }^m G_{r_1} G_{r_2}
	\right)
	\otimes 
	\left(
	\sum_{\substack{
                       r_1,r_2=1\\
                       r_1 < r_2}
                       }^m G_{r_1} G_{r_2}
	-\sum_{\substack{
                       r_1,r_2=1\\
                       r_1 > r_2}
                       }^m G_{r_1} G_{r_2}
	\right)
\end{align*}
with 

\begin{align}\label{eq:31a}
	\delta_{com} = \left\{
	\begin{array}{l l}
		0, & \quad \text{- commutative noise}
		\\
		1, & \quad \text{- non-commutative noise}
	\end{array}
	\right. .
\end{align}
}
\end{Theorem}

\begin{proof}
	Proof follows from Theorems 3.9 and 3.10 in \cite{Buckwar2012a}.
\end{proof}

{\blue
\textit{Example.} (\cite{Buckwar2011}) In 1-dimensional case condition \eqref{eq:28} with matrix \eqref{eq:31} for commutative noise reads as

$$
\underbrace{
\mathfrak{R} (f) + \frac{1}{2} \sum_{r=1}^M |g_r|^2
}_{S_{dif}}
+ 
\frac{1}{2}h |f|^2 + \frac{1}{4}h \sum_{r=1}^M |g_r|^4 
		+ \frac{1}{8}h \Bigg| \sum_{\substack{
                       r_1,r_2=1\\
                       r_1 \neq r_2}
                       }^M g_{r_1} g_{r_2} \Bigg|^2 
< 0 .
$$
This example clearly illustrates that classical Milstein method has extremely poor stability properties and is not able to catch the qualitative behavior of the stochastic systems with multi-dimensional noise.
}

\subsection{Stability of split-step methods}
In this section we establish stability matrices in the form convenient for analyzing their impact arising from different deterministic drift solvers.

Split-step Milstein method \eqref{eq:4} applied to the linear system \eqref{eq:22} reads


{\red
\begin{align}\label{eq:36}
	\nonumber
	\hat{X}_n &= X_n + h \Phi \left( t_n,t_{n+1}, X_n, \hat{X}_n \right) = P X_n
	\\ \nonumber
	X_{n+1} &= 
	 \left( I + \left[ \sqrt{h} \sum_{r=1}^m G_r \xi_n^r  + \sum_{r_1,r_2=1}^m G_{r_1} G_{r_2} I_{(r_1,r_2)} \right] 
    \Bigg[ \eta I + (1-\eta) P^{-1} \Bigg]	\right)
	 P X_n
	 \\
	 &=
	 \underbrace{
	 \left( 
	 Q^{-1} + \sqrt{h} \sum_{r=1}^m G_{r} \xi_n^r + \frac{h}{2} \sum_{r_1,r_2=1}^m G_{r_1} G_{r_2} \xi_n^{r_1} \xi_n^{r_2} 
	 + \delta_{com} \sum_{\substack{
                       r_1,r_2=1\\
                       r_1 \neq r_2
                      }
            }^m
     G_{r_1} G_{r_2} A_{(r_1,r_2)}^p
	\right) 
	 }_\text{$R_n$} QP X_n ,
\end{align}
where $\delta_{com}$ is defined in \eqref{eq:31a}, $P$ is a stability matrix of given deterministic solver and matrix $Q=\eta I + (1-\eta) P^{-1}$.
}


Modified split-step Milstein method \eqref{eq:5} applied to the linear SDE \eqref{eq:22} has the form similar to the above with matrix $P$ given by $X_n + h\Phi \left( t_n,t_{n+1}, X_n, \hat{X}_n \right) - h \frac{1}{2} \sum_{r=1}^m G_{r}^2 \hat{X}_n
	= P X_n$ and multiple It\^o integrals substituted by Stratonovich integrals. 

{\red
\begin{Theorem}
The mean-square stability matrix of the split-step and modified split-step Milstein methods in \eqref{eq:36} applied to the system \eqref{eq:22} is given by

\begin{align}\label{eq:38}
	S &= \mathsf{E} (R_n Q P \otimes R_n Q P) = \mathsf{E} (R_n \otimes R_n) (QP \otimes QP),
\end{align}
where the split-step method \eqref{eq:4} yields 

\begin{align}\label{eq:39}
	S_{SSM} = \Big( \big( Q \otimes Q \big)^{-1} + S_{Mil}^{stoch} \Big) \Big( QP \otimes QP \Big)
\end{align}
and for the modified split-step method \eqref{eq:5} we have

\begin{align}\label{eq:41}
	S_{MSSM}
			= \Bigg( \big( Q \otimes Q \big)^{-1} + S_{Mil}^{stoch} + \frac{h}{2}\sum_{r=1}^m \Big( (G_r^2 \otimes Q^{-1}) + (Q^{-1} \otimes G_r^2) \Big) + \frac{h^2}{4} \sum_{r=1}^m (G_r^2 \otimes G_r^2) \Bigg) \Bigg( QP \otimes QP \Bigg) .
\end{align}

%

%

\end{Theorem}

\begin{proof}
Proof follows immediately from Theorem 4.1.
\end{proof}
}

\subsection{Deterministic stability matrices of SSCTM, MSSCTM, SSAMM and MSSAMM methods}
It is straightforward to verify that deterministic stability matrices of SSCTM, MSSCTM, SSAMM and MSSAMM methods are given as follows

\begin{align}\label{eq:43}
	P_{SSCTM} &= \Big(I - h \theta F \Big)^{-1} \Big(I + h (1-\theta) F \Big),
	\\ \label{eq:44}
	P_{MSSCTM} &= \Big( I - h \theta F + \eta H \Big)^{-1} \Big( I + h (1-\theta) F - (1-\eta) H \Big),
	\\ \label{eq:45}
	P_{SSAMM} &=  \Big( I - h \gamma_1 F \Big)^{-1} 
     \left(I + h F \left( \frac{1}{2} I +  \theta  \Big( I - h \gamma_1 F \Big)^{-1} \Big(I + h \gamma_2 F \Big) \right) \right),
     \\ \label{eq:46}
	P_{MSSAMM} &=  \Big( I - h \gamma_1 F + \eta H \Big)^{-1} 
     \left(I + h F \left( \frac{1}{2} I +  \theta  \Big( I - h \gamma_1 F \Big)^{-1} \Big(I + h \gamma_2 F \Big) \right) - (1-\eta) H \right),
\end{align}
where $	H = \frac{h}{2} \sum_{r=1}^m G_{r}^2$, $ \gamma_1 = \frac{1}{2} - \theta$, $\gamma_2 = \frac{1}{2} + \theta $


\subsection{Application to different test equations}

\newblue{
The standard rule of thumb for analysis of stability of numerical methods is to build and compare their stability regions, i.e. the sets of points in the space of system parameters for which the method is stable in some specified sense. Unfortunately, this approach does not allow to visualize stability of general stochastic systems since we are limited to three parameter dimensions only. We cannot hope to consider all possible combinations of drift and diffusion geometries with this number of available parameters. Therefore there is no simple way to choose a universal test equation for system of SDEs.

To illustrate stability properties of the proposed numerical schemes we consider three test systems which, in our opinion, illustrate the most important factors to numerical stability of SDEs. 
}
%

First test system has simple diagonal drift matrix and two non-commutative \newred{destabilizing in the mean-square} noise terms:

\begin{align}\label{eq:47}
	dX(t) = 
	\underbrace{
	\begin{pmatrix}
		\lambda & 0 \\[0.5em]
		0 & \lambda
	\end{pmatrix}
	}_\text{F}
	X(t) dt
	+
	\underbrace{
	\begin{pmatrix}
		\epsilon & 0 \\[0.5em]
		0 & -\epsilon
	\end{pmatrix}
	}_\text{$G_1$}
	X(t) dW_1(t)
	+
	\underbrace{
	\begin{pmatrix}
		0 & \sigma \\[0.5em]
		\sigma & 0
	\end{pmatrix}
	}_\text{$G_2$}
	X(t) dW_2(t).
\end{align}
\newblue{
Diffusion matrices $G_1$ and $G_2$ act along and orthogonal to the deterministic flow respectively.
This system was proposed in \cite{Buckwar2012b} and shows the effect of the discretization of different diffusion structures on the numerical stability of the method.
}

\newblue{
The form of the second test system is motivated by the study of the pseudospectra of non-normal matrices and operators \cite{Trefethen2005}. It is known that even asymptotically stable non-normal systems may exhibit large transient growth. It was shown in \cite{Higham2006, Buckwar2012a} that such systems may be destabilized in the asymptotic mean-square sense by the small amount of noise acting orthogonally to the deterministic flow. 
Thus a second test system is used to show interaction between the drift and diffusion structures. It is given by

}

\begin{align}\label{eq:48}
	dX(t) = 
	\underbrace{
	\begin{pmatrix}
		\lambda & b \\[0.5em]
		0 & \lambda
	\end{pmatrix}
	}_\text{F}
	X(t) dt
	+
	\underbrace{
	\begin{pmatrix}
		0 & \sigma \\[0.5em]
		-\sigma & 0
	\end{pmatrix}
	}_\text{$G$}
	X(t) dW(t),
\end{align}
\newblue{where parameter $b$ determines the departure from the normality of the matrix.}

\newblue{
The aim of the third test equation is to study the impact on the stability of the numerical method arising from the dimension of the noise. In this case we are not interested in the particular structures of the drift and diffusion components and consider a one-dimensional multi-channel SDE of the form:

}

\begin{align}\label{eq:48a}
	dX(t) = 
	\lambda
	X(t) dt
	+
	\sigma X(t)
	\sum_{r=1}^m
	dW_r(t).
\end{align}

%

In the following we will compute and plot stability regions of the considered methods. All proofs are based on the particular form of the stability matrix and Lemmas 4.1 and 4.2 and hence are omitted.

We start with stability regions of the differential systems in \eqref{eq:47}-\eqref{eq:48}. Applying results of Lemma 4.1 yields

\begin{Lemma}
(\cite{Buckwar2012a, Buckwar2012b})
1. The zero solution of \eqref{eq:47} is asymptotically stable if and only if

\begin{align}\label{eq:49}
	2x + y^2 + z^2 < 0
\end{align}
with $x=h \lambda$, $y^2 = h \epsilon^2$, $z^2 = h \sigma^2$.

2. The zero solution of \eqref{eq:48} is asymptotically stable if and only if

\begin{align}\label{eq:50}
	2 x+\frac{1}{3} \sqrt[3]{27 y^4 z^2+3 \sqrt{81 y^8 z^4+48 y^4 z^8}+8 z^6}+\frac{4 z^4}{3 \sqrt[3]{27 y^4 z^2+3 \sqrt{81 y^8 z^4+48 y^4 z^8}+8 z^6}}-\frac{z^2}{3} < 0
\end{align}
with $x=h \lambda$, $y^2 = h b$, $z^2 = h \sigma^2$.
%
\end{Lemma}


\begin{figure}[p]
	\centering
		\begin{subfigure}[t]{0.45\textwidth}
                \centering
                \includegraphics[width=\textwidth]{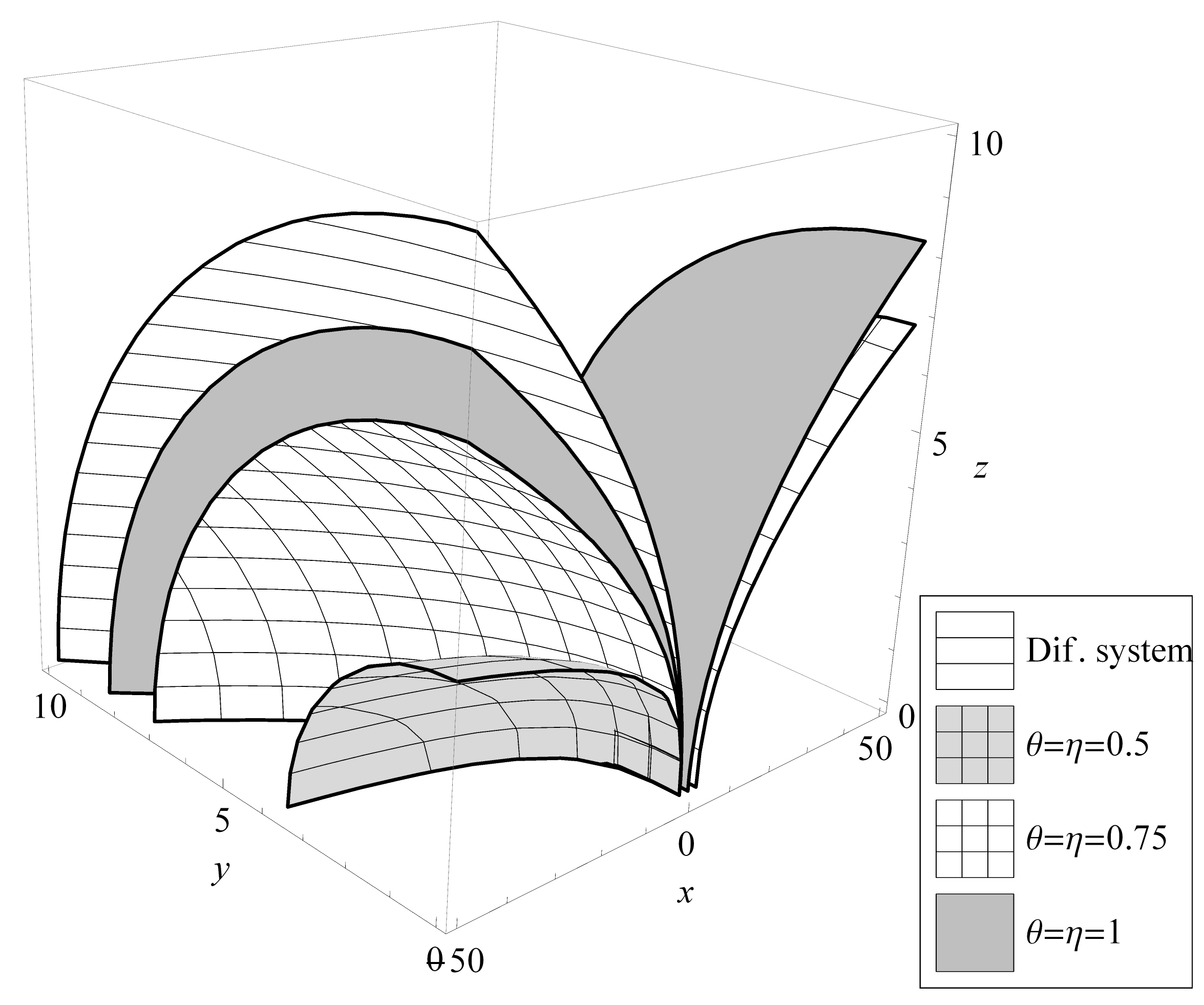}
                \caption{$\theta=\eta$}
        \end{subfigure}
	\qquad
		\begin{subfigure}[t]{0.45\textwidth}
                \centering
                \includegraphics[width=\textwidth]{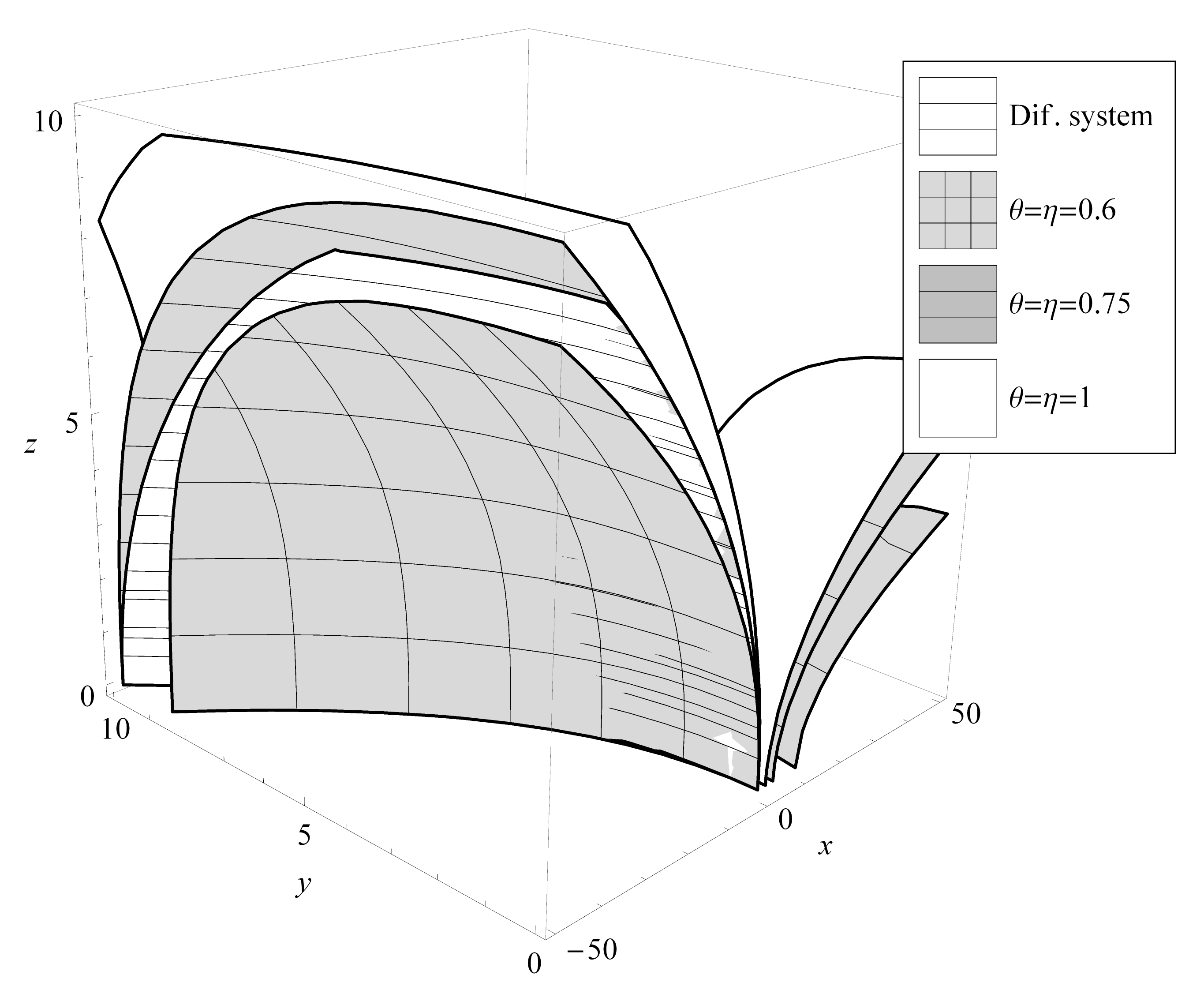}
                \caption{$\theta=\eta$}
        \end{subfigure}
	\qquad
		\begin{subfigure}[t]{0.45\textwidth}
                \centering
                \includegraphics[width=\textwidth]{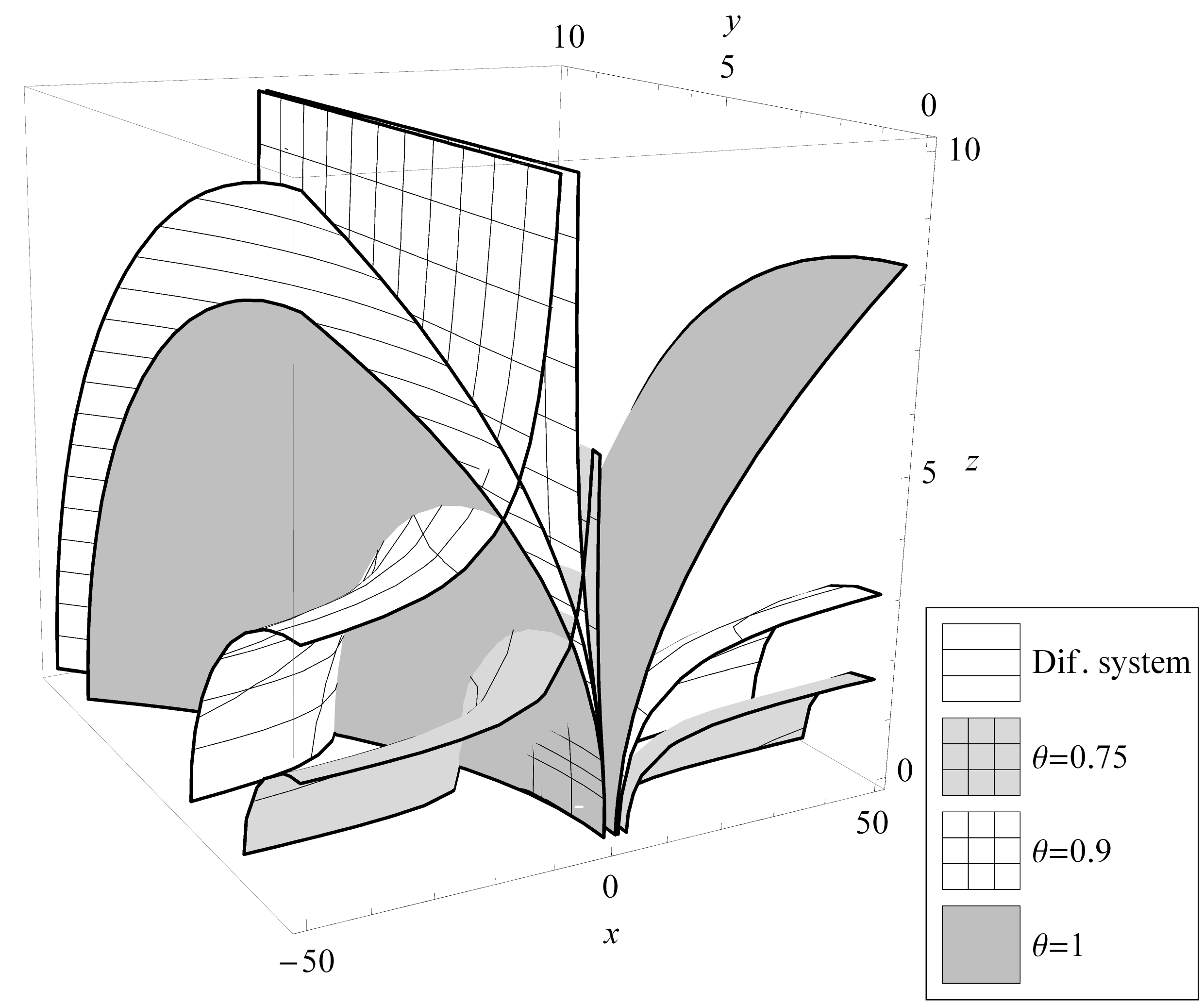}
                \caption{$\eta=1$}
        \end{subfigure}
	\qquad
		\begin{subfigure}[t]{0.45\textwidth}
                \centering
                \includegraphics[width=\textwidth]{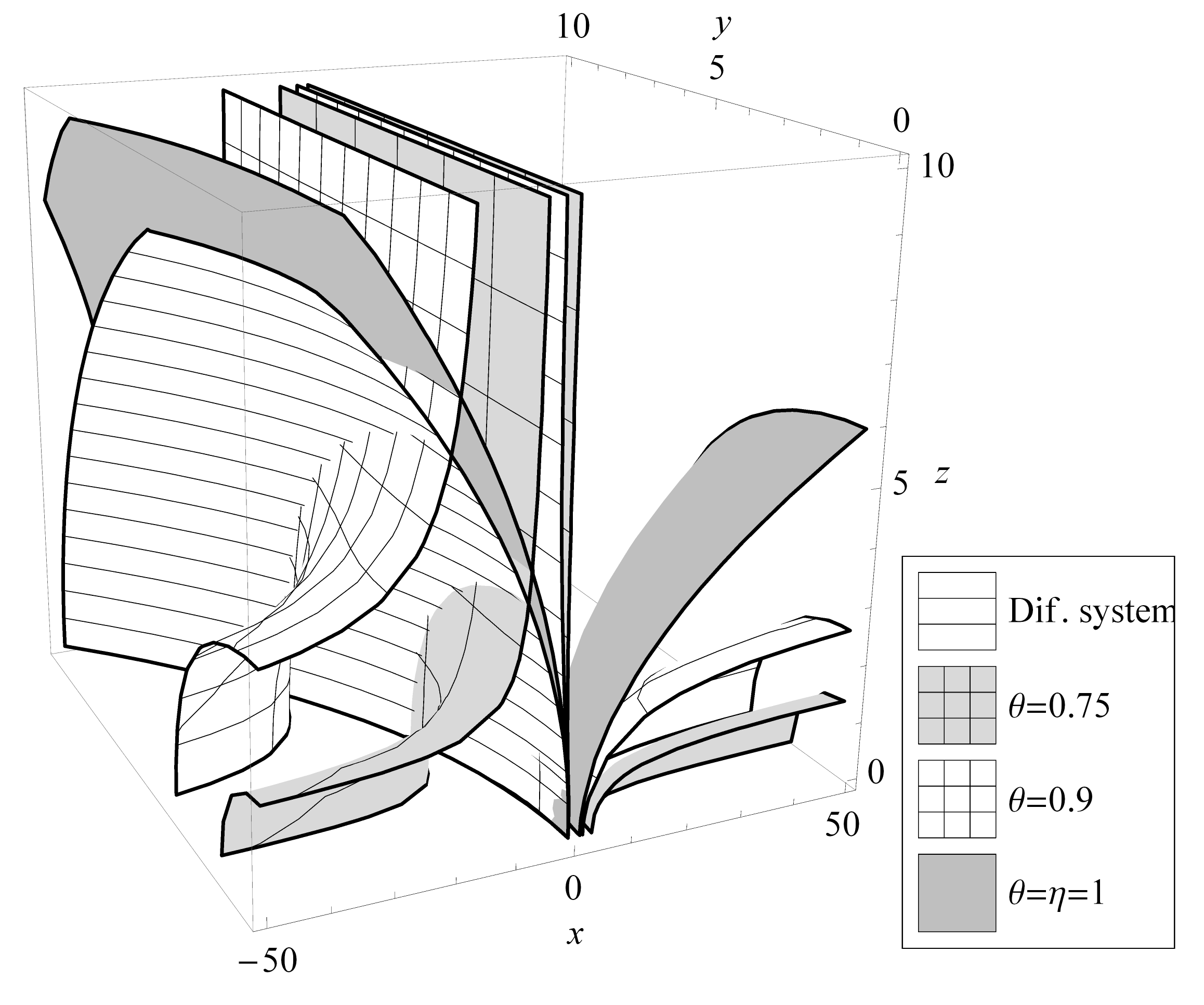}
                \caption{$\eta=1$}
        \end{subfigure}
	\qquad
		\begin{subfigure}[t]{0.45\textwidth}
                \centering
                \includegraphics[width=\textwidth]{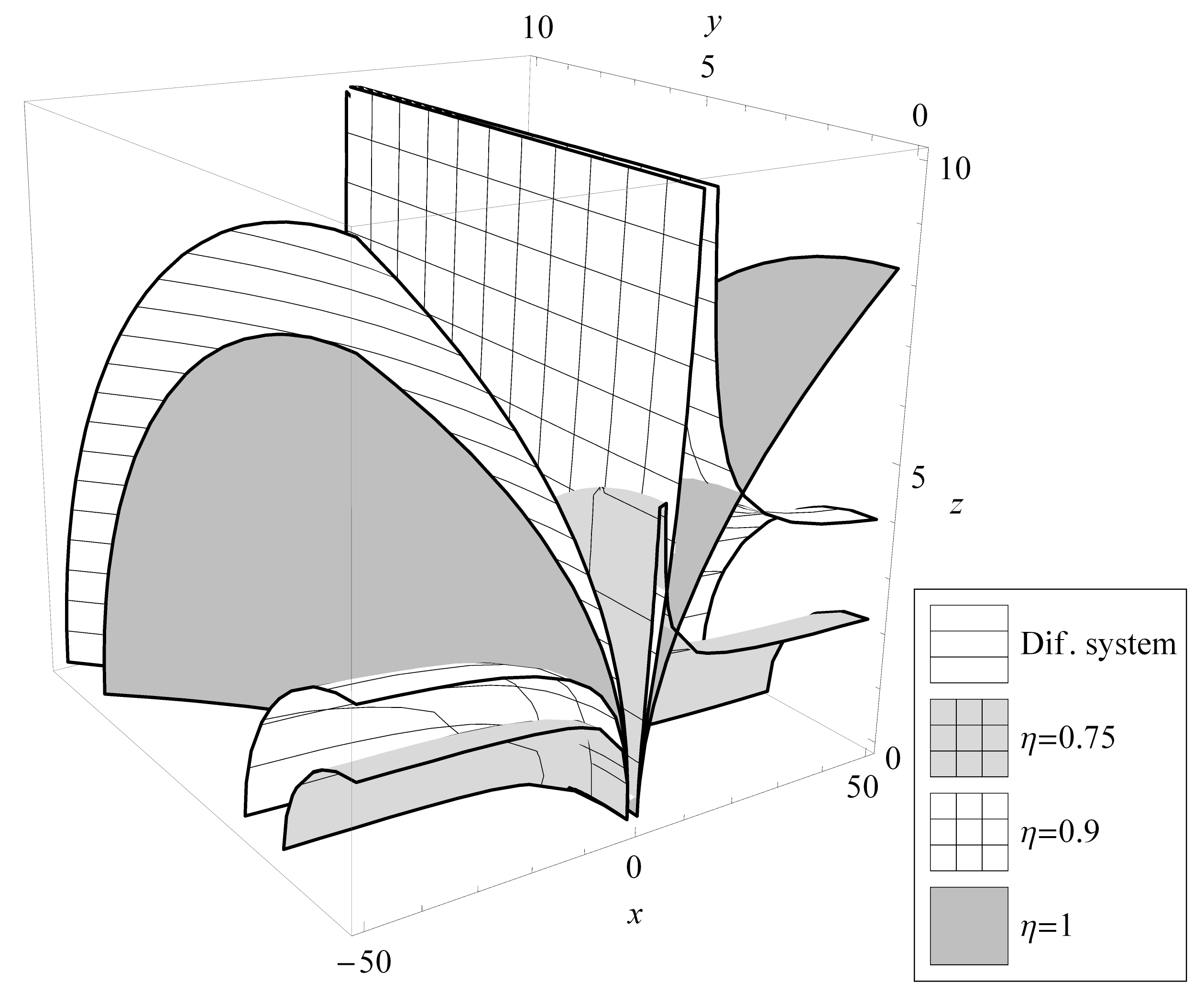}
                \caption{$\theta=1$}
        \end{subfigure}
	\qquad
		\begin{subfigure}[t]{0.45\textwidth}
                \centering
                \includegraphics[width=\textwidth]{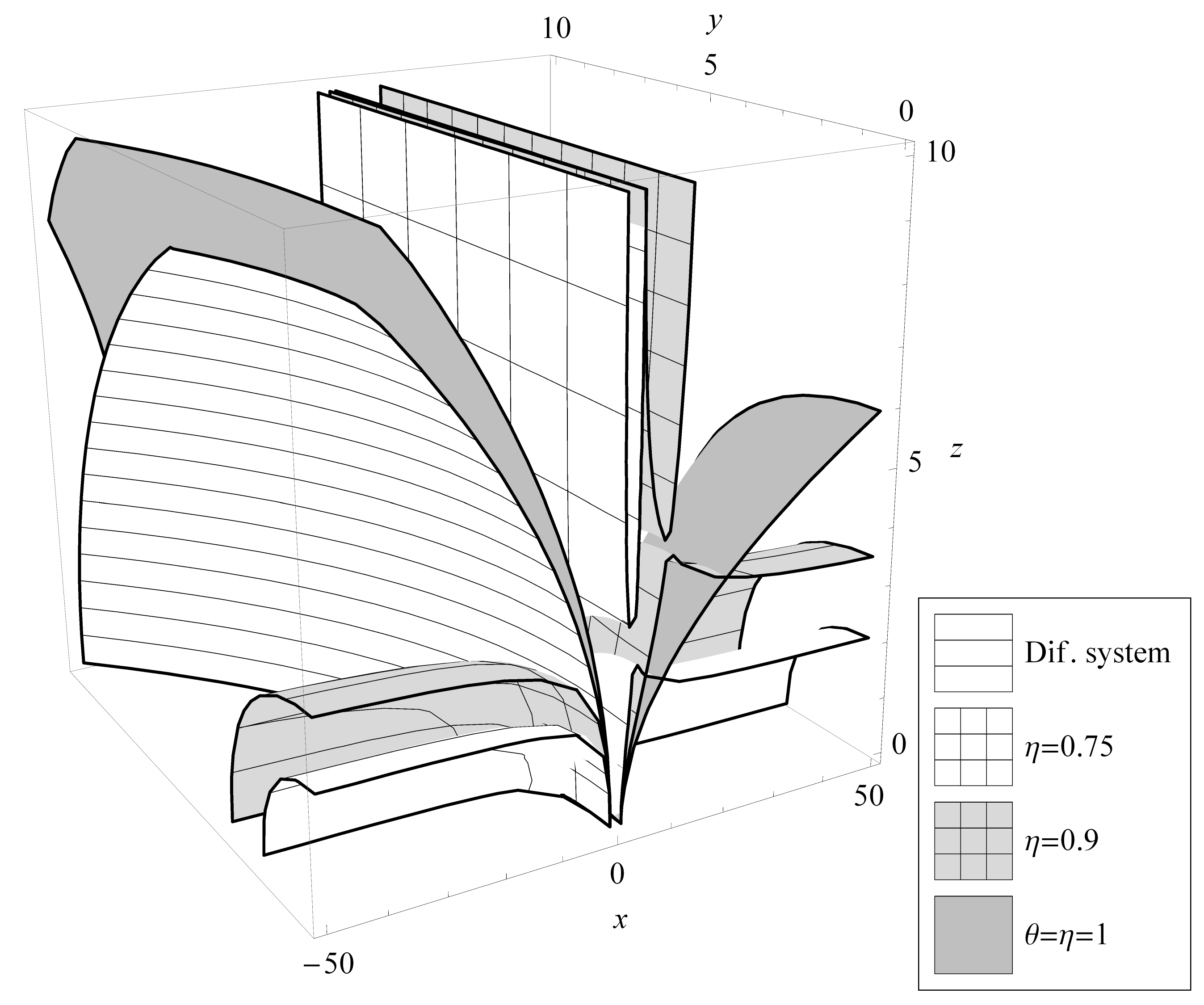}
                \caption{$\theta=1$}
        \end{subfigure}
	\caption{Portions of the mean-square stability regions for the SSCTM  (a, c, e) and MSSCTM (b, d, f) method for different values of parameters $\theta$ and $\eta$ applied to the test systems \eqref{eq:47} }
	\label{fig:1}
\end{figure}

\begin{figure}[p]
	\centering
		\begin{subfigure}[t]{0.45\textwidth}
                \centering
                \includegraphics[width=\textwidth]{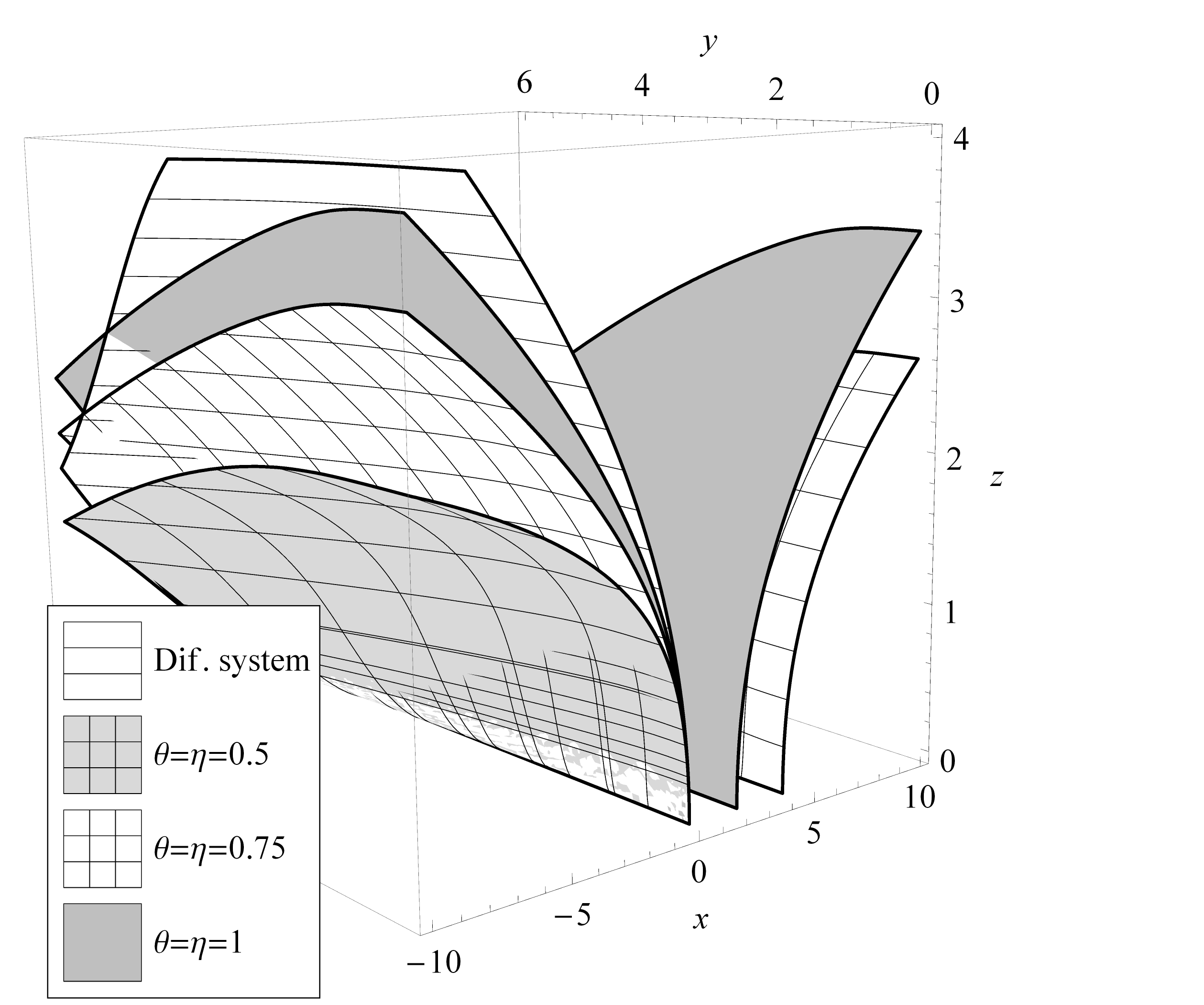}
                \caption{$\theta=\eta$}
        \end{subfigure}
	\qquad
		\begin{subfigure}[t]{0.45\textwidth}
                \centering
                \includegraphics[width=\textwidth]{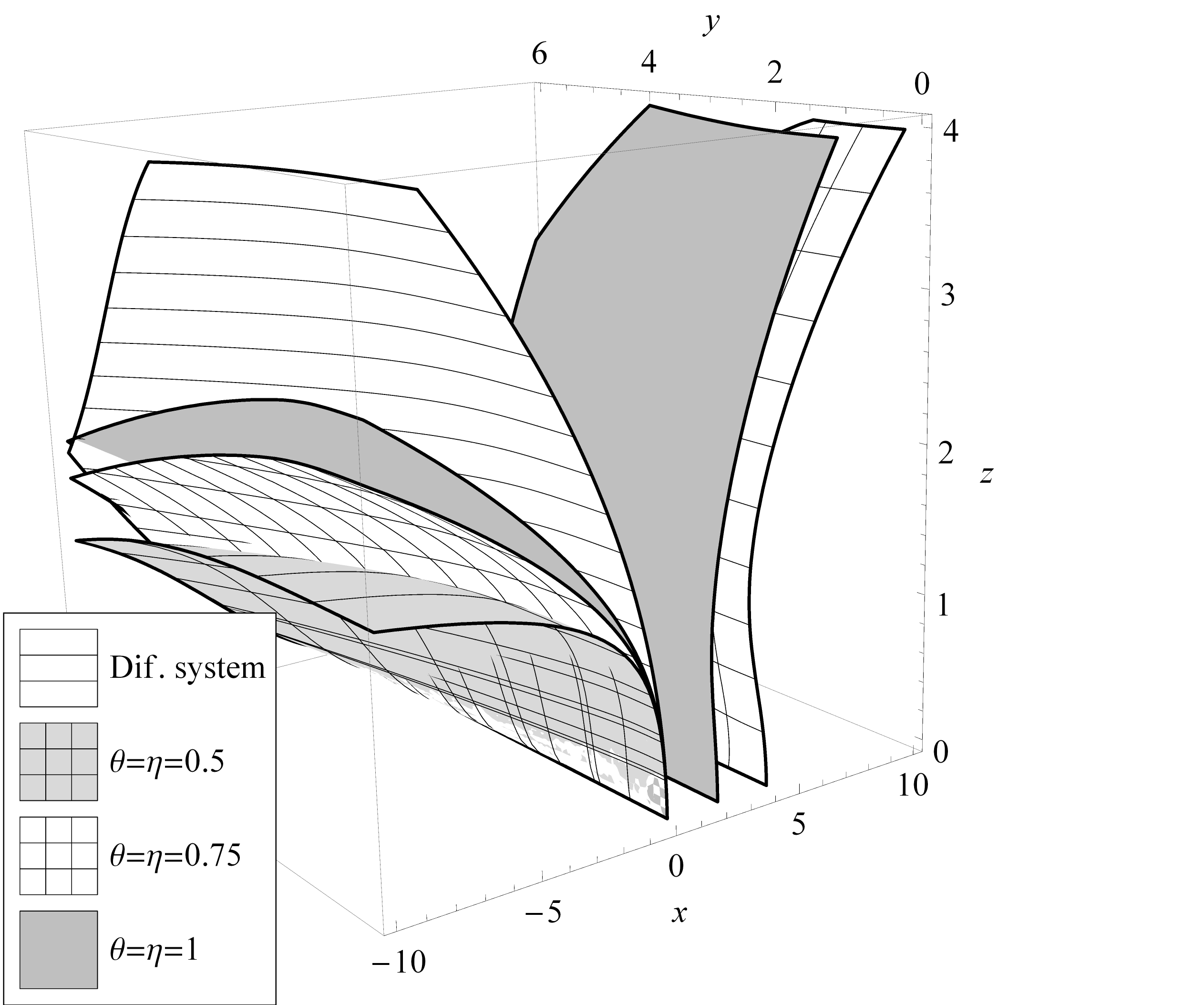}
                \caption{$\theta=\eta$}
        \end{subfigure}
	\qquad
		\begin{subfigure}[t]{0.45\textwidth}
                \centering
                \includegraphics[width=\textwidth]{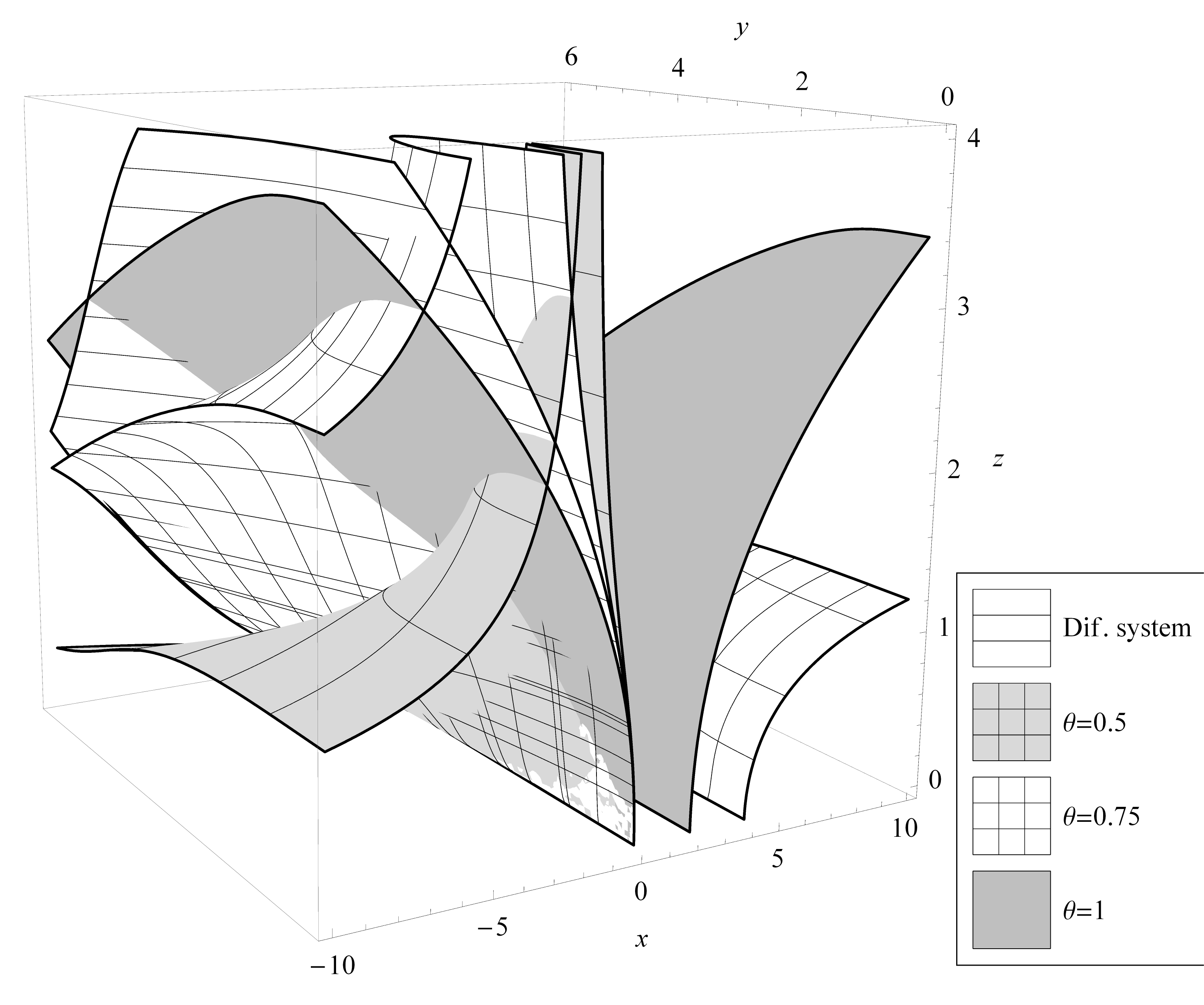}
                \caption{$\eta=1$}
        \end{subfigure}
	\qquad
		\begin{subfigure}[t]{0.45\textwidth}
                \centering
                \includegraphics[width=\textwidth]{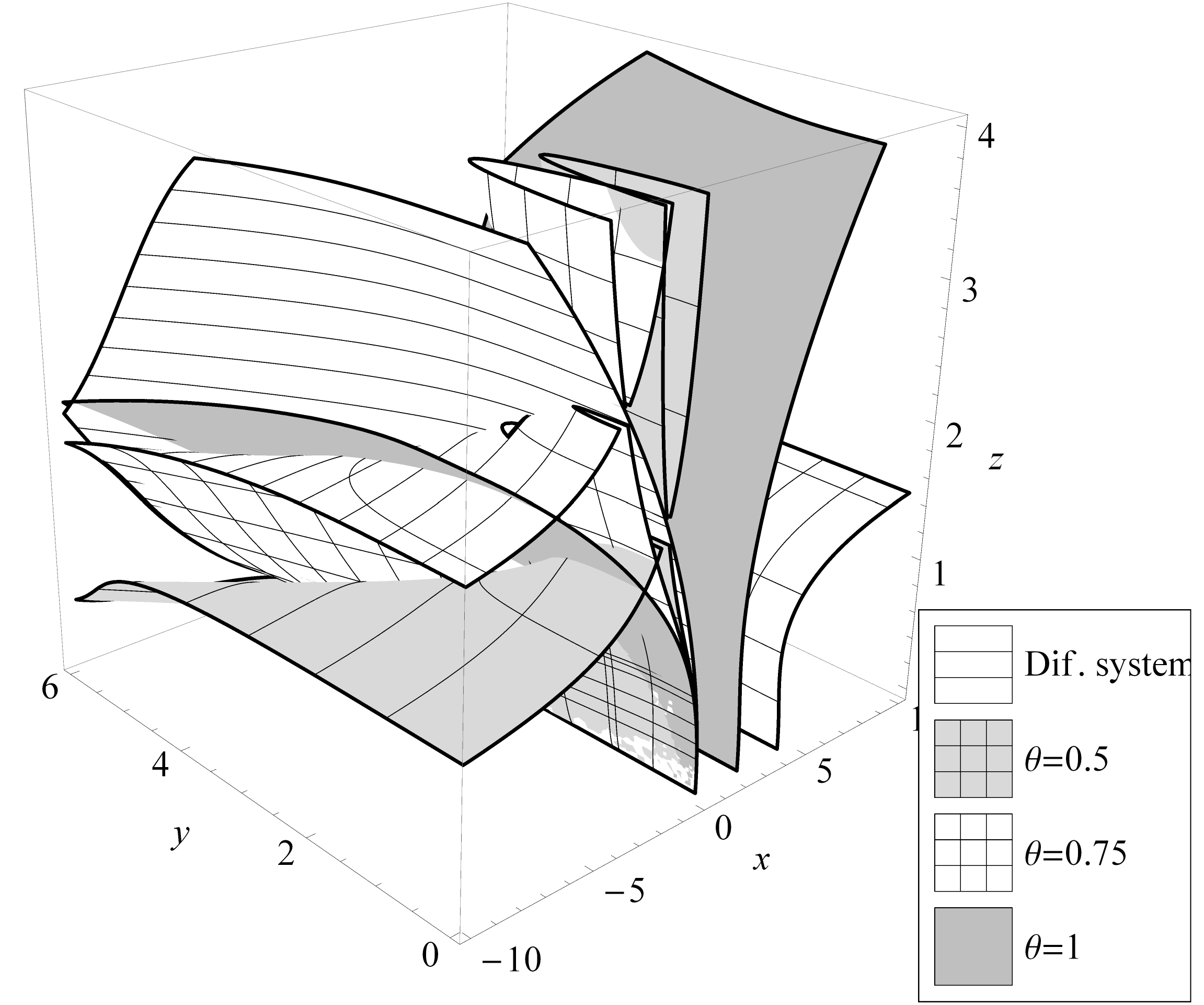}
                \caption{$\eta=1$}
        \end{subfigure}
	\qquad
		\begin{subfigure}[t]{0.45\textwidth}
                \centering
                \includegraphics[width=\textwidth]{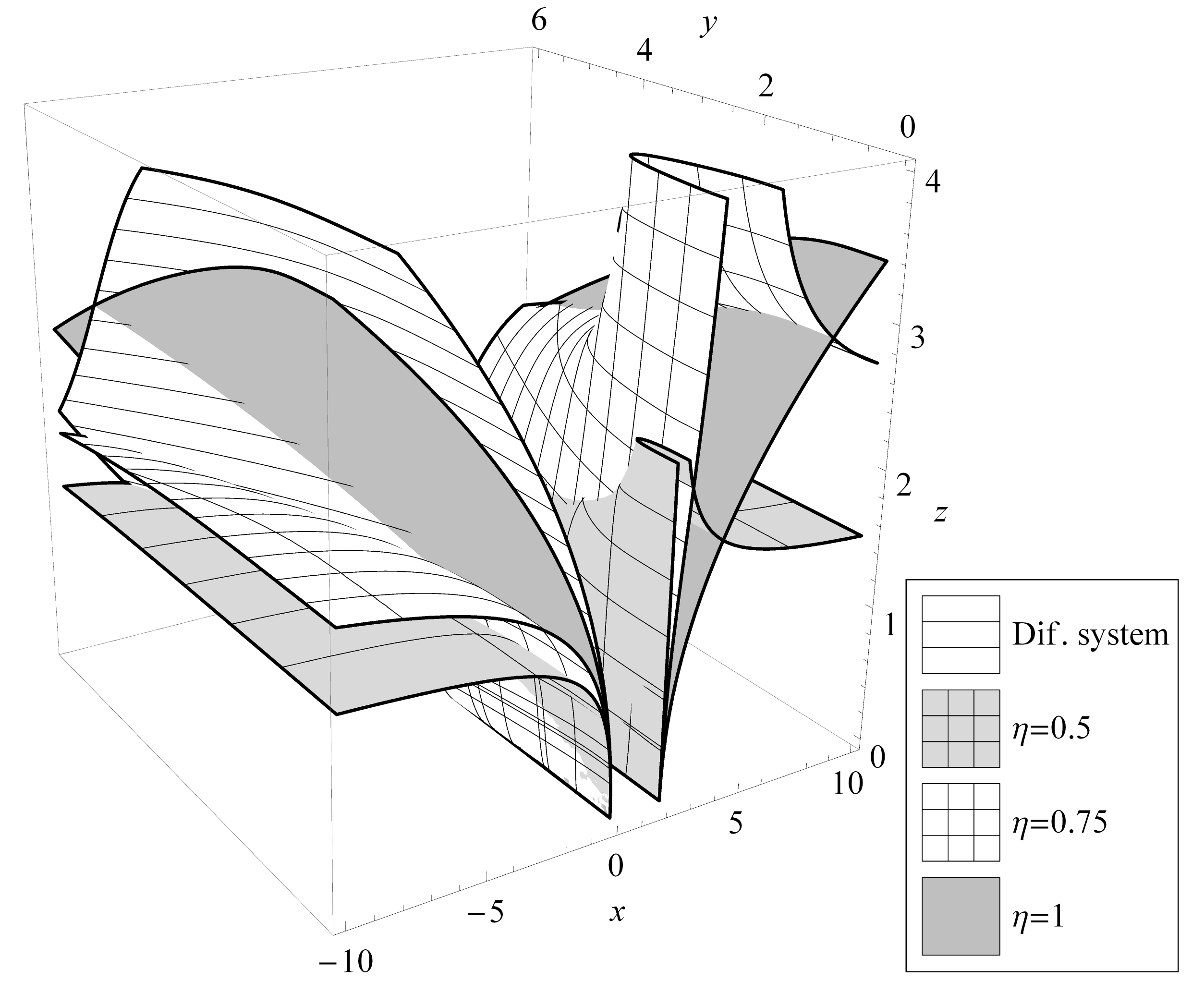}
                \caption{$\theta=1$}
        \end{subfigure}
	\qquad
		\begin{subfigure}[t]{0.45\textwidth}
                \centering
                \includegraphics[width=\textwidth]{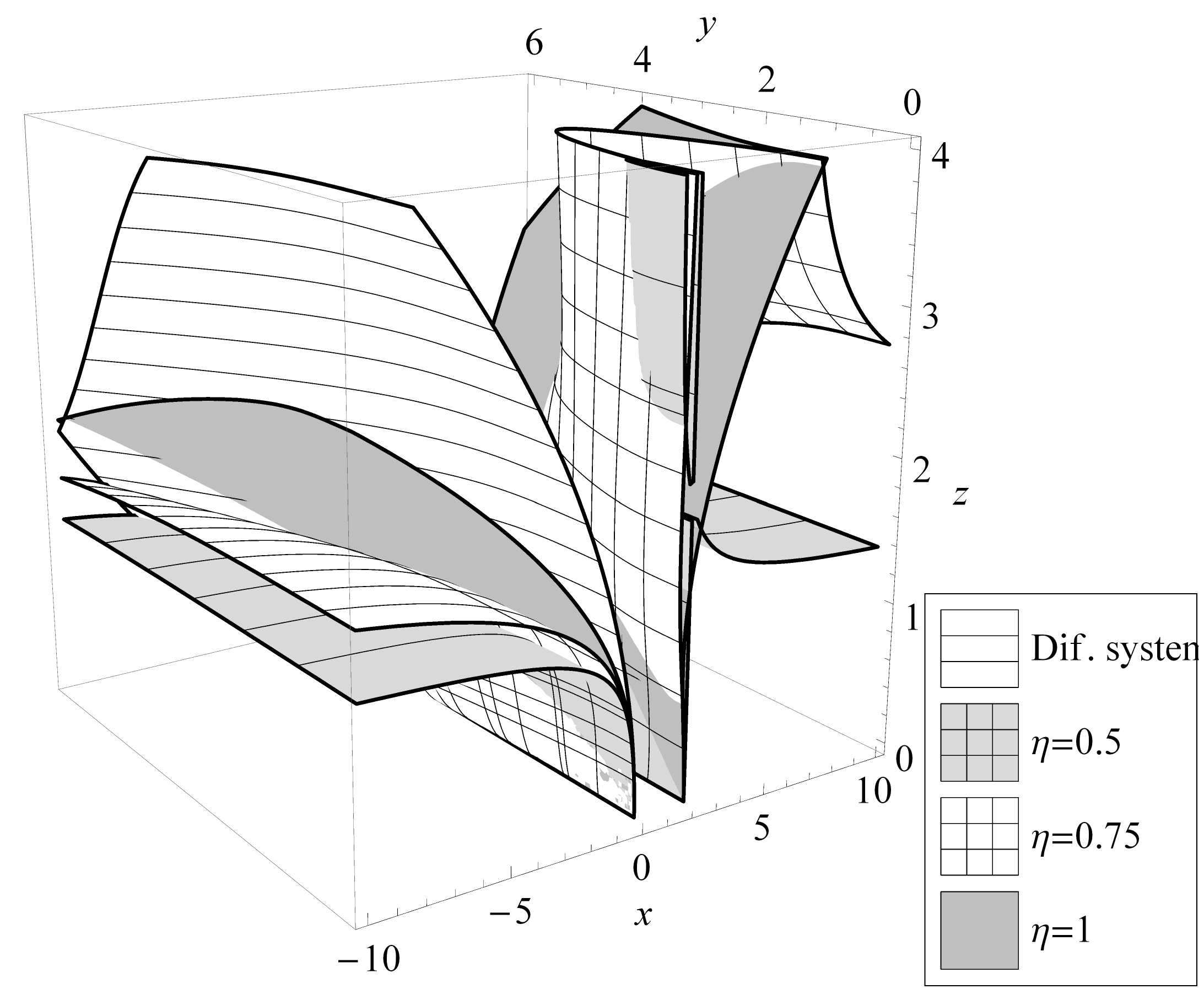}
                \caption{$\theta=1$}
        \end{subfigure}
	\caption{Portions of the mean-square stability regions for the SSCTM  (a, c, e) and MSSCTM (b, d, f) method for different values of parameters $\theta$ and $\eta$ applied to the test systems \eqref{eq:48} }
	\label{fig:2}
\end{figure}

\newblue{
SSCTM and MSSCTM methods contain two parameters, $\theta$ and $\eta$, which determine the degree of implicitness of these schemes.
At a first glance it is reasonable to expect better stability properties of the numerical solution for bigger values of $\theta$ and $\eta$. Figures \ref{fig:1} and \ref{fig:2} show that this is not true in general. One can see that for certain values of the parameters $\lambda$, $b$, $\epsilon$ and $\sigma$ in the test systems \eqref{eq:47} and \eqref{eq:48} SSCTM / MSSCTM methods possess bigger stability regions for different values of $\theta$ and $\eta$. A similar result was obtained in \cite{Guo2013} for the scalar test equation. 
On the other hand it is easy to see that stability properties of the original differential system in the whole stability domain are best recovered for the parameter values $\theta=\eta \in [0.75,1]$.

For the general differential system, the problem of finding the optimal parameters $\theta$ and $\eta$ is very difficult by itself.
Therefore it is reasonable to set $\theta=\eta=1$ unless one knows explicitly how to choose the best parameters.
These values correspond to DSSBM and MSSBM methods which were discussed in Section 2. Therefore in the following we provide stability conditions for these methods only. 

}

\begin{Lemma}
For the test equation \eqref{eq:47} the zero solution of stochastic difference equation is asymptotically stable if and only if it satisfies the following conditions:

a) for DSSBM method

\begin{align}\label{eq:52}
	1 - 2 (1-x)^2+\left(y^2+z^2+1\right)^2 <0
\end{align}

b) for MSSBM method

\begin{align}\label{eq:53}
	\left(y^2+z^2+1\right)^2 + 2 x \left(y^2+z^2-x+2\right)-\left(y^2 z^2+1\right) < 0
\end{align}

c) for SSAMM method

\begin{align}\label{eq:54}
	\left(\left(4 \theta ^2+4 \theta -1\right) x^2+8 \theta  x+4\right)^2 \left(\left(y^2+z^2+1\right)^2+1\right)-2 ((2 \theta -1) x+2)^4 <0
\end{align}

d) for MSSAMM method

\begin{align}\label{eq:54a}
	\nonumber
	&\left(\left(4 \theta ^2+4 \theta -1\right) x^2+8 \theta  x+4\right)^2 
	\left( 3(y^2+z^2)^2 + 8(y^2+z^2) - 2 y^2 z^2 + 4 \right)
	\\
	&-4 \Big( (2 \theta -1) x+2 \Big)^2 \left((2 \theta -1) + 2 x+y^2+z^2\right)^2 < 0
\end{align}
where $x$, $y$ and $z$ are defined as in \eqref{eq:49}.

\end{Lemma}

\begin{figure}[h]
	\centering
        \begin{subfigure}[t]{0.45\textwidth}
                \centering
                \includegraphics[width=\textwidth]{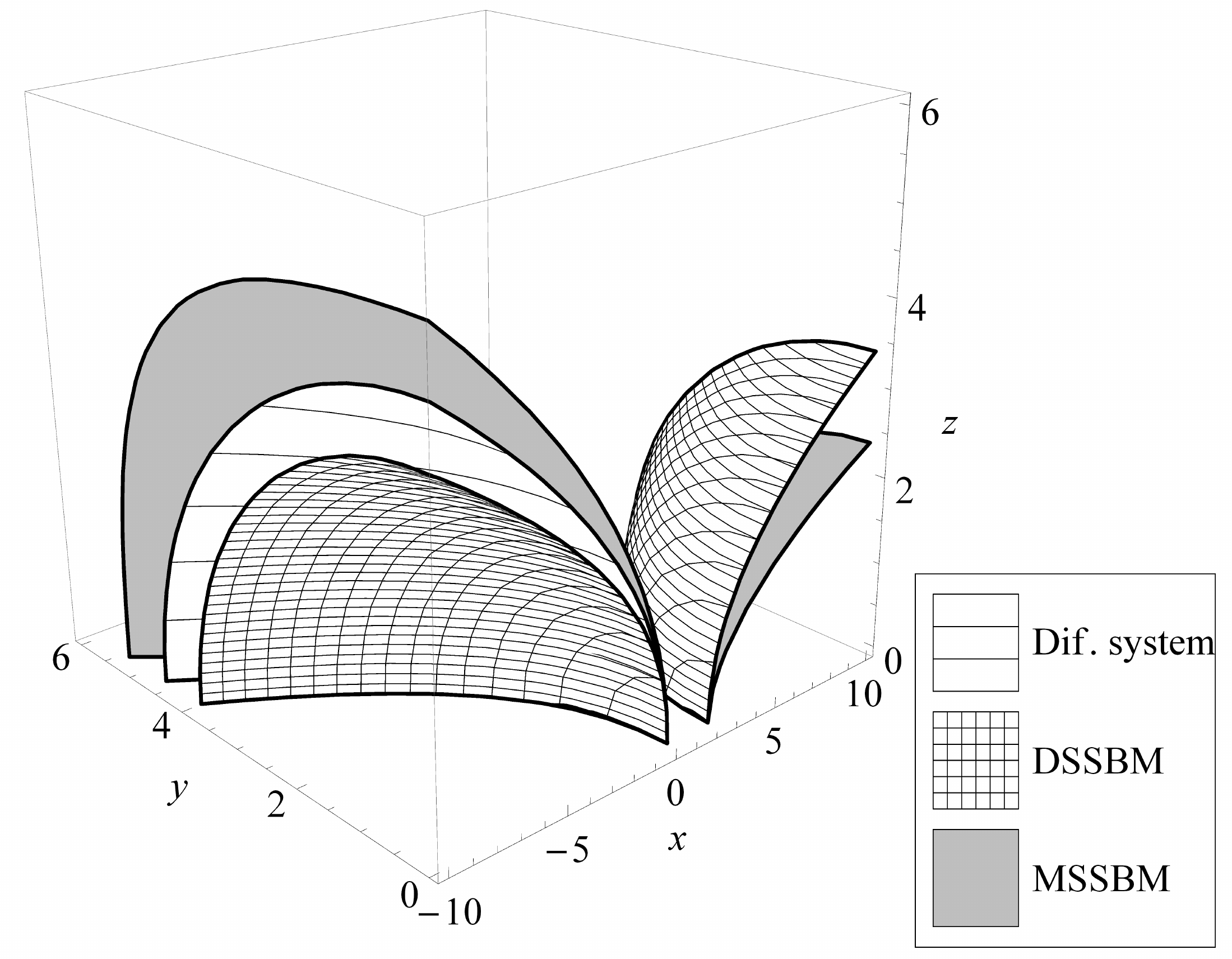}
                \caption{DSSBM and MSSBM}
        \end{subfigure}
	\qquad
		\begin{subfigure}[t]{0.45\textwidth}
                \centering
                \includegraphics[width=\textwidth]{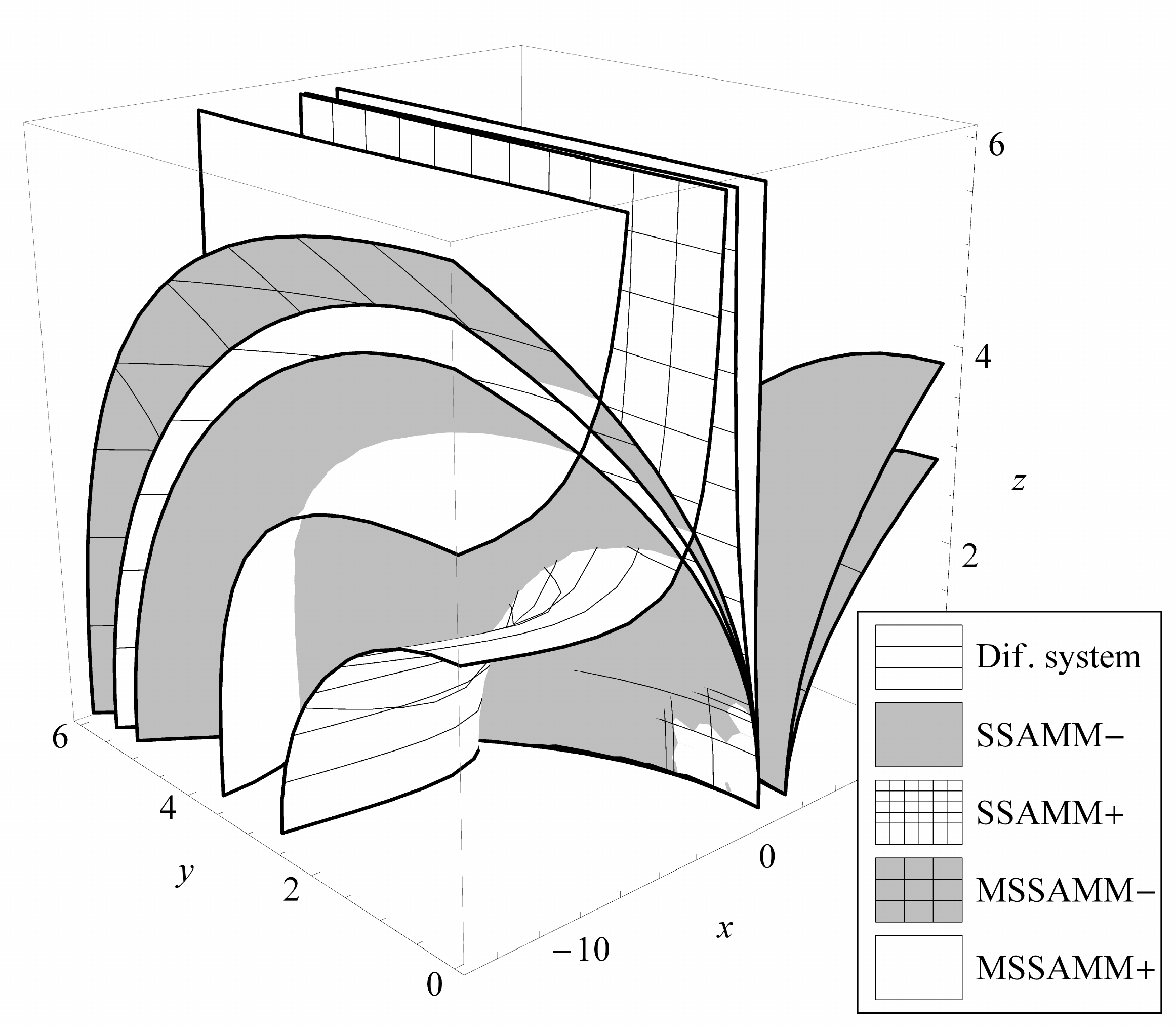}
                \caption{SSAMM}
        \end{subfigure}
	\centering
        \begin{subfigure}[t]{0.45\textwidth}
                \centering
                \includegraphics[width=\textwidth]{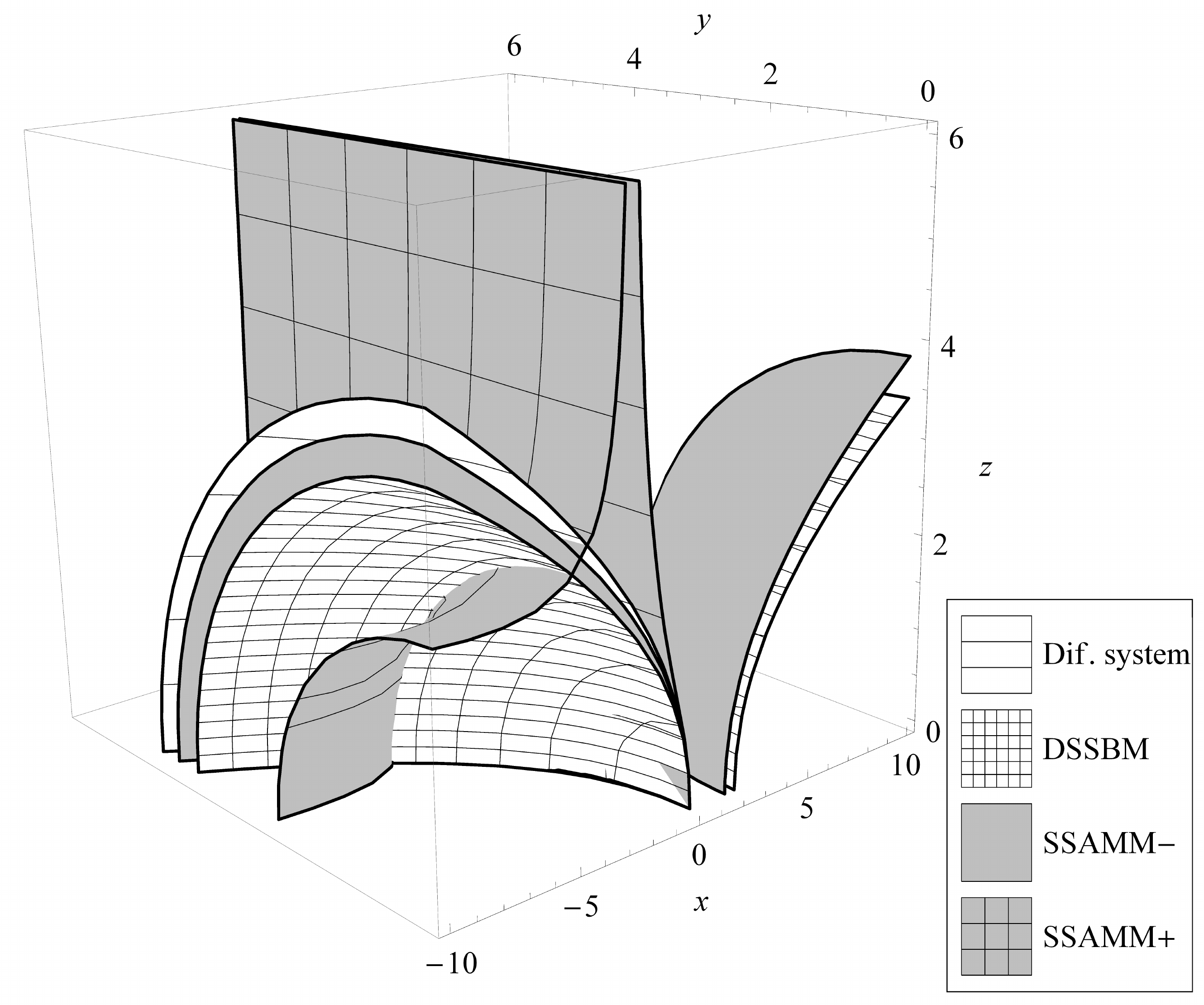}
                \caption{Split-step methods}
        \end{subfigure}
	\qquad
		\begin{subfigure}[t]{0.45\textwidth}
                \centering
                \includegraphics[width=\textwidth]{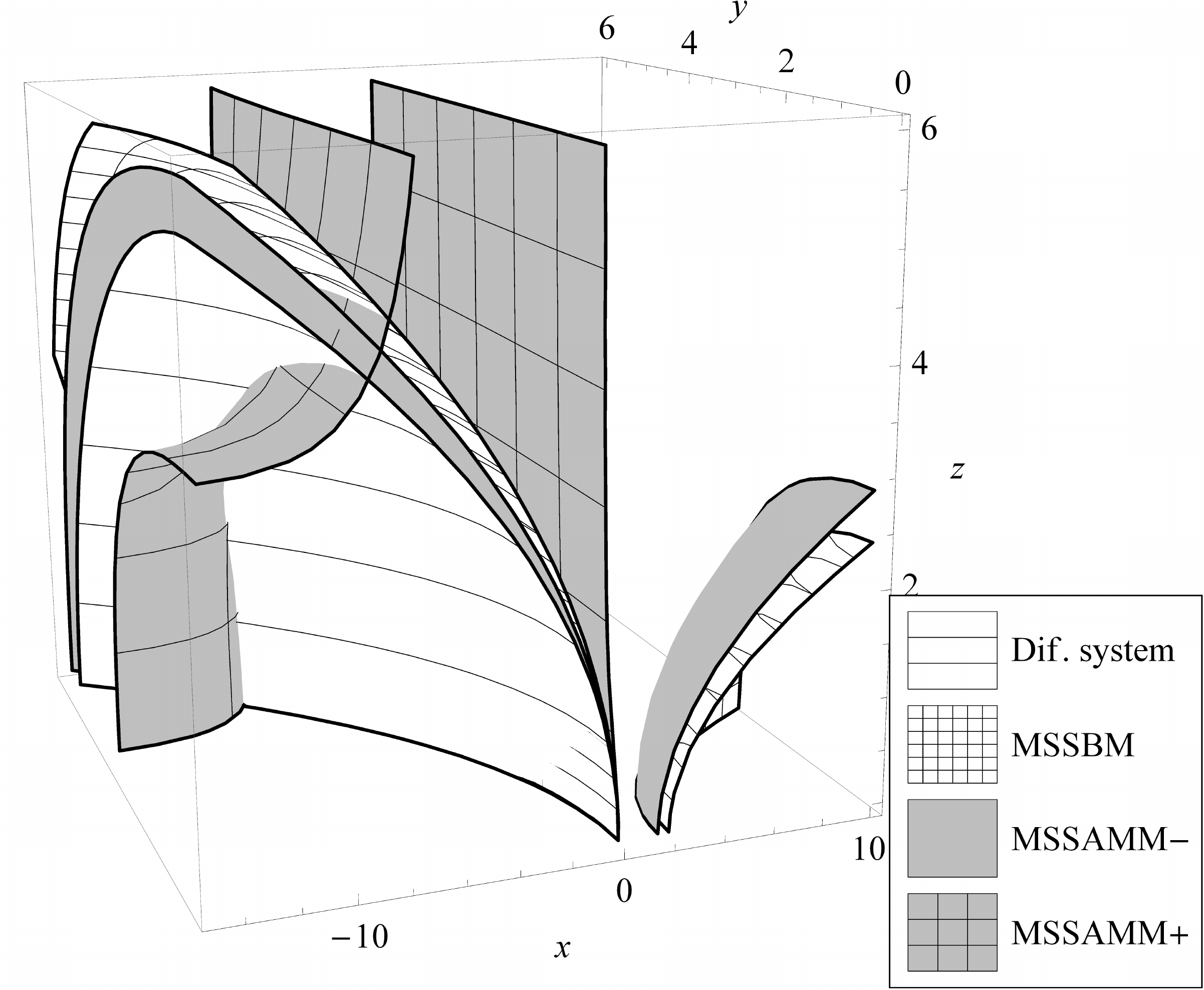}
                \caption{Modified split-step methods}
        \end{subfigure}
    \caption{Portions of the mean-square stability regions for the split-step and modified split-step methods applied to test system \eqref{eq:47}}.
	\label{fig:3}
\end{figure}

{\color{black} Figure \ref{fig:3} displays the MS-stability regions of the split-step methods \eqref{eq:4} and \eqref{eq:5} showing better stability properties of the latter when applied to the test system \eqref{eq:47}. Figure 3 illustrates the impact arising from the choice of the drift integrator. One can see that both regular and modified variants of SSAMM- method are most similar in properties to the original differential system making this method very attractive for practical calculations. On the other hand it is clear that SSAMM+ method is not the best choice for solving systems of this type.}

The following Lemma gives the stability conditions for the stochastic difference equations applied to the system with non-normal drift structure.

\begin{Lemma} For the test system \eqref{eq:48} stability condition cannot be written in the closed form and is determined as a spectral radius of the following stability matrix

\begin{align}\label{eq:56}
\left(
\begin{array}{cccc}
 a_1 & a_2 & a_2 & a_3 + a_5 \\
 0 & a_1 & -a_3 & a_2-a_4 \\
 0 & -a_3 & a_1 & a_2-a_4 \\
 a_3 & a_4 & a_4 & a_1 + a_6 \\
\end{array}
\right)
\end{align}
where
\end{Lemma}

\textit{a) for DSSBM method}
\begin{align*}
	a_1 &= \frac{z^4+2}{2 (x-1)^2},
	\qquad
	a_2 = -\frac{y^2 \left(z^4+2\right)}{2 (x-1)^3},
	\qquad
	a_3 = \frac{z^2}{(x-1)^2},
	\\
	a_4 &= -\frac{y^2 z^2}{(x-1)^3},
	\qquad
	a_5 = \frac{ (z^4+2) y^4}{2(x-1)^4},
	\qquad
	a_6 = \frac{ z^2 y^4}{(x-1)^4},
\end{align*}

\textit{b) for MSSBM method}

\begin{align*}
	a_1 &= \frac{3 z^4 - 4z^2 + 4}{\left(-z^2 - 2x + 2 \right)^2},
	\qquad
	a_2 = \frac{2 y^2 \left(3 z^4 - 4z^2 + 4\right)}{\left(-z^2 - 2x + 2\right)^3},
	\qquad
	a_3 = \frac{4 z^2}{\left(-z^2 - 2x + 2\right)^2},
	\\
	a_4 &= \frac{8 y^2 z^2}{\left(-z^2 - 2x + 2\right)^3},
	\qquad
	a_5 = \frac{4 y^4 \left(3 z^4 - 4z^2 + 4\right)}{\left(-z^2 - 2x + 2\right)^4},
	\qquad
	a_6 = \frac{16 z^2 y^4}{\left(-z^2 - 2x + 2\right)^4},
\end{align*}

\textit{c) for SSAMM method}

\begin{align*}
	a_1 &= \frac{\left(z^4+2\right) \left((4 \theta  (\theta +1)-1) x^2+8 \theta  x+4\right)^2}{2 (-2 \theta  x+x-2)^4},
	\\
	a_2 &= \frac{2 y^2 \left(z^4+2\right) ((6 \theta -1) x+2) \left((4 \theta  (\theta +1)-1) x^2+8 \theta  x+4\right)}{((2 \theta -1) x+2)^5},
	\\
	a_3 &= \frac{z^2 \left((4 \theta  (\theta +1)-1) x^2+8 \theta  x+4\right)^2}{(-2 \theta  x+x-2)^4},
	\\
	a_4 &= \frac{4 y^2 z^2 ((6 \theta -1) x+2) \left((4 \theta  (\theta +1)-1) x^2+8 \theta  x+4\right)}{((2 \theta -1) x+2)^5},
	\\
	a_5 &= \frac{8 y^4 \left(z^4+2\right) (-6 \theta  x+x-2)^2}{(-2 \theta  x+x-2)^6},
	\\
	a_6 &= \frac{16 y^4 z^2 (-6 \theta  x+x-2)^2}{(-2 \theta  x+x-2)^6},
\end{align*}
\textit{where} $x$, $y$, $z$ \textit{ are defined as in \eqref{eq:50}}.

\textit{Parameters for MSSAMM method are omitted due to the space constraints.}

\begin{figure}[t]
	\centering
        \begin{subfigure}[t]{0.45\textwidth}
                \centering
                \includegraphics[width=\textwidth]{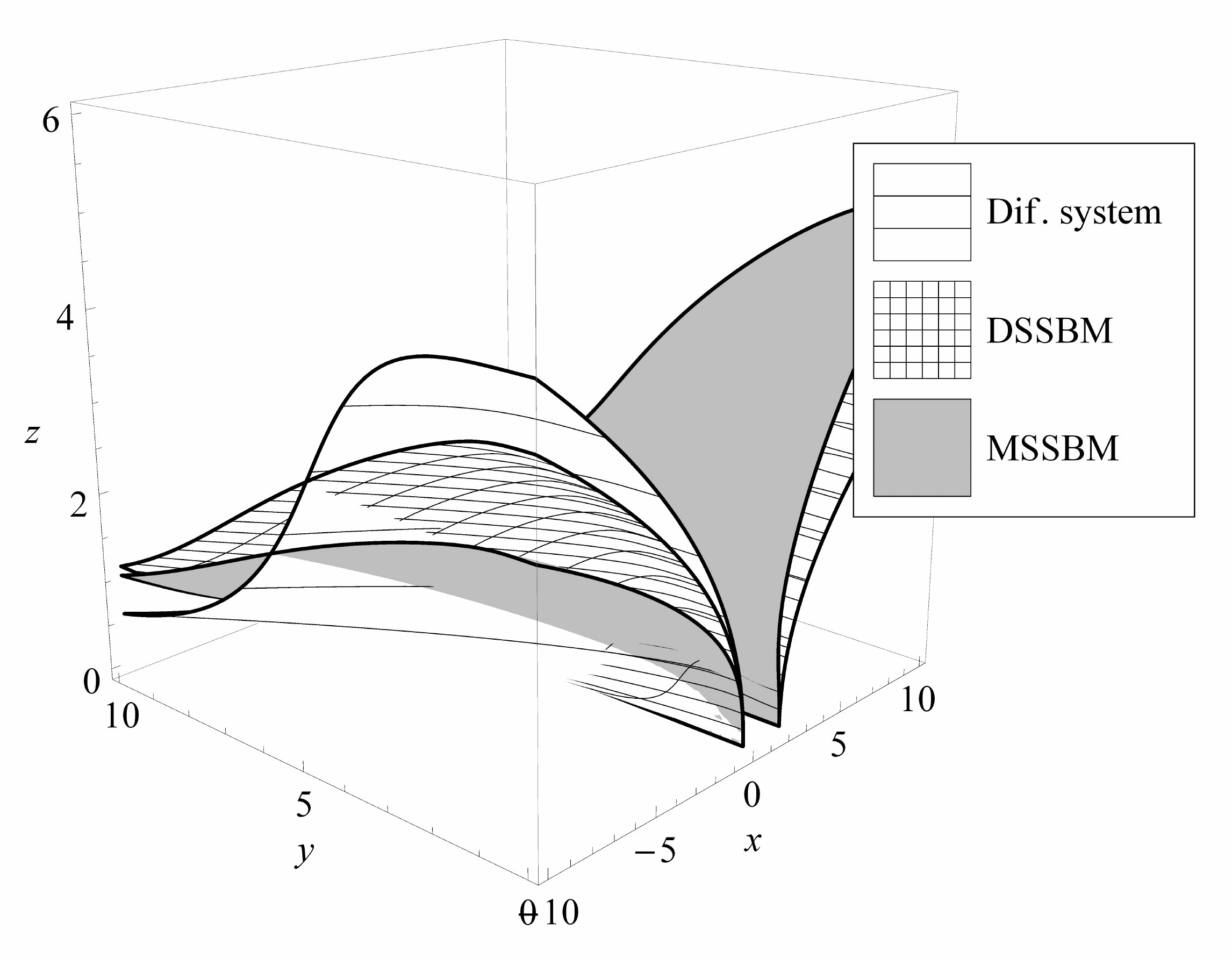}
                \caption{DSSBM and MSSBM}
        \end{subfigure}
	\qquad
		\begin{subfigure}[t]{0.45\textwidth}
                \centering
                \includegraphics[width=\textwidth]{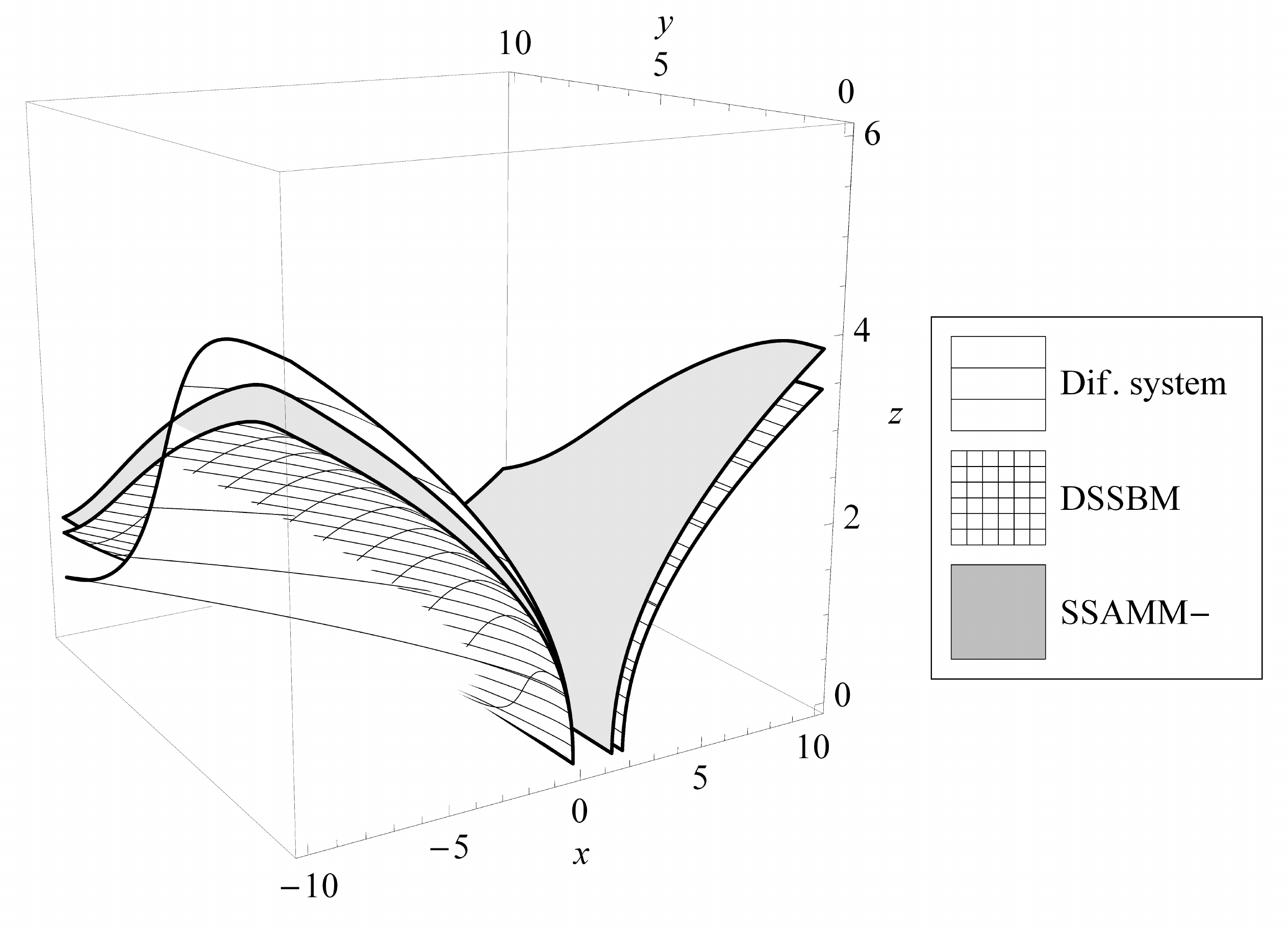}
                \caption{DSSBM and SSAMM-}
        \end{subfigure}
	\caption{Portions of the mean-square stability regions for the split-step and modified split-step methods applied to test system \eqref{eq:48}}.
	\label{fig:4}
\end{figure}

It is impossible to compute exact eigenvalues for these matrices and therefore we use numerical approximation. Stability regions computed in this way are illustrated in the Figure \ref{fig:4}. {\color{black} It is easy to see that in contrast to the system \eqref{eq:47} for the test system \eqref{eq:48} DSSBM method demonstrates better stability properties than MSSBM method. SSAMM and MSSAMM methods behave similarly and hence are not displayed. However it can be seen that SSAMM- still has the best stability properties among considered methods.

The above result is impossible to obtain using the standard stability analysis based on the scalar test equation for which the modified split-step methods always have improved stability properties. This again shows that stability of stochastic differential systems and numerical methods is highly dependent on the structures of the drift and diffusion components and the way they interact with each other.

It is easy to obtain stability conditions for the third test equation \eqref{eq:48a} using Theorem 4.2 and deterministic stability matrices \eqref{eq:43}-\eqref{eq:46} and therefore this trivial step is omitted. Stability regions for this equation are illustrated on Figure \ref{fig:5}. Figure 5a shows the influence of the number of noise sources on the stability of the SDE. By comparing stability regions of DSSBM and MSSBM methods from Figure \ref{fig:5}b one can see that for the increasing number of noise channels modified methods lose their advantages over the regular methods. Other methods behave similarly. One can also see that for systems with multidimensional noise stability of all the proposed methods deteriorates significantly. However Figures \ref{fig:5}c and \ref{fig:5}d show that for certain parameter values SSAMM+ and MSSAMM+ methods demonstrate the best stability properties.

}

\begin{figure}
	\centering
        \begin{subfigure}[t]{0.45\textwidth}
                \centering
                \includegraphics[width=\textwidth]{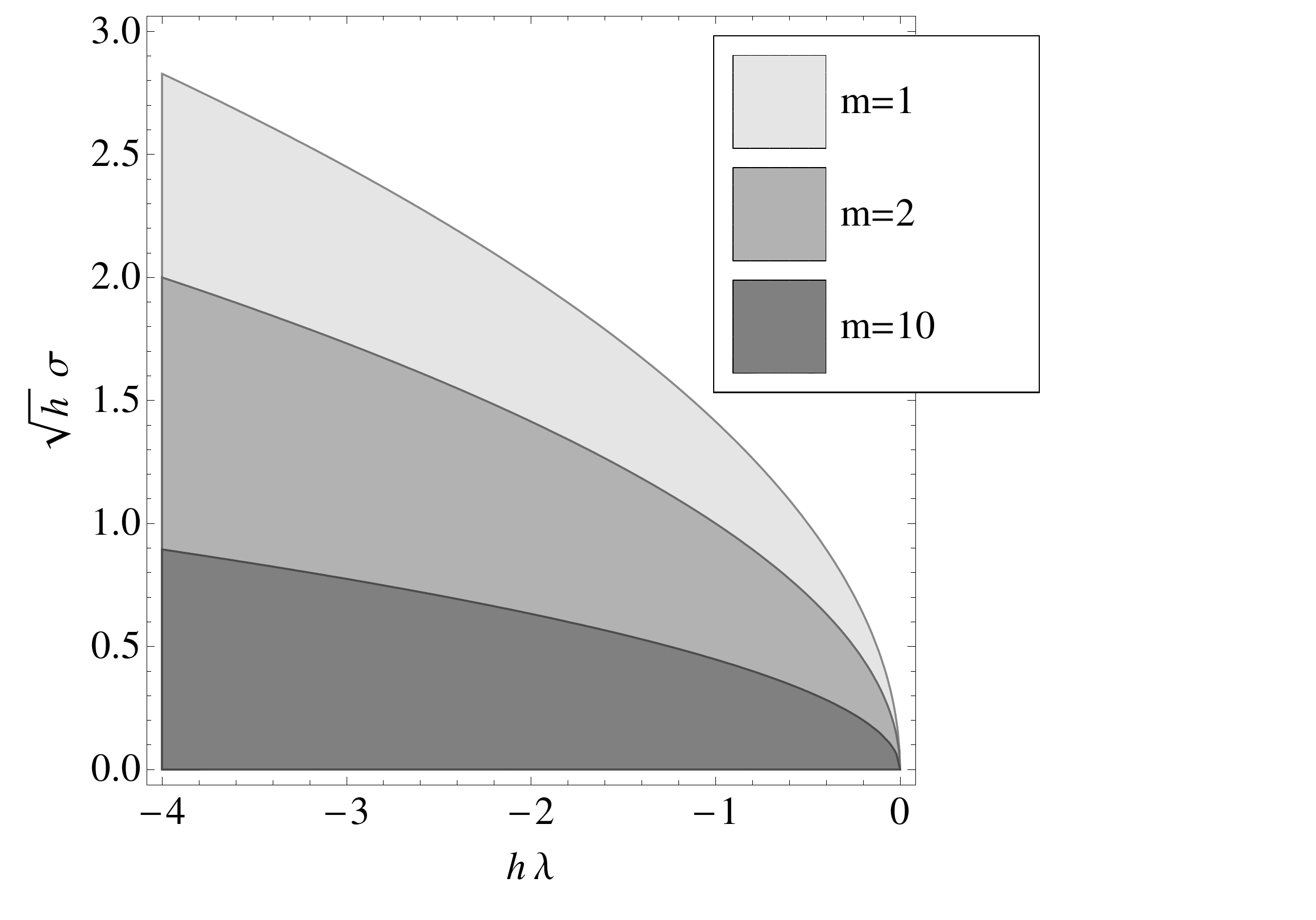}
                \caption{Stability regions of the differential system for different numbers of noise channels}
        \end{subfigure}
	\qquad
		\begin{subfigure}[t]{0.45\textwidth}
                \centering
                \includegraphics[width=\textwidth]{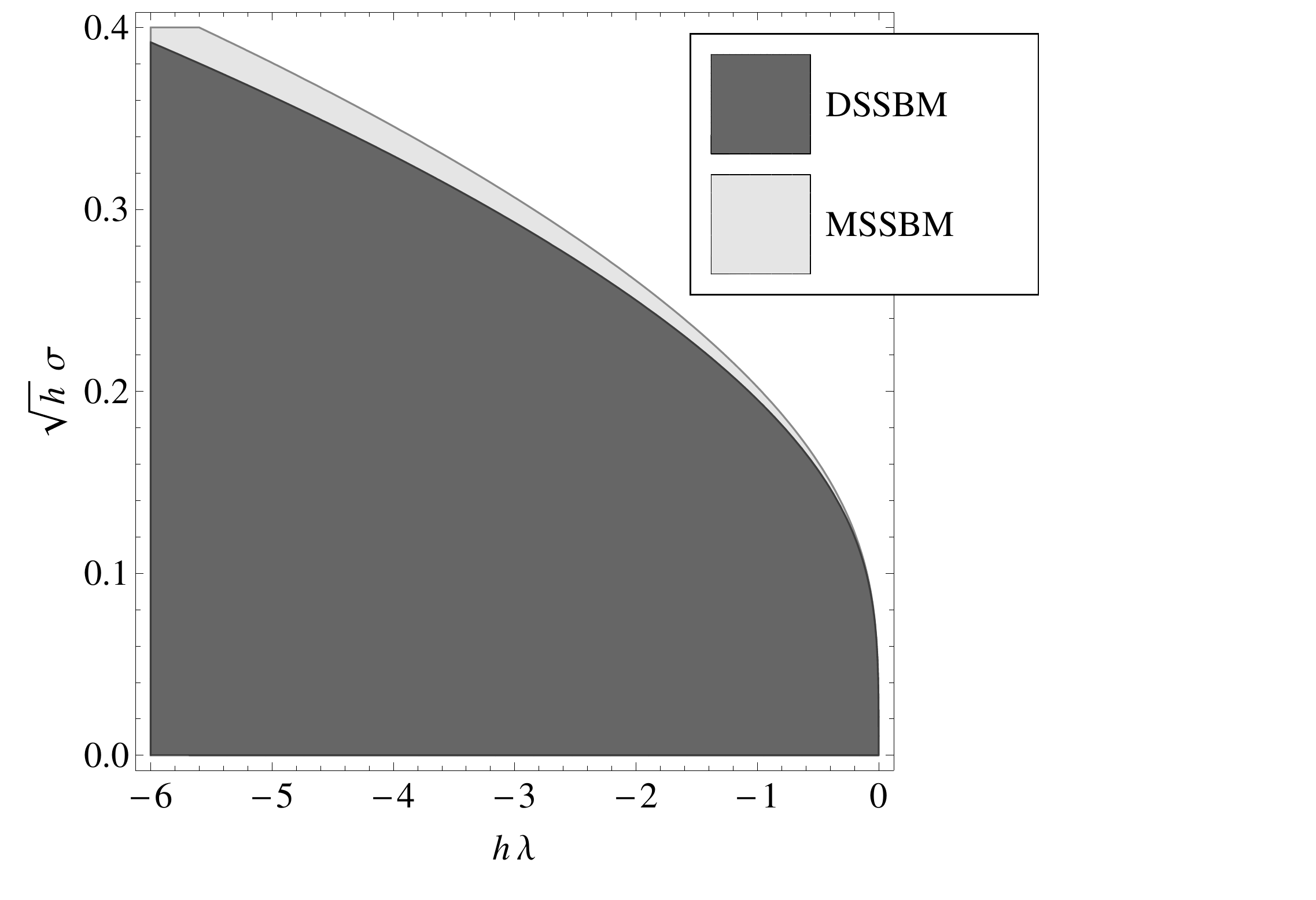}
                \caption{Stability regions of DSSBM and MSSBM methods for the system with 10 noise channels}
        \end{subfigure}
	\centering
        \begin{subfigure}[t]{0.45\textwidth}
                \centering
                \includegraphics[width=\textwidth]{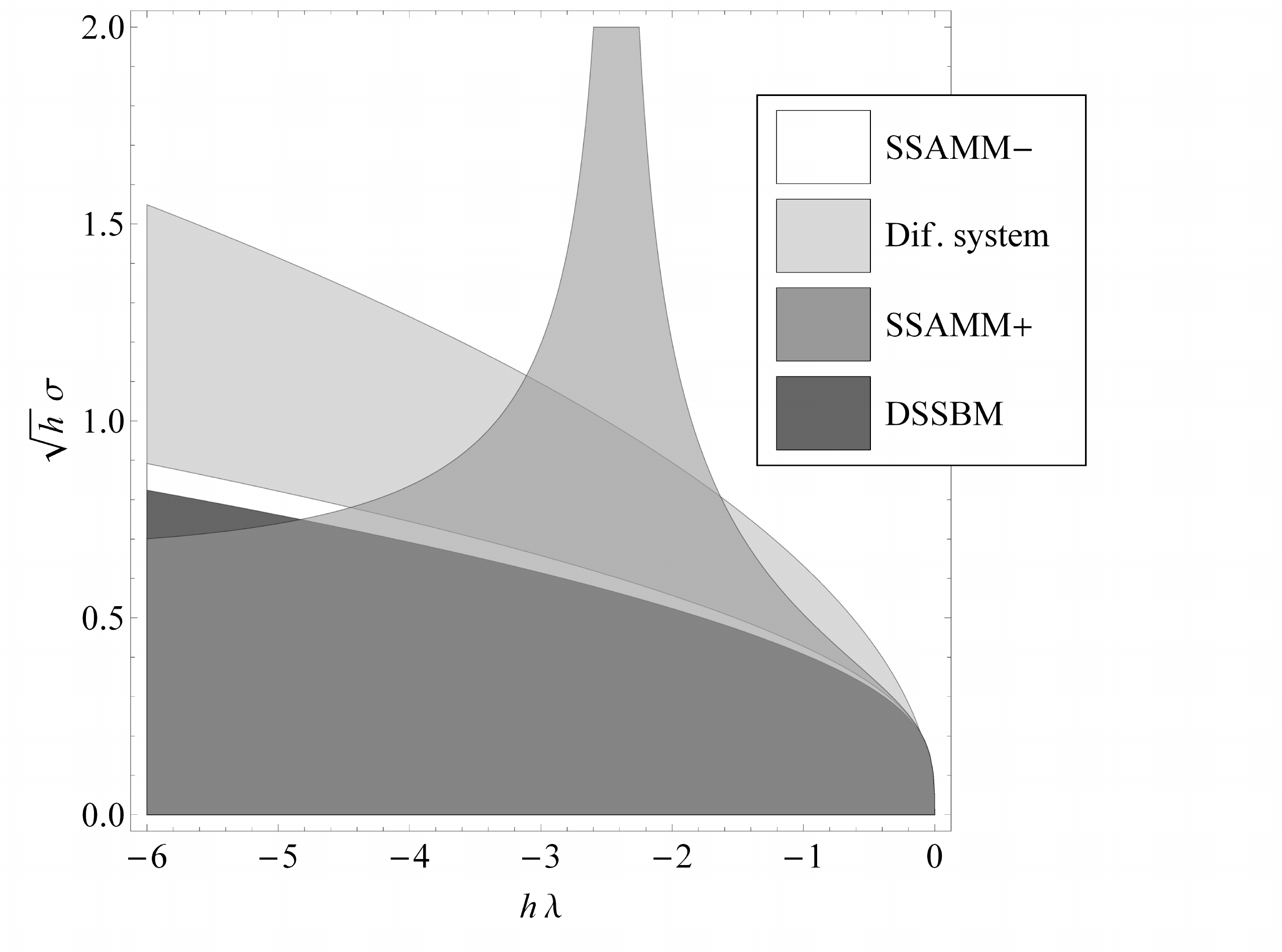}
                \caption{Regular methods}
        \end{subfigure}
	\qquad
		\begin{subfigure}[t]{0.45\textwidth}
                \centering
                \includegraphics[width=\textwidth]{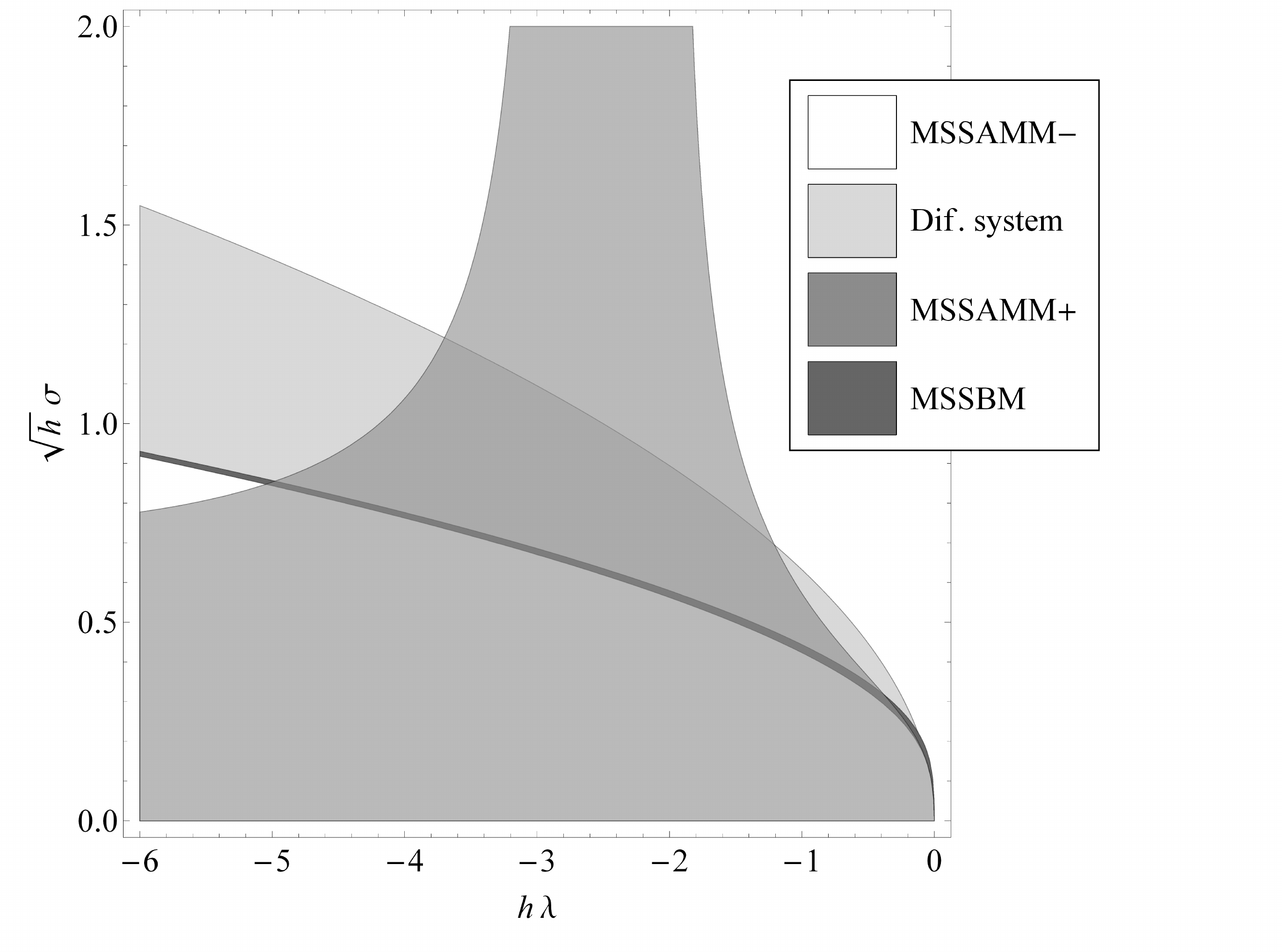}
                \caption{Modified methods}
        \end{subfigure}
	\caption{Mean-square stability regions for the test system \eqref{eq:48a}}.
	\label{fig:5}
\end{figure}

\section{Numerical results}
In this section we give four numerical examples \newblue{ to confirm theoretical results from previous sections and} illustrate effectiveness of the proposed methods in solving stiff stochastic multi-channel systems.

\textbf{Example 1.} To \newblue{confirm} the strong order of convergence of the proposed methods we consider the following multi-dimensional linear stochastic system \cite{Roesler2010}

\begin{align}\label{eq:59}
	dX = A X dt + \sum_{k=1}^m B^k X dW^k, \qquad X_0 = (1,...,1)^T \in \mathbb{R}^d
\end{align}
where $A_{i,j}=\frac{1}{20}$ if $i \neq j$, $A_{ii}=-\frac{3}{2}$ and $B_{i,j}^k=\frac{1}{100}$ if $i \neq j$, $B_{ii}^k=\frac{1}{5}$.

Exact solution of the above system is known and given by the expression
$$
	X = X_0 \exp \left( \left( A-\frac{1}{2} \sum_{k=1}^m \left(B^k\right)^2 \right)t + \sum_{k=1}^m B^k W^k \right).
$$

We computed the numerical solution for the case $d=m=5$ for time steps $h=2^{-1}...2^{-10}$. Results in the Table \ref{tab:1} and in the Figure \ref{fig:6} show the strong order of convergence 1 for all methods which is in accordance to theoretical results. All schemes have similar level of accuracy.

\begin{figure}[h]
	\centering
                \includegraphics[width=0.5\textwidth]{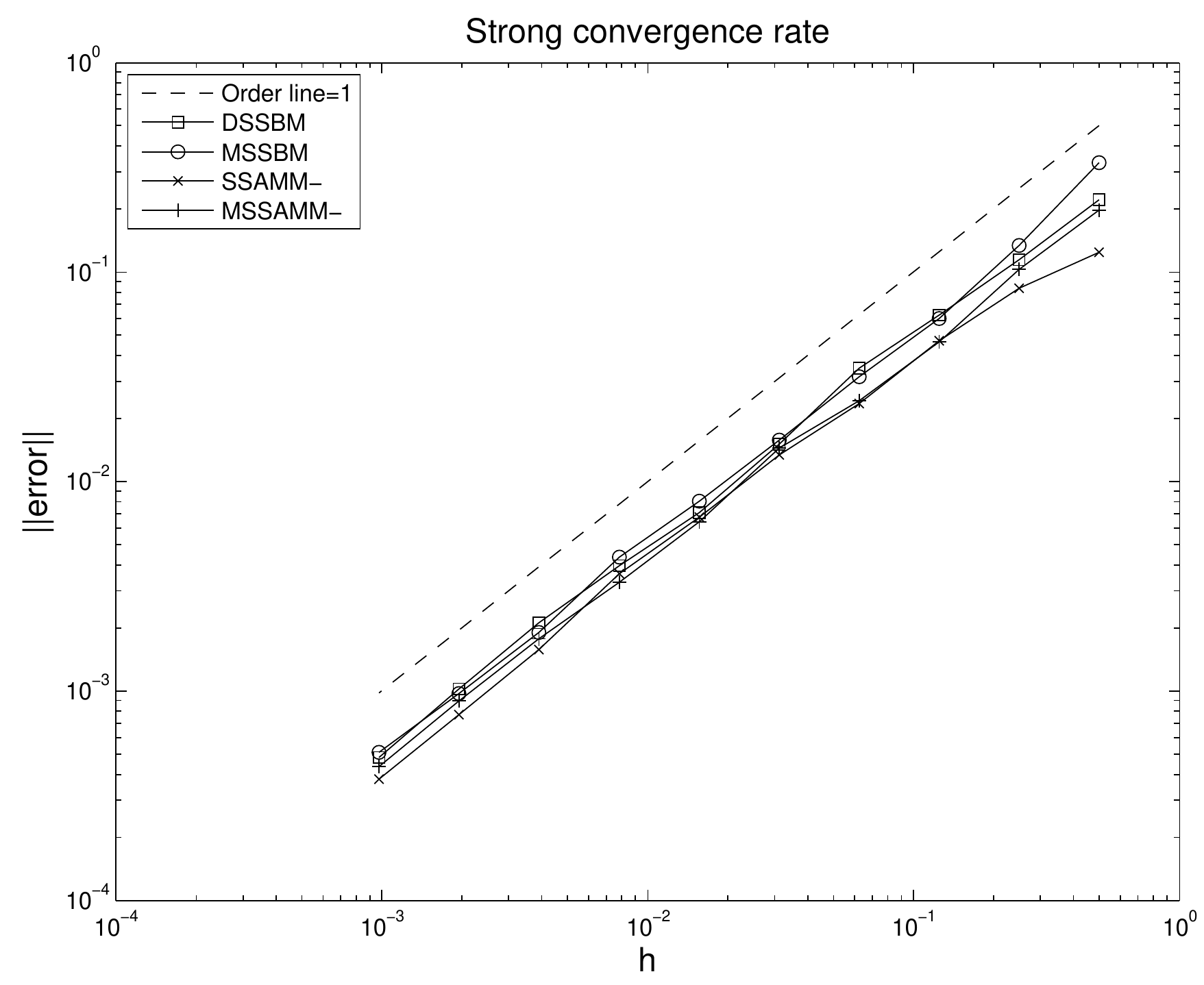}
                \caption{Strong errors for the equation \eqref{eq:59}}
	\label{fig:6}
\end{figure}

\begin{table}[h]
	\begin{center}
	\begin{tabular}{  c  c  c  c  c  c  }
\hline
\hline
   h &      DSSBM      &   MSSBM      &  SSAMM-     &   MSSAMM-   \\
\hline
$2^{- 1}$&$   2.207\times 10^{-1}$  &$   3.335\times 10^{-1}$  &$   1.243\times 10^{-1}$  &$   1.977\times 10^{-1}$  \\
$2^{- 2}$&$   1.148\times 10^{-1}$  &$   1.340\times 10^{-1}$  &$   8.368\times 10^{-2}$  &$   1.029\times 10^{-1}$  \\
$2^{- 3}$&$   6.241\times 10^{-2}$  &$   5.988\times 10^{-2}$  &$   4.683\times 10^{-2}$  &$   4.641\times 10^{-2}$  \\
$2^{- 4}$&$   3.480\times 10^{-2}$  &$   3.159\times 10^{-2}$  &$   2.361\times 10^{-2}$  &$   2.425\times 10^{-2}$  \\
$2^{- 5}$&$   1.511\times 10^{-2}$  &$   1.572\times 10^{-2}$  &$   1.339\times 10^{-2}$  &$   1.439\times 10^{-2}$  \\
$2^{- 6}$&$   7.090\times 10^{-3}$  &$   8.037\times 10^{-3}$  &$   6.760\times 10^{-3}$  &$   6.411\times 10^{-3}$  \\
$2^{- 7}$&$   3.960\times 10^{-3}$  &$   4.354\times 10^{-3}$  &$   3.637\times 10^{-3}$  &$   3.294\times 10^{-3}$  \\
$2^{- 8}$&$   2.112\times 10^{-3}$  &$   1.901\times 10^{-3}$  &$   1.575\times 10^{-3}$  &$   1.774\times 10^{-3}$  \\
\hline
\hline
	\end{tabular}
	\end{center}
	\caption{Strong errors for the equation \eqref{eq:59}}
	\label{tab:1}
\end{table}

{\color{black} To study the impact on the stability of the numerical methods arising from the dimension of the noise, we numerically computed spectral radius of the stability matrices for different number of noise channels. We note that the differential system is stable for the considered parameter values. Results from Tables \ref{tab:2}-\ref{tab:4} clearly illustrate that even for simple linear systems with commutative noise the stability condition can become very restrictive for the increasing number of noise channels. We note that the best result was obtained by the SSAMM+ and MSSAMM+ schemes which is in accordance to the stability analysis performed in the previous section. 
}

\begin{table}
	\begin{center}
	\begin{tabular}{  c  c  c  c  c  c  c  c }
\hline
\hline
   h &        DSSBM      &    SSAMM+     &   SSAMM-  &  MSSBM   &  MSSAMM+     &   MSSAMM-   \\
\hline
1.000 & 1.10  unst. & 0.24  \textbf{stab.}  & 4.48  unst. & 0.97   \textbf{stab.}  & 0.19   \textbf{stab.}  & 4.18  unst.  \\ 
  0.900 & 1.05  unst. & 0.04  \textbf{stab.}  & 3.77  unst. & 0.94   \textbf{stab.}  & 0.04   \textbf{stab.}  & 3.55  unst.  \\ 
  0.800 & 1.00  unst. & 0.07  \textbf{stab.}  & 3.14  unst. & 0.91   \textbf{stab.}  & 0.07   \textbf{stab.}  & 2.99  unst.  \\ 
  0.700 & 0.95  \textbf{stab.}  & 0.13  \textbf{stab.}  & 2.58  unst. & 0.88   \textbf{stab.}  & 0.12   \textbf{stab.}  & 2.49  unst.  \\ 
  0.600 & 0.90  \textbf{stab.}  & 0.19  \textbf{stab.}  & 2.11  unst. & 0.85   \textbf{stab.}  & 0.18   \textbf{stab.}  & 2.05  unst.  \\ 
  0.500 & 0.86  \textbf{stab.}  & 0.31  \textbf{stab.}  & 1.71  unst. & 0.82   \textbf{stab.}  & 0.28   \textbf{stab.}  & 1.69  unst.  \\ 
  0.400 & 0.82  \textbf{stab.}  & 0.45  \textbf{stab.}  & 1.39  unst. & 0.80   \textbf{stab.}  & 0.42   \textbf{stab.}  & 1.38  unst.  \\ 
  0.300 & 0.80  \textbf{stab.}  & 0.58  \textbf{stab.}  & 1.14  unst. & 0.79   \textbf{stab.}  & 0.56   \textbf{stab.}  & 1.15  unst.  \\ 
  0.200 & 0.80  \textbf{stab.}  & 0.70  \textbf{stab.}  & 0.98  \textbf{stab.}  & 0.80   \textbf{stab.}  & 0.69   \textbf{stab.}  & 0.99   \textbf{stab.}   \\ 
  0.100 & 0.86  \textbf{stab.}  & 0.83  \textbf{stab.}  & 0.91  \textbf{stab.}  & 0.86   \textbf{stab.}  & 0.83   \textbf{stab.}  & 0.92   \textbf{stab.}   \\ 
\hline
\hline
	\end{tabular}
	\end{center}
	\caption{Values of the spectral radius of the stability matrices for DSSBM, MSSBM, SSAMM and MSSAMM methods for the system \eqref{eq:59} with $d=5$, $m=9$.}
	\label{tab:2}
\end{table}

\begin{table}
	\begin{center}
	\begin{tabular}{  c  c  c  c  c  c  c  c }
\hline
\hline
   h &      DSSBM      &    SSAMM+     &   SSAMM-  &  MSSBM   &  MSSAMM+     &   MSSAMM-   \\
\hline
 1.000 & 2.21  unst. & 0.48  \textbf{stab.}  & 8.98  unst. & 1.80  unst. & 0.34   \textbf{stab.}  & 7.85  unst.  \\ 
  0.900 & 2.06  unst. & 0.09  \textbf{stab.}  & 7.38  unst. & 1.71  unst. & 0.06   \textbf{stab.}  & 6.52  unst.  \\ 
  0.800 & 1.91  unst. & 0.12  \textbf{stab.}  & 5.96  unst. & 1.61  unst. & 0.10   \textbf{stab.}  & 5.33  unst.  \\ 
  0.700 & 1.75  unst. & 0.20  \textbf{stab.}  & 4.73  unst. & 1.50  unst. & 0.17   \textbf{stab.}  & 4.29  unst.  \\ 
  0.600 & 1.58  unst. & 0.30  \textbf{stab.}  & 3.68  unst. & 1.39  unst. & 0.26   \textbf{stab.}  & 3.39  unst.  \\ 
  0.500 & 1.41  unst. & 0.51  \textbf{stab.}  & 2.80  unst. & 1.27  unst. & 0.43   \textbf{stab.}  & 2.62  unst.  \\ 
  0.400 & 1.24  unst. & 0.68  \textbf{stab.}  & 2.09  unst. & 1.15  unst. & 0.60   \textbf{stab.}  & 2.00  unst.  \\ 
  0.300 & 1.08  unst. & 0.79  \textbf{stab.}  & 1.55  unst. & 1.04  unst. & 0.73   \textbf{stab.}  & 1.51  unst.  \\ 
  0.200 & 0.96  \textbf{stab.}  & 0.84  \textbf{stab.}  & 1.17  unst. & 0.95   \textbf{stab.}  & 0.81   \textbf{stab.}  & 1.17  unst.  \\ 
  0.100 & 0.91  \textbf{stab.}  & 0.88  \textbf{stab.}  & 0.97  \textbf{stab.}  & 0.91   \textbf{stab.}  & 0.88   \textbf{stab.}  & 0.97   \textbf{stab.}   \\ 
\hline
\hline
	\end{tabular}
	\end{center}
	\caption{Values of the spectral radius of the stability matrices for DSSBM, MSSBM, SSAMM and MSSAMM methods for the system \eqref{eq:59} with $d=5$, $m=11$.}
	\label{tab:3}
\end{table}

\begin{table}
	\begin{center}
	\begin{tabular}{  c  c  c  c  c  c  c  c }
\hline
\hline
   h &         DSSBM      &    SSAMM+     &   SSAMM-  &  MSSBM   &  MSSAMM+     &   MSSAMM-   \\
\hline
  1.000 & 4.15  unst. & 0.90  \textbf{stab.}  & 16.86  unst. & 3.18  unst. & 0.58   \textbf{stab.}  & 14.00  unst.  \\ 
  0.900 & 3.83  unst. & 0.16  \textbf{stab.}  & 13.70  unst. & 2.98  unst. & 0.11   \textbf{stab.}  & 11.49  unst.  \\ 
  0.800 & 3.49  unst. & 0.20  \textbf{stab.}  & 10.91  unst. & 2.77  unst. & 0.16   \textbf{stab.}  & 9.26  unst.  \\ 
  0.700 & 3.13  unst. & 0.32  \textbf{stab.}  & 8.49  unst. & 2.53  unst. & 0.27   \textbf{stab.}  & 7.30  unst.  \\ 
  0.600 & 2.75  unst. & 0.53  \textbf{stab.}  & 6.42  unst. & 2.28  unst. & 0.40   \textbf{stab.}  & 5.62  unst.  \\ 
  0.500 & 2.36  unst. & 0.86  \textbf{stab.}  & 4.70  unst. & 2.02  unst. & 0.68   \textbf{stab.}  & 4.20  unst.  \\ 
  0.400 & 1.96  unst. & 1.08  unst. & 3.31  unst. & 1.73  unst. & 0.90   \textbf{stab.}  & 3.03  unst.  \\ 
  0.300 & 1.58  unst. & 1.15  unst. & 2.25  unst. & 1.45  unst. & 1.02  unst. & 2.12  unst.  \\ 
  0.200 & 1.23  unst. & 1.08  unst. & 1.50  unst. & 1.18  unst. & 1.01  unst. & 1.46  unst.  \\ 
  0.100 & 1.00  unst. & 0.97  \textbf{stab.}  & 1.06  unst. & 0.99   \textbf{stab.}  & 0.95   \textbf{stab.}  & 1.06  unst.  \\ 
\hline
\hline
	\end{tabular}
	\end{center}
	\caption{Values of the spectral radius of the stability matrices for DSSBM, MSSBM, SSAMM and MSSAMM methods for the system \eqref{eq:59} with $d=5$, $m=13$.}
	\label{tab:4}
\end{table}

\vspace{5mm}
\newblue{
\textbf{Example 2. Stiff problem.}
Second example is a two-dimensional stiff stochastic system \cite{Milstein1998}

\begin{align}\label{eq:ex3a_1}
	d
	X
	&=
	\beta
	\begin{pmatrix}
		0  & 1   \\[0.5em]
		-1 & 0
	\end{pmatrix}
	X
	+
	\frac{\sigma}{2}
	\begin{pmatrix}
		1  & 1   \\[0.5em]
		1 & 1
	\end{pmatrix}
	X
	dW_1
	+
	\frac{\rho}{2}
	\begin{pmatrix}
		1  & -1   \\[0.5em]
		-1 & 1
	\end{pmatrix}
	X
	dW_2.
\end{align}

This system was simulated with  $\beta = 5$, $\sigma=4$ and $\rho=0.5$ on the interval $T=[0,20]$ with initial position at $(X_1(0), X_2(0)) = (1, 0)$. As it can be seen from Figure \ref{fig:7}, all methods have similar dynamics. Their trajectories stay close to the origin which replicates the exact solution. However the methods respond differently to the transients which appear in the long time simulation. These transients are due to the stiffness of the stochastic system in \eqref{eq:ex3a_1}. We see that DSSBM and SSAMM methods generate stable solutions while the Milstein method blows up in the vicinity of the transient which is an undesirable behavior for asymptotically stable system.

It is also interesting to compare the behavior of the SSAMM- and the SSCTM ($\theta=0.5$) methods since their deterministic parts are geometric integrators of the corresponding rotating deterministic system. Figures \ref{fig:7} e-h show that for the system in \eqref{eq:ex3a_1} SSAMM- method is more stable than SSCTM ($\theta=0.5$) method.

}

\begin{figure}[p]
	\centering
        \begin{subfigure}[t]{0.45\textwidth}
                \centering
                \includegraphics[width=\textwidth]{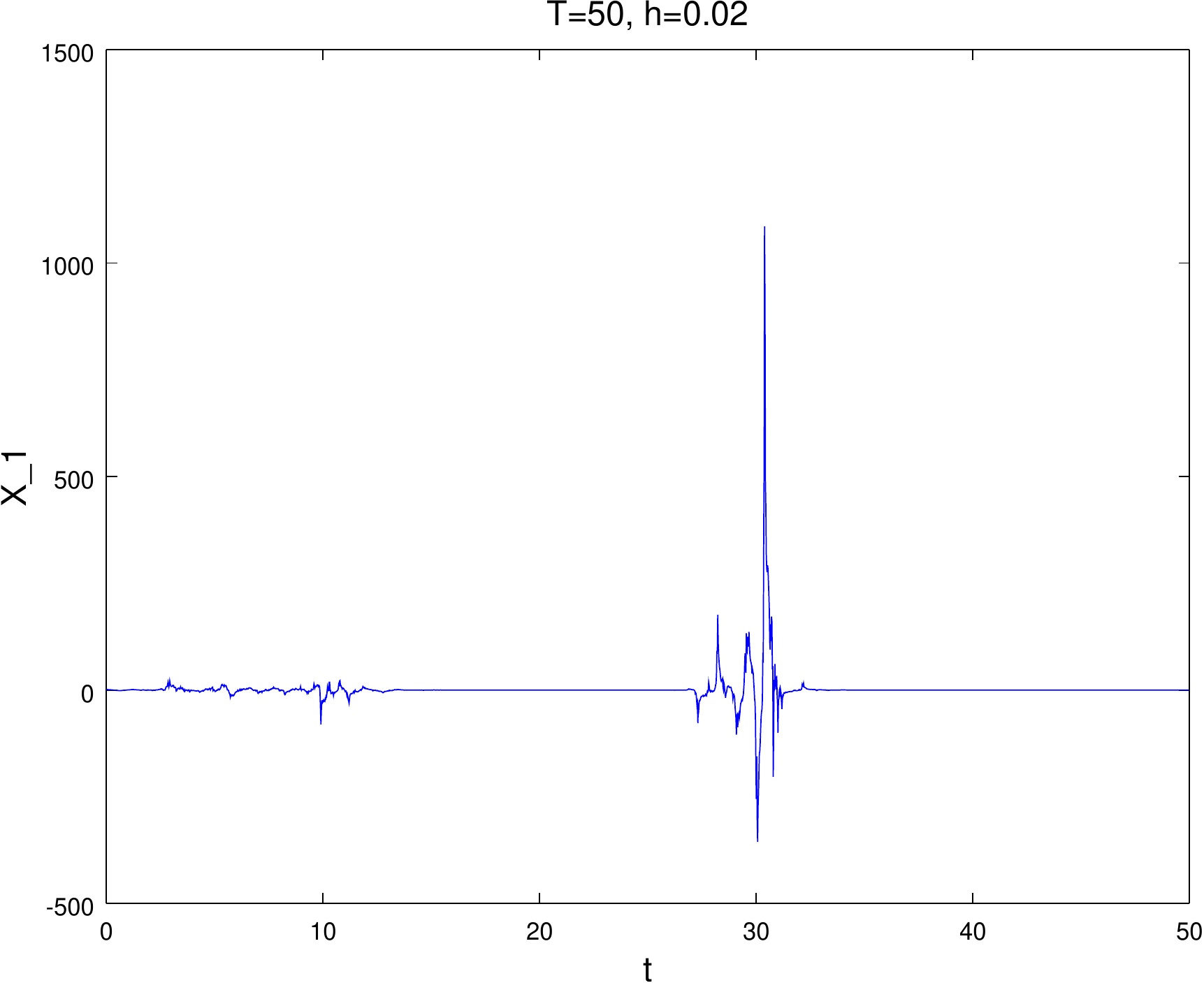}
                \caption{Milstein method}
        \end{subfigure}
	\qquad
		\begin{subfigure}[t]{0.45\textwidth}
                \centering
                \includegraphics[width=\textwidth]{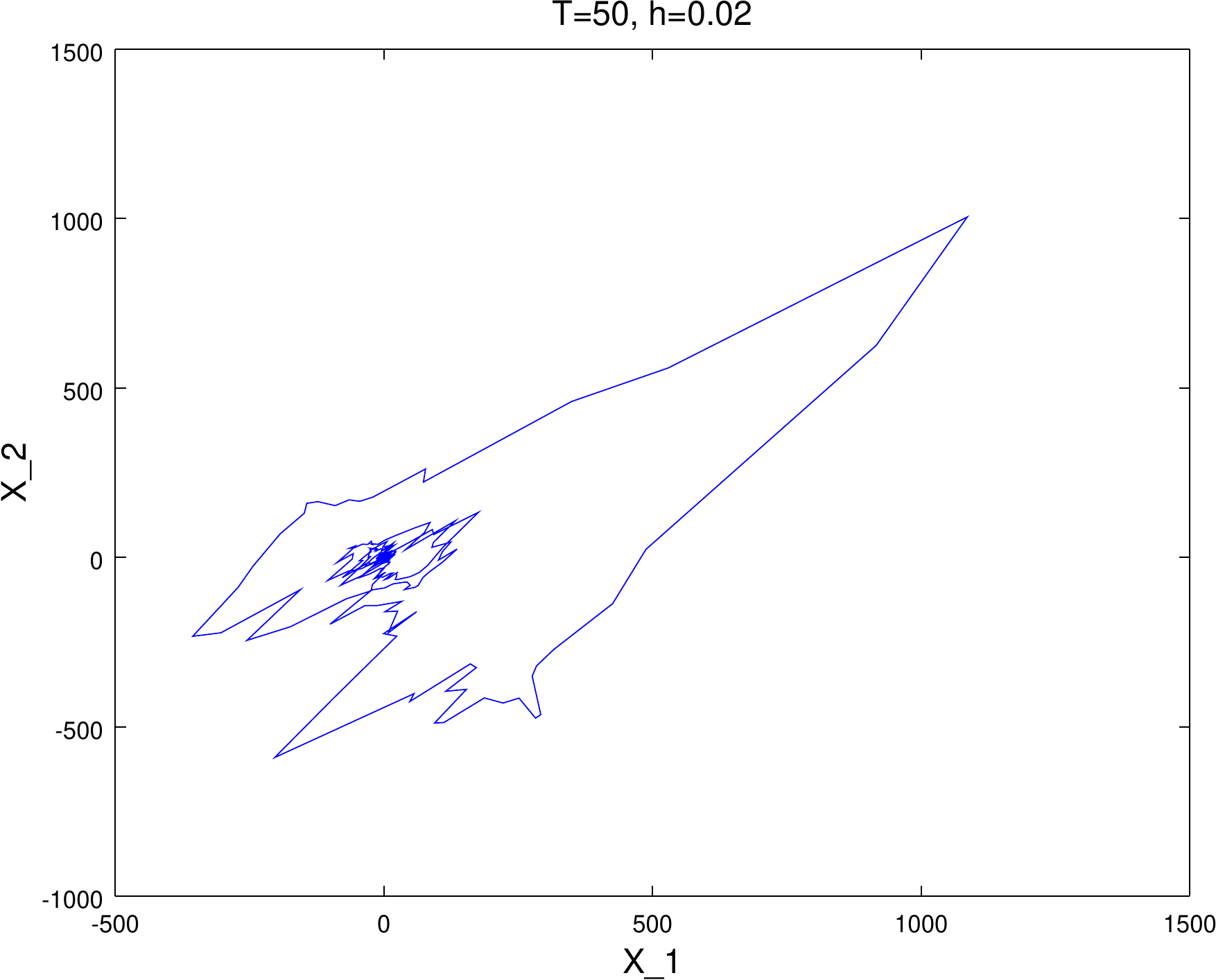}
                \caption{Milstein method}
        \end{subfigure}
        \begin{subfigure}[t]{0.45\textwidth}
                \centering
                \includegraphics[width=\textwidth]{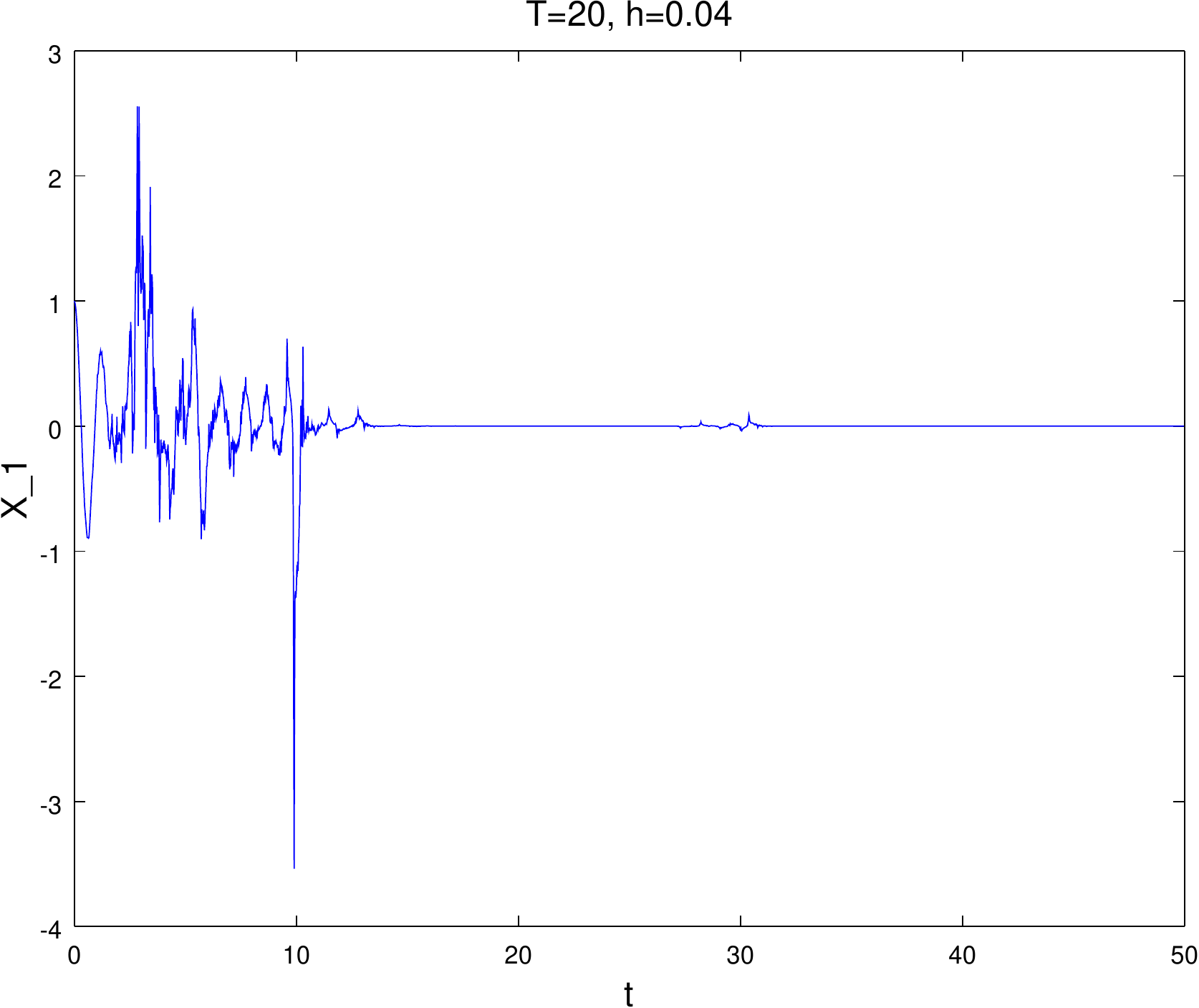}
                \caption{ DSSBM method}
        \end{subfigure}
	\qquad
		\begin{subfigure}[t]{0.45\textwidth}
                \centering
                \includegraphics[width=\textwidth]{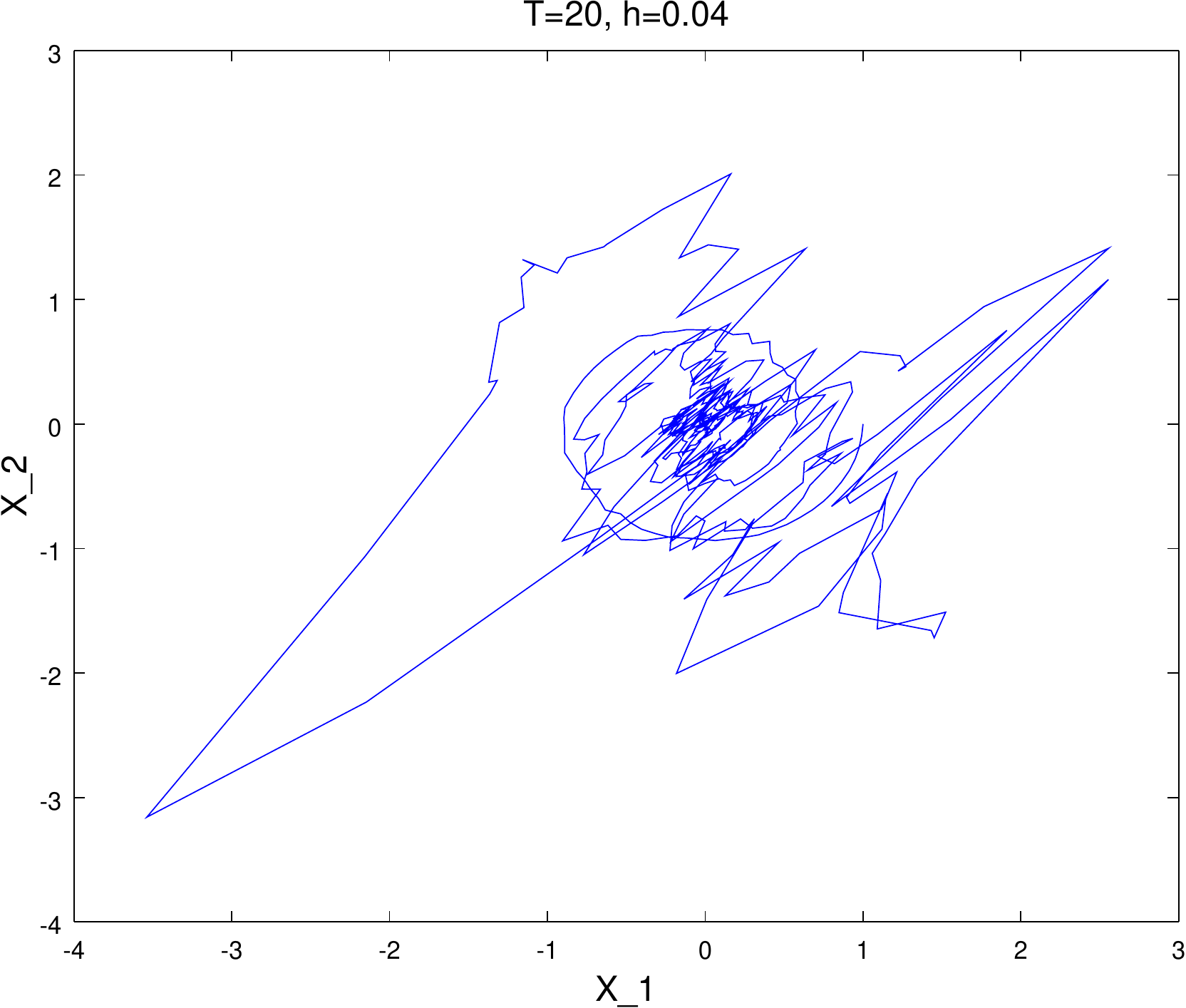}
                \caption{ DSSBM method}
        \end{subfigure}
        \begin{subfigure}[t]{0.45\textwidth}
                \centering
                \includegraphics[width=\textwidth]{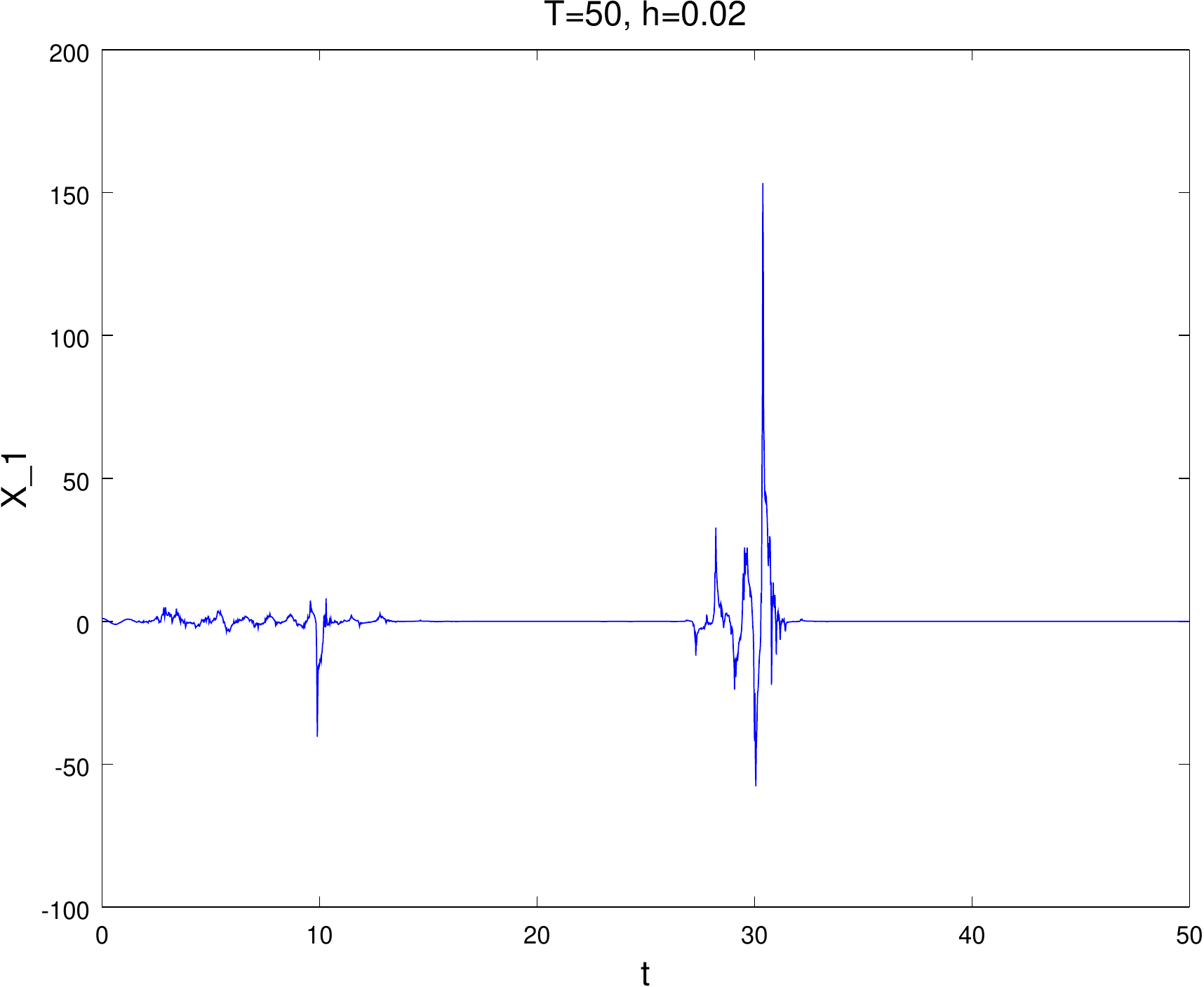}
                \caption{SSCTM method, $\theta=0.5$}
        \end{subfigure}
	\qquad
		\begin{subfigure}[t]{0.45\textwidth}
                \centering
                \includegraphics[width=\textwidth]{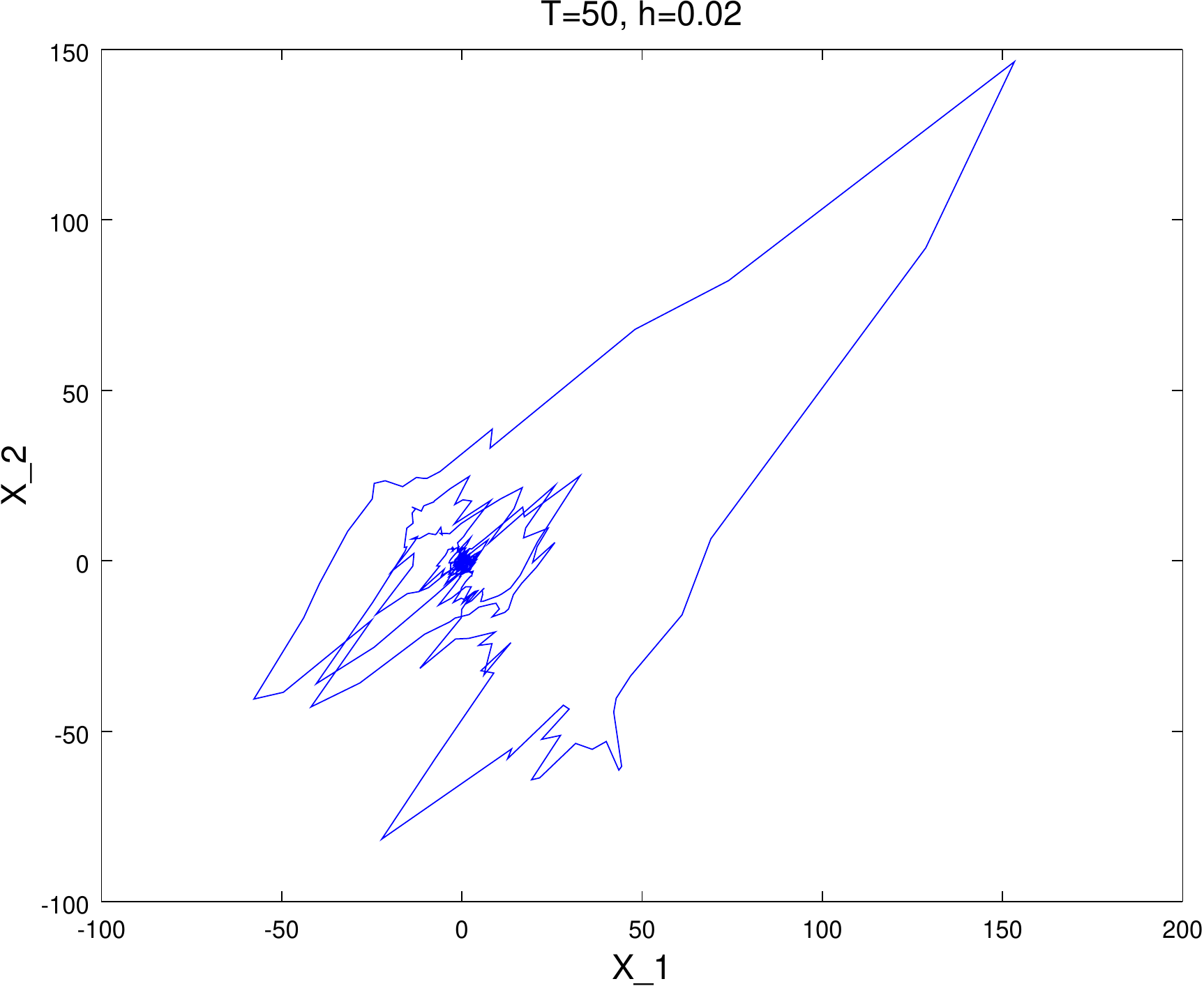}
                \caption{SSCTM method, $\theta=0.5$}
        \end{subfigure}
        \begin{subfigure}[t]{0.45\textwidth}
                \centering
                \includegraphics[width=\textwidth]{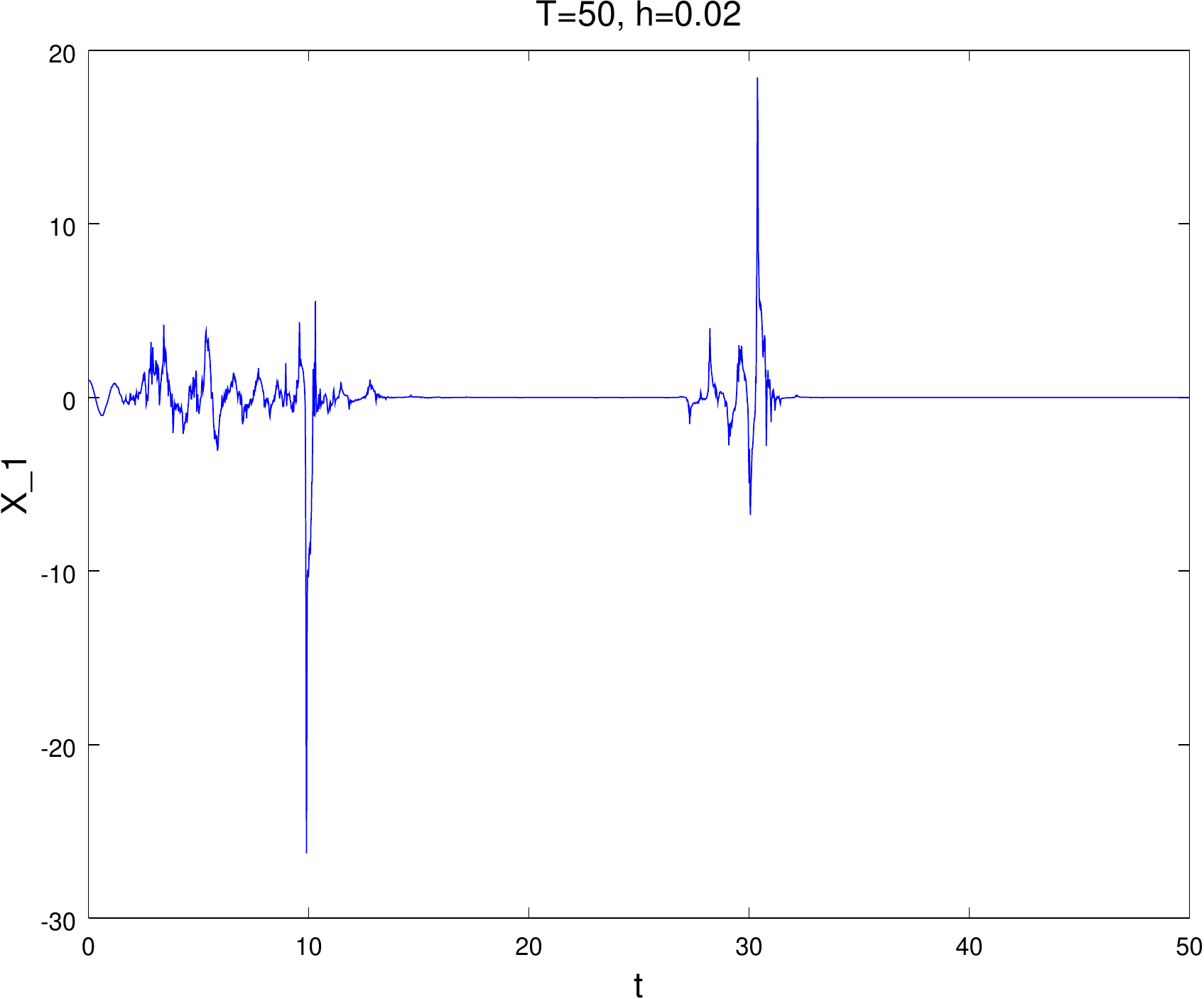}
                \caption{SSAMM method}
        \end{subfigure}
	\qquad
		\begin{subfigure}[t]{0.45\textwidth}
                \centering
                \includegraphics[width=\textwidth]{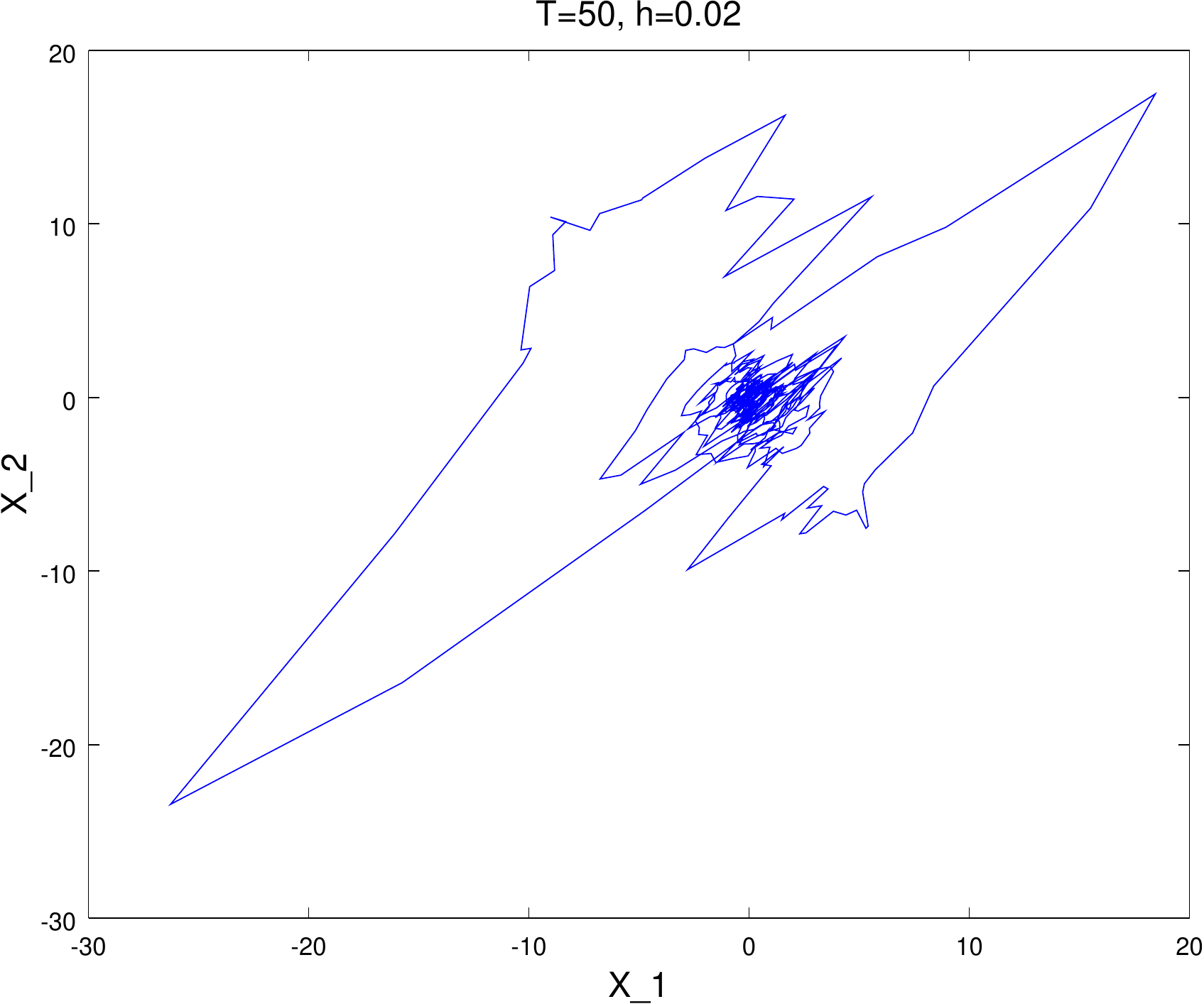}
                \caption{SSAMM method}
        \end{subfigure}
        \caption{Numerical simulation of SDE \eqref{eq:ex3a_1}}
        \label{fig:7}
\end{figure}

\vspace{5mm}
\textbf{Example 3. Application problem: chemical {L}angevin equations.}

When in the well-stirred system of $d$ biochemical species $S_1,...,S_d$ interact at a constant temperature through $m$ chemical reactions $R_1,...,R_m$ inside some fixed volume $\Omega$ in the regime which is far from thermodynamic limit, the state of the system can be modeled by chemical Langevin equations (CLE) \cite{Gillespie2000}.

We consider the following 3-species 6-reaction system \cite{Sotiropoulos2008}

\begin{align*}
	S_1 + S_2 \overset{c_1}{\rightarrow} S_3; \qquad S_1 + S_2 \overset{c_2}{\leftarrow} S_3;
	\\
	S_1 + S_3 \overset{c_3}{\rightarrow} S_2; \qquad S_1 + S_3 \overset{c_4}{\leftarrow} S_2;
	\\
	S_2 + S_3 \overset{c_5}{\rightarrow} S_1; \qquad S_2 + S_3 \overset{c_6}{\leftarrow} S_1
\end{align*}
with reaction rate constants $c_1 = 10^3$, $c_2 = 10^3$, $c_3 = 10^{-5}$, $c_4=10$, $c_5=1$ and $c_6=10^6$. The initial conditions are $X_1(0) = 10^3$, $X_2(0) = 10^3$ and $X_3(0) = 10^6$.

This reaction network can be described by the following system of CLEs
\begin{align}\label{eq:ex3_1}
	d X_i(t)
	=
	\sum_{j=1}^6 \nu_{ij} \alpha_j(X(t)) dt
	+
	\sum_{j=1}^6 \nu_{ij} \sqrt{\alpha_j(X(t))} dW_j
	\qquad i = 1,2,3.
\end{align}

In the above system $X(t)=\Big( X_1(t), X_2(t), X_3(t) \Big)^T$ is a state vector and $X_i(t)$ denotes the number of molecules of species $S_i$ in the system at time $t$. $\nu_{ij}$ is a stoichiometric matrix with state-change vectors as columns

\begin{align*}
	\nu = 
	\begin{pmatrix}
		-1 & 1 & -1 & 1 & 1 & -1 \\[0.5em]
		-1 & 1 & 1 & -1 & -1 & 1 \\[0.5em]
		1 & -1 & -1 & 1 & -1 & 1 \\
	\end{pmatrix}.
\end{align*}
Each state-change vector shows the change in the vector of molecular population induced by a single occurrence of a particular reaction.

The propensities $\alpha_j$ of the reaction channels are, respectively,

\begin{alignat*}{2}
	&\alpha_1(X(t)) = c_1 X_1(t) X_2(t), \quad & \alpha_2(X(t)) = c_2 X_3(t), & \\
	&\alpha_3(X(t)) = c_3 X_1(t) X_3(t), \quad & \alpha_4(X(t)) = c_4 X_2(t), & \\
	&\alpha_5(X(t)) = c_5 X_2(t) X_3(t), \quad & \alpha_6(X(t)) = c_6 X_1(t). & 
\end{alignat*}
The product of a propensity function with $dt$ gives the probability that a particular reaction will occur in the next infinitesimal time $dt$.
This system represents an inherently stiff nonlinear system of CLEs, since difference in the scales between propensities has several orders of magnitude. 

The system in \eqref{eq:ex3_1} was simulated over the time interval $[0,0.01]$. The tolerance of Newton iterations during the simulation was set to $10^{-6}$. 
Results of the simulation for different number of sample paths can be seen in Figure \ref{fig:8}. We notice that solution oscillates around the initial state of the system which replicates correct behavior of the considered chemical network with given initial state and reaction rate parameters.

The step size has to be chosen carefully since it can destabilize equilibrium of the numerical solution. For instance, for the considered methods it must be smaller than $10^{-4.5} -- 10^{-5}$. It has to be noted that the explicit Milstein method cannot generate stable solutions of the system \eqref{eq:ex3_1} for $h>1.2 \cdot 10^{-6}$.

Figure \ref{fig:9} depicts error versus CPU time. It is seen that the computational effort of algorithms grows linearly with respect to the desired tolerance of the solution. Also one can see that modified methods are more expensive than regular methods. Moreover, in the case of the Langevian system in \eqref{eq:ex3_1} we cannot benefit from the better stability properties of the modified methods because bigger step sizes can destabilize equilibrium of the system.

Additional attention must be paid to positivity preserving of the numerical scheme applied to the system \eqref{eq:ex3_1} since the number of species cannot be negative. In fact this requirement introduces an even more severe restriction on the time step than it does stiffness. Results reported in \cite{Sotiropoulos2008} show that in order to preserve positivity of the solution using the explicit Milstein scheme time step cannot exceed $5 \times 10^{-7}$. 
To guarantee positivity of the method we took absolute absolute values in the square roots, which is one of the standard tricks in simulating CLEs. A brief look at the averaged solution (Figure \ref{fig:8}) exhibits the acceptability of this approach.

\begin{figure}[p]
	\centering
        \begin{subfigure}[t]{0.45\textwidth}
                \centering
                \includegraphics[width=\textwidth]{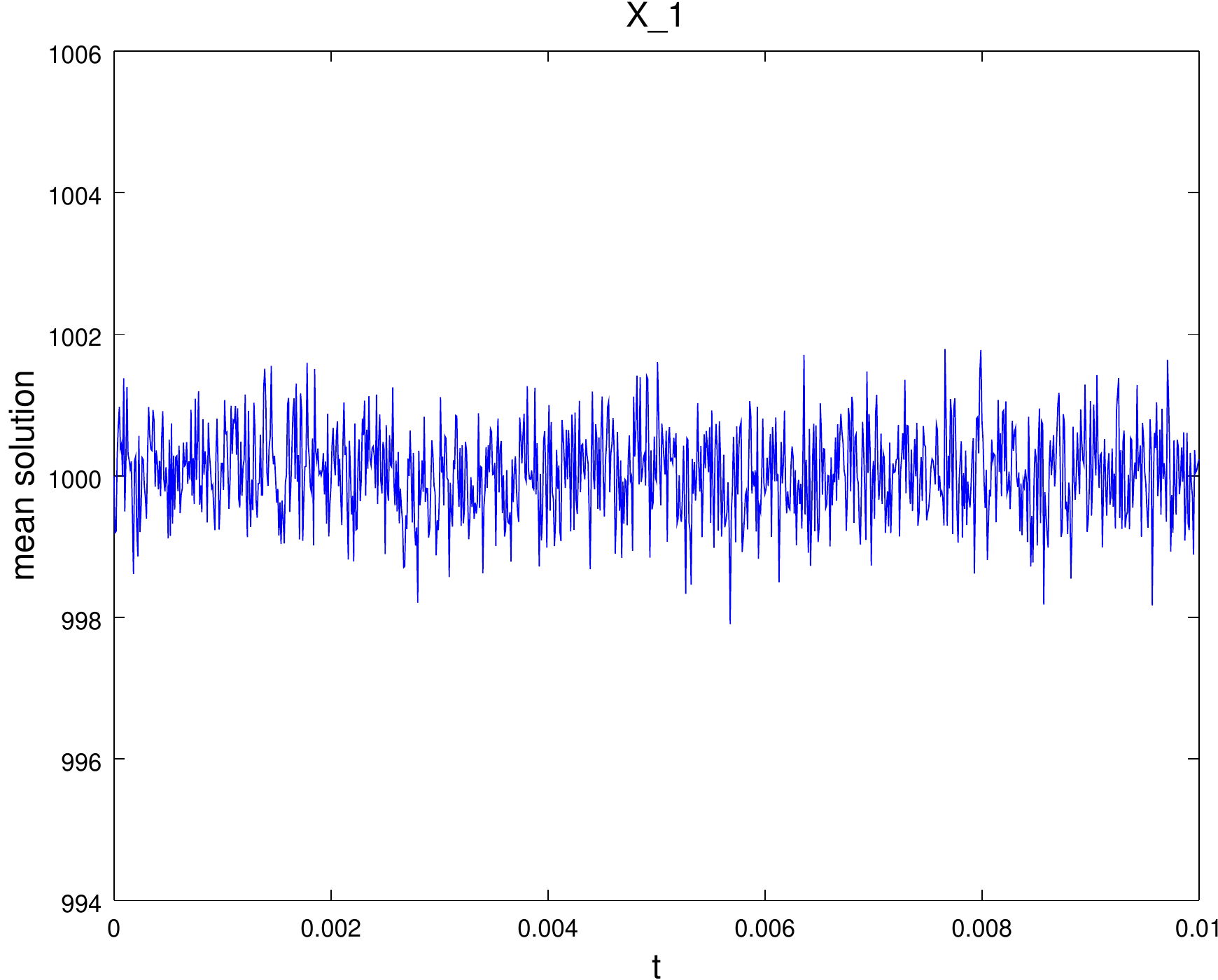}
                \caption{Mean solution for $X_1$}
        \end{subfigure}
	\qquad
		\begin{subfigure}[t]{0.45\textwidth}
                \centering
                \includegraphics[width=\textwidth]{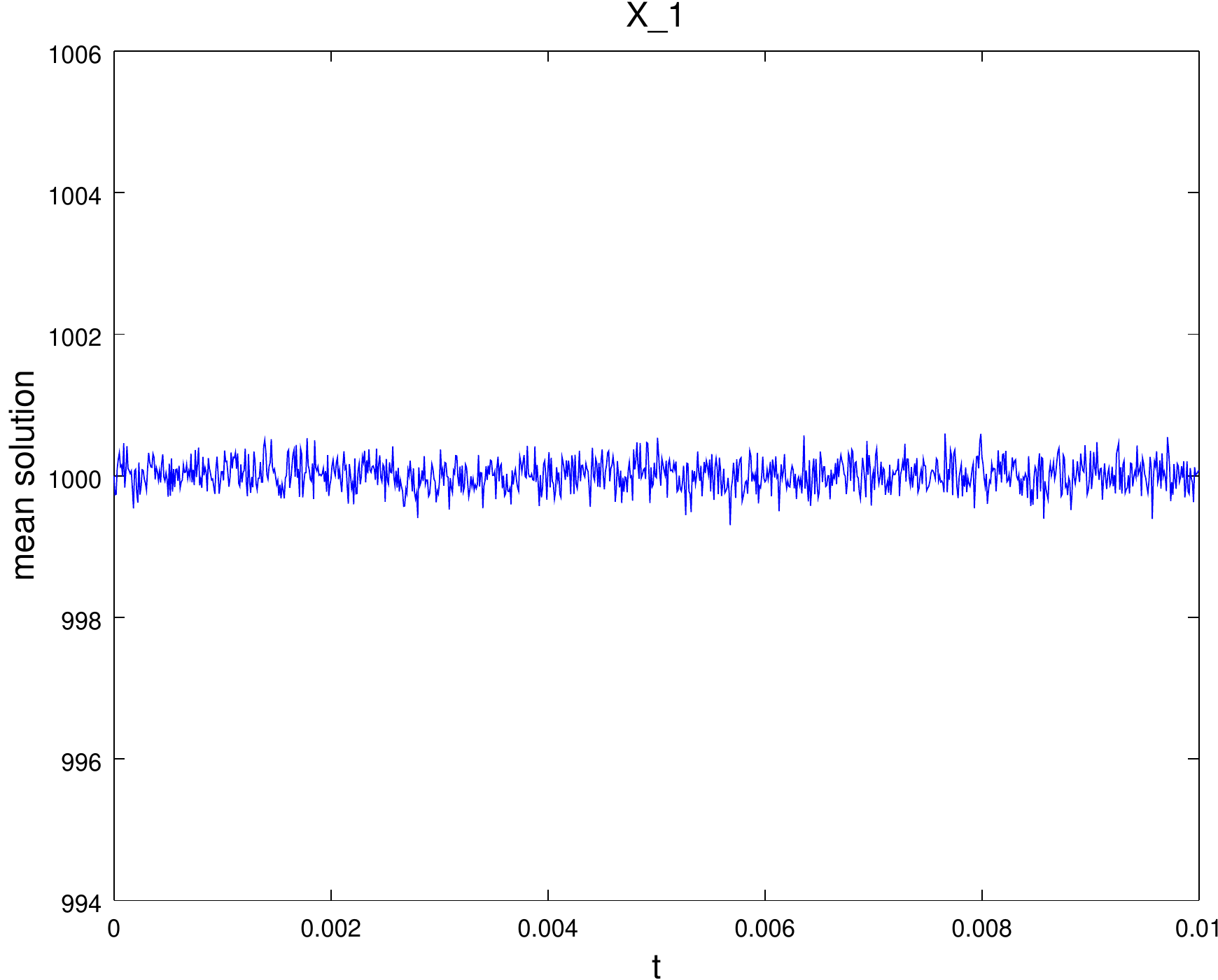}
                \caption{Mean solution for $X_2$}
        \end{subfigure}
	\qquad
		\begin{subfigure}[t]{0.45\textwidth}
                \centering
                \includegraphics[width=\textwidth]{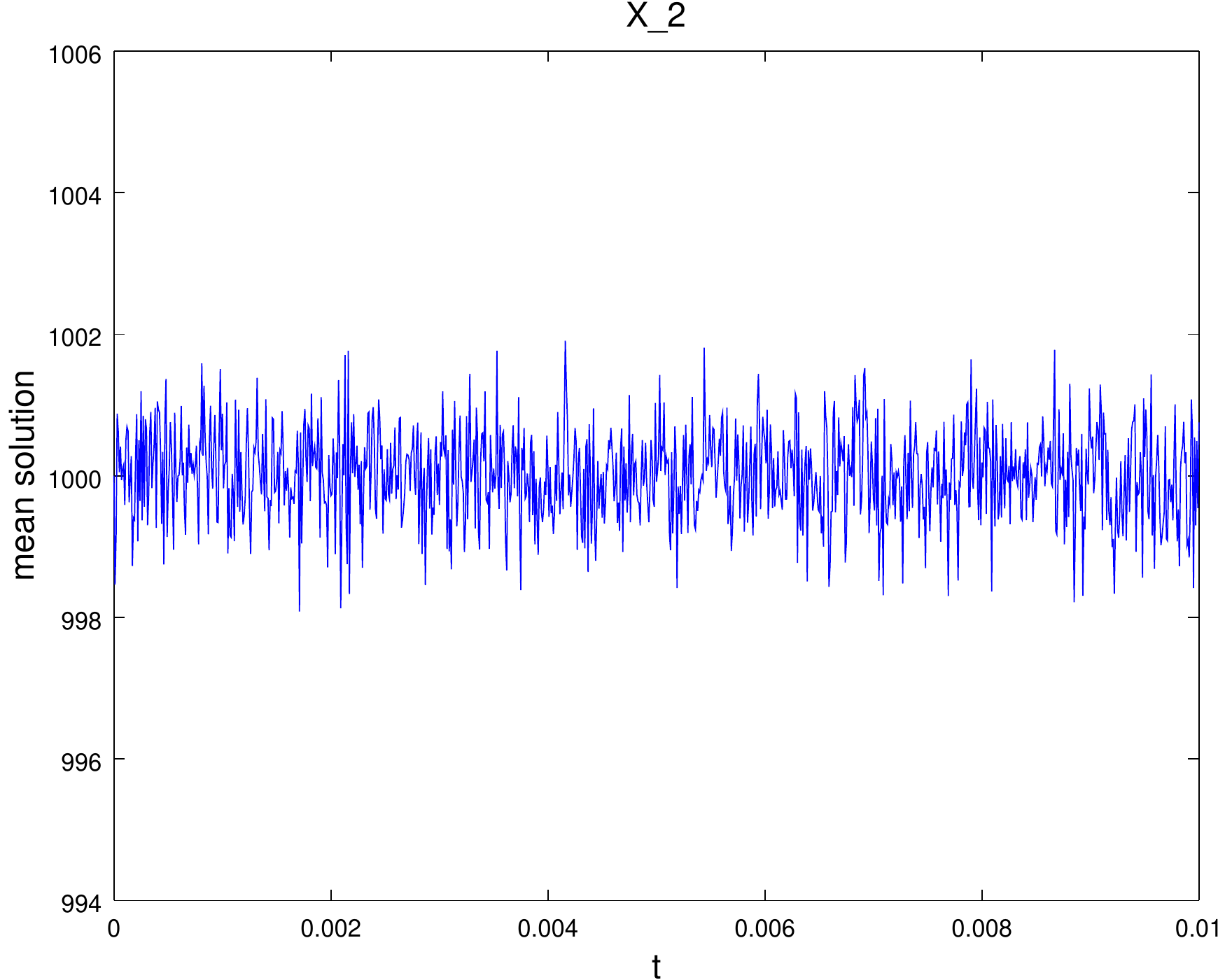}
                \caption{Mean solution for $X_2$}
        \end{subfigure}
	\qquad
		\begin{subfigure}[t]{0.45\textwidth}
                \centering
                \includegraphics[width=\textwidth]{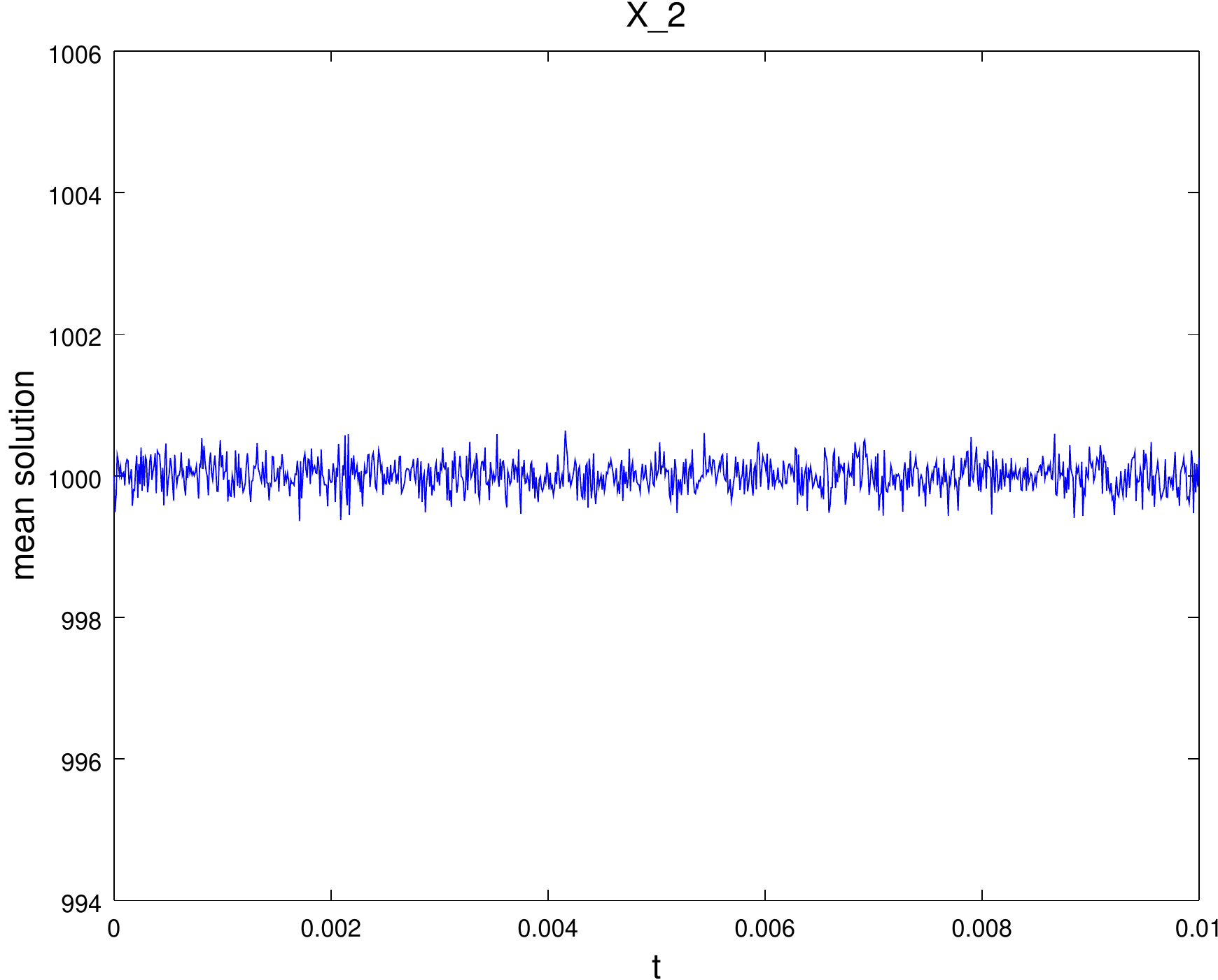}
                \caption{Mean solution for $X_2$}
        \end{subfigure}
	\qquad
		\begin{subfigure}[t]{0.45\textwidth}
                \centering
                \includegraphics[width=\textwidth]{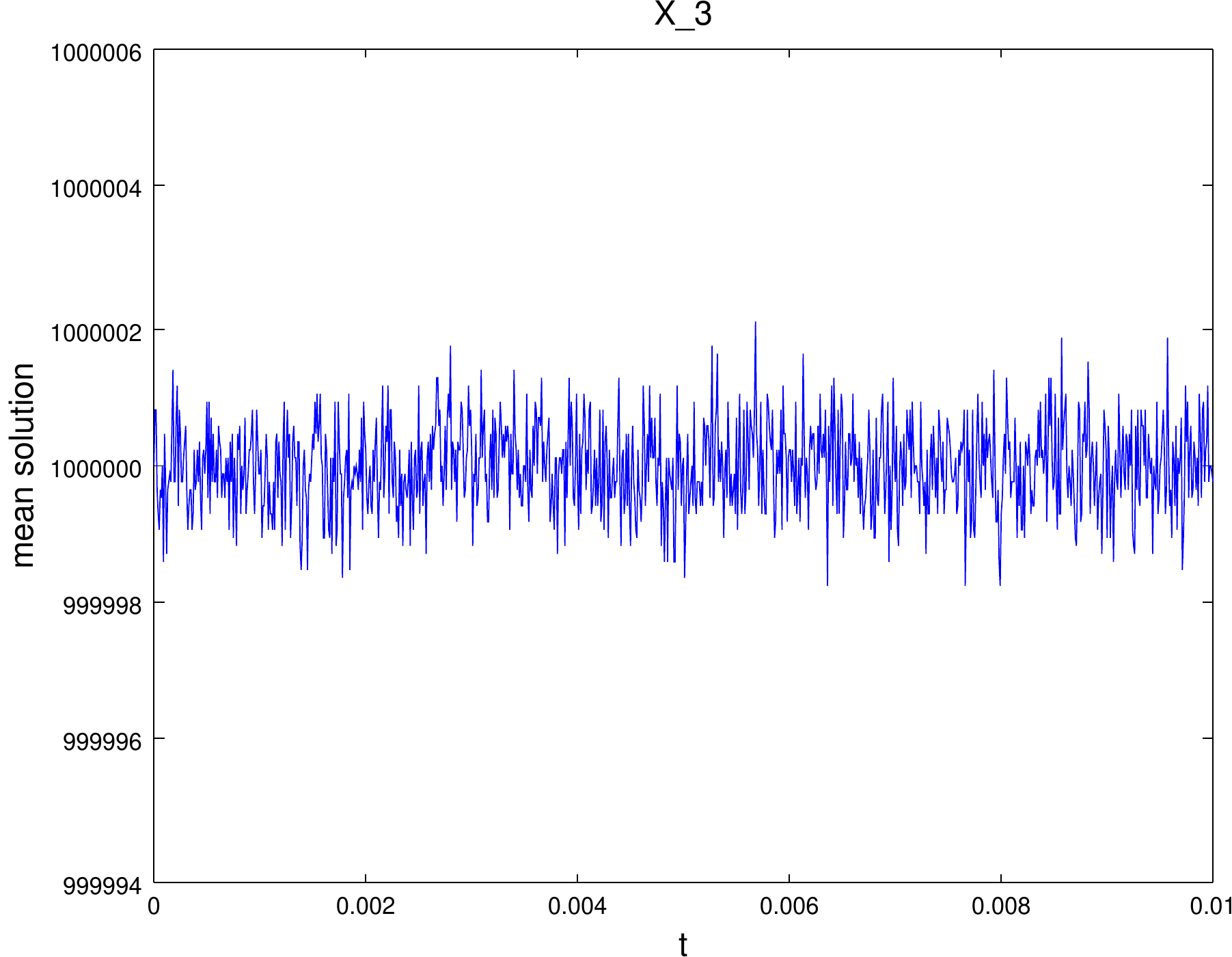}
                \caption{Mean solution for $X_3$}
        \end{subfigure}
	\qquad
		\begin{subfigure}[t]{0.45\textwidth}
                \centering
                \includegraphics[width=\textwidth]{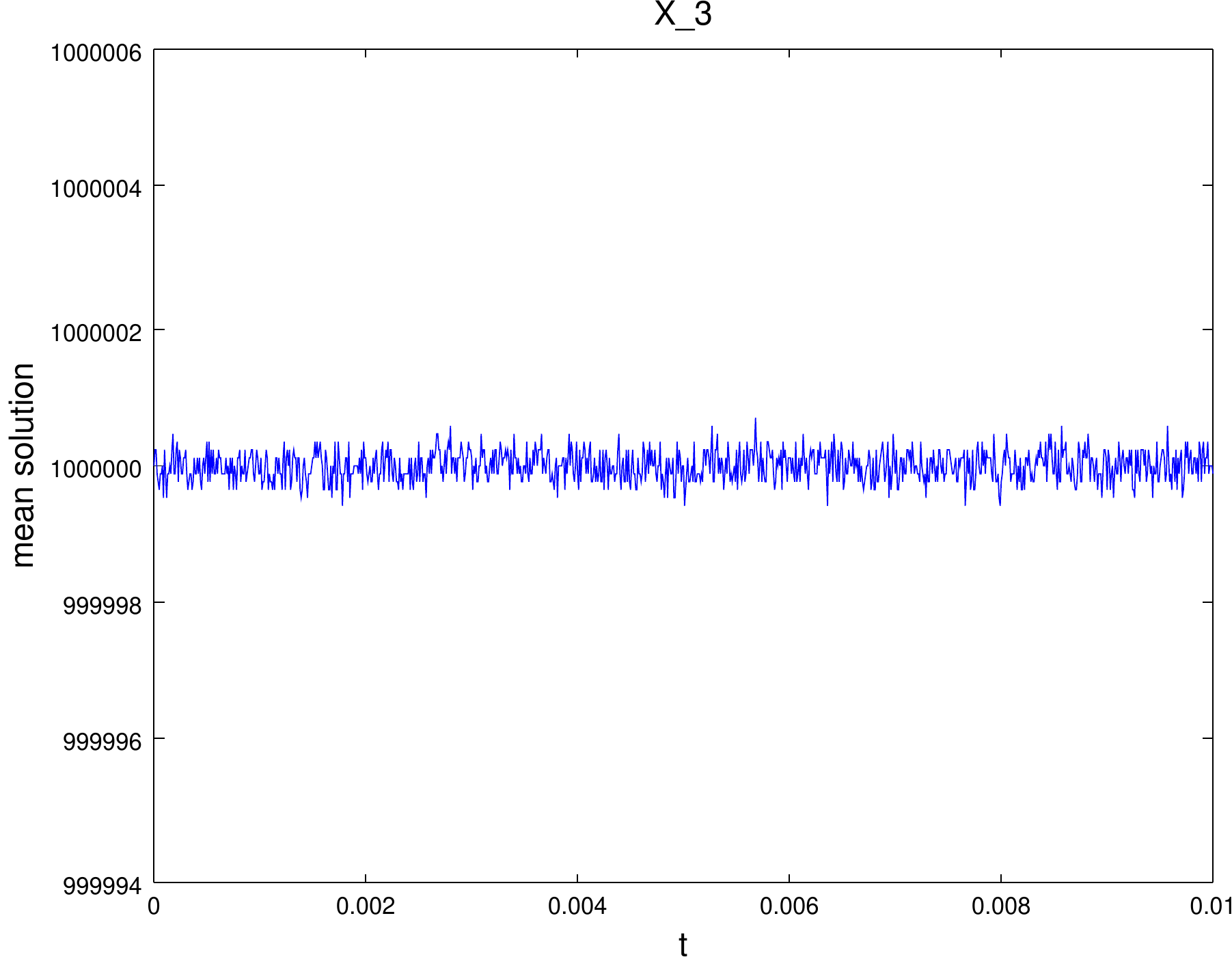}
                \caption{Mean solution for $X_3$}
        \end{subfigure}
        \caption{Mean solutions of the equation \eqref{eq:ex3_1} for $10^5$ (on the left) and $10^6$ (on the right) computed samples}
        \label{fig:8}
\end{figure}

\begin{figure}[h]
	\centering
    \includegraphics[width=0.5\textwidth]{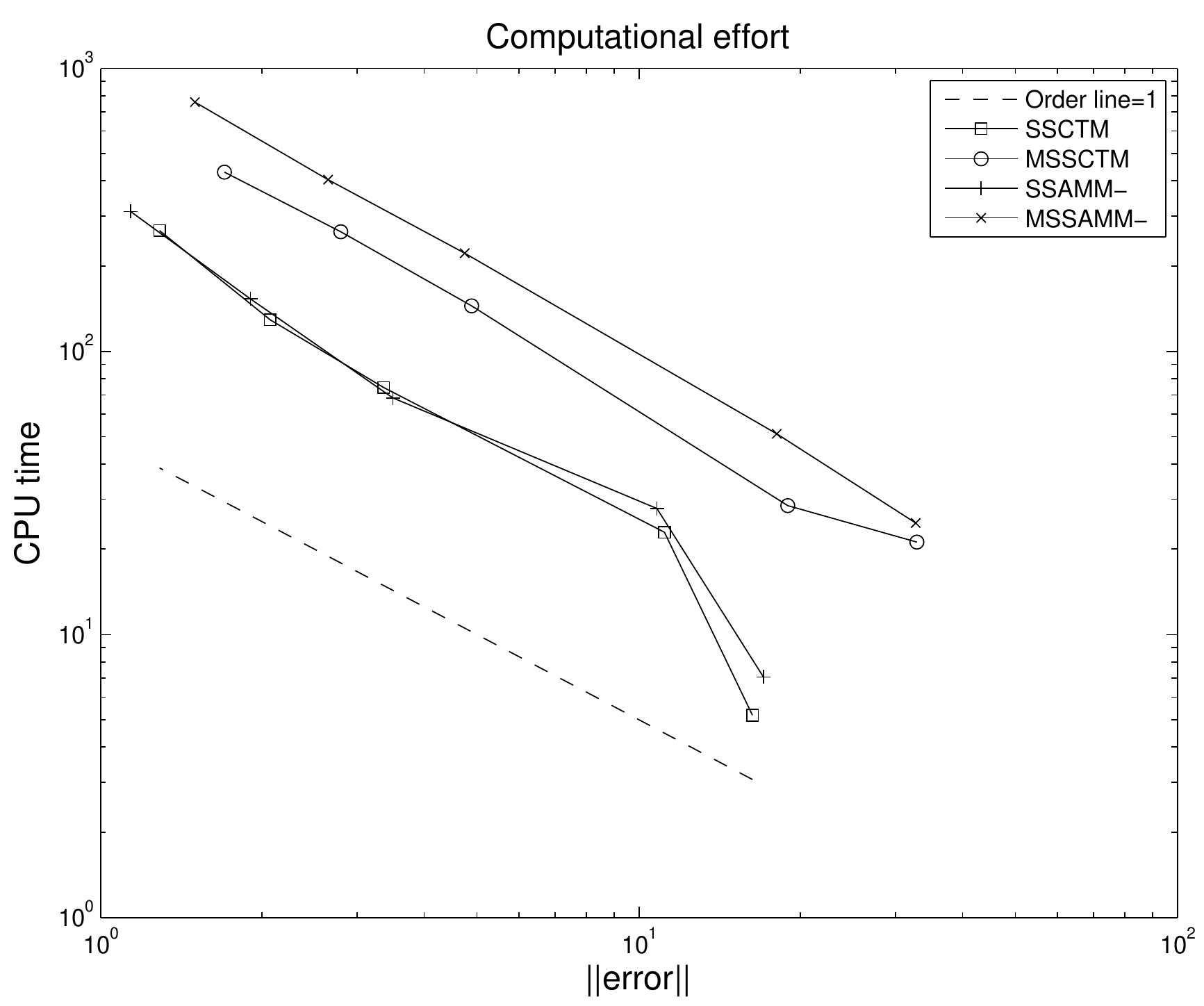}
    \caption{Computational cost of the equation \eqref{eq:ex3_1}} 
    \label{fig:9}
\end{figure}

\vspace{10mm}

\section{Conclusions.} 
 In this paper, we have investigated split-step Milstein methods for the solution of stiff stochastic differential equations driven by multi-channel noise. We focused on some of the features which are prominent in dealing with systems of SDEs and cannot be captured by scalar stability analysis. The dimension of the noise and geometry of interaction between drift and diffusion components are among the most important factors which influence the numerical stability of stochastic systems. The proposed numerical methods were tested on several problems discussed in the literature, each of these problems has a different nature of stochastic challenges. A linear problem with commutative noise was used as a benchmark for testing the convergence of the methods. From that problem, we also found that the dimension of the noise is a very strong constraint on the stability of the methods. We can make a conjecture that it may be even a stronger constraint  on the mean-square stability of the methods than noise geometry. 
 We have also tested the effectiveness of the methods on a stiff stochastic system of chemical Langevian equations. In addition, we found that for a large number of channels, the computational cost in implementing the modified spit-step methods is  higher than the gain in the stability.

\section{Acknowledgment.}
We appreciate the valuable comments of referees. This helped to improve the quality of the paper.

%
%
%

\appendix

\section{ }
\label{appendix-sec1}

In the general multi-dimensional case the Milstein scheme takes the form \cite{Kloeden1992}

\begin{align}\label{app:1}
	\begin{bmatrix}
		y^1_{n+1} \\
		\vdots \\
		y^d_{n+1}
	\end{bmatrix}
	=
	\begin{bmatrix}
		y^1_{n} \\
		\vdots \\
		y^d_{n}
	\end{bmatrix}
    + 
    h 
    \begin{bmatrix}
		f^1 \\
		\vdots \\
		f^d
	\end{bmatrix}
	+ 
	\sum_{j=1}^m 
	\begin{bmatrix}
		g^1_j \\
		\vdots \\
		g^d_j
	\end{bmatrix}	
	\Delta W_j 
	+ 
	\sum_{j_1=1}^m \sum_{j_2=1}^m \sum_{k=1}^d 
	\frac{\partial}{\partial y^k}
	\begin{bmatrix}
		g^1_{j_2} \\
		\vdots \\
		g^d_{j_2}
	\end{bmatrix}	
	g_{j_1}^k  I_{(j_1,j_2)}
\end{align}

Using vector-matrix notation the above expression can be rewritten in more compact and computationally convenient form as

\begin{align}\label{app:2}
    X_{n+1} &= X_n + h f + g \Delta W + \sum_{j=1}^d \frac{\partial g}{\partial y^k} (B^j)^T
\end{align}
where $y, f \in \mathbb{R}^d$ are solution and drift vectors respectively, $g \in \mathbb{R}^{d \times m}$ is a diffusion matrix, $\Delta W \in \mathbb{R}^m$ is an $m$-dimensional vector of Wiener increments and $B = g I \in \mathbb{R}^{d \times m}$ is a product of $d \times m$ diffusion matrix $g$ and $m \times m$ matrix of double stochastic integrals $I$. $B^j$ denotes the $j$-th row of matrix $B$.

In expanded form this reads as 

\begin{align}\label{app:3}
	\begin{bmatrix}
		y^1_{n+1} \\
		\vdots \\
		y^d_{n+1}
	\end{bmatrix}
	=
	\begin{bmatrix}
		y^1_{n} \\
		\vdots \\
		y^d_{n}
	\end{bmatrix}
    + 
    h 
    \begin{bmatrix}
		f^1 \\
		\vdots \\
		f^d
	\end{bmatrix}
	+ 
	\begin{bmatrix}
		g^1_1 & \hdots & g^1_m \\
		\vdots & \ddots & \vdots \\
		g^d_1 & \hdots & g^d_m \\
	\end{bmatrix}	
    \begin{bmatrix}
		\Delta W_1  \\
		\vdots \\
		\Delta W_m
	\end{bmatrix}
	+ 
	\sum_{j=1}^d 
	\begin{bmatrix}
		\frac{\partial g^1_1}{\partial y^j} & \hdots & \frac{\partial g^1_m}{\partial y^j} \\
		\vdots & \ddots & \vdots \\
		\frac{\partial g^d_1}{\partial y^j} & \hdots & \frac{\partial g^d_m}{\partial y^j} \\
	\end{bmatrix}	
	\begin{bmatrix}
		b^j_{1} \\
		\vdots \\
		b^j_{m}
	\end{bmatrix}	
\end{align}

To derive representation in \eqref{app:3} let take a look at the last term of equation \eqref{app:1}

\begin{align}
	& \sum_{j_1=1}^m \sum_{j_2=1}^m \sum_{k=1}^d g_{j_1}^k \frac{\partial g_{j_2}^i(t_n,X_n)}{\partial y^k} I_{(j_1,j_2)}
	= 
	\sum_{j_2=1}^m \sum_{k=1}^d \frac{\partial g_{j_2}^i}{\partial y^k} \underbrace{\sum_{j_1=1}^m g_{j_1}^k I_{(j_1,j_2)}}_{B=B_{j_2}^k=g \cdot I}
\end{align}

If $J_{g^i} = (J_{g^i})_k^{j_2}$ is used to denote the Jacobian matrix of $i$-th row of matrix $g$ then the above identity converts to 

\begin{align}\label{app:4}
	\sum_{j_2=1}^m \sum_{k=1}^d \frac{\partial g_{j_2}^i}{\partial y^k} \underbrace{\sum_{j_1=1}^m g_{j_1}^k I_{(j_1,j_2)}}_{B=B_{j_2}^k=g \cdot I}
	=
	\sum_{j_2=1}^m \sum_{k=1}^d (J_{g^i})_k^{j_2} B_{j_2}^k = Tr(J_{g^i} \cdot B)
	=
	\left[ vec(J_{g^i}^T) \right]^T \cdot vec(B)
\end{align}

Then we obtain

\begin{align}
	\sum_{j_1=1}^m \sum_{j_2=1}^m \sum_{k=1}^d 
	\frac{\partial}{\partial y^k}
	\begin{bmatrix}
		g^1_{j_2} \\
		\vdots \\
		g^d_{j_2}
	\end{bmatrix}	
	g_{j_1}^k  I_{(j_1,j_2)}
	=
	\begin{bmatrix}
		\left[ vec(J_{g^1}^T) \right]^T \\
		\vdots \\
		\left[ vec(J_{g^n}^T) \right]^T \\
	\end{bmatrix}	
	\cdot vec(B)
	=
	\sum_{j=1}^m J_{g_j} B_j
	=
	\sum_{j=1}^d \frac{\partial g}{\partial y^k} (B^j)^T
\end{align}
where $J_{g_j}$ denotes the Jacobian matrix of $j$-th column of matrix $g$ and $B_j$ is the $j$-th column of matrix $B$.

For efficient calculation of iterated stochastic integrals $I_{(j_1,j_2)}$ we refer to \cite{wiktorsson2001,Gilsing2007}.




\bibliographystyle{model1b-num-names}

\bibliography{DV-AK-VR-Split-step-2014}

\end{document}